\documentclass[a4paper,10pt]{amsart}
\usepackage[utf8]{inputenc}
\usepackage{amsthm}
\usepackage{amsmath}
\usepackage{bm}
\usepackage{amsfonts}
\usepackage{amssymb}
\usepackage{mathtools}
\usepackage{mathrsfs}
\usepackage{setspace}
\usepackage{xcolor}
\usepackage{textgreek}
\usepackage{todonotes}
\usepackage[a4paper, left=3.5cm, right=3.5cm, top=4cm, bottom=4cm]{geometry}
\usepackage{thmtools} 
\usepackage{soul}

\usepackage[pdfdisplaydoctitle,colorlinks,breaklinks,urlcolor=blue,linkcolor=blue,citecolor=blue]{hyperref} 

\newtheorem*{thm*}{Theorem}
\newtheorem{thm}{Theorem}[section]

\newtheorem{defn}[thm]{Definition}
\newtheorem{cor}[thm]{Corollary}

\newtheorem{lem}[thm]{Lemma}

\newtheorem{prop}[thm]{Proposition}

\newcommand{\N}{\mathbb{N}}
\newcommand{\R}{\mathbb{R}}
\newcommand{\PP}{\mathbb{P}}

\theoremstyle{remark}
\newtheorem{rmk}[thm]{Remark}

\numberwithin{equation}{section}

\title[Perturbation theory two-scale systems in fluid dynamics]{Second order perturbation theory of two-scale systems in fluid dynamics}
\author[A. Debussche]{Arnaud Debussche}
  \address{Univ Rennes, CNRS, IRMAR - UMR 6625, F-35000 Rennes, France}
  \email{\href{arnaud.debussche(at)ens-rennes.fr}{arnaud.debussche(at)ens-rennes.fr}}
\author[U. Pappalettera]{Umberto Pappalettera}
  \address{Scuola Normale Superiore, Piazza dei Cavalieri, 7, 56126 Pisa, Italia}
  \email{\href{umberto.pappalettera(at)sns.it}{umberto.pappalettera(at)sns.it}}
\keywords{Fluid dynamics, Transport noise, It\"o-Stokes drift.}
\date\today

\begin{document}

\begin{abstract}
In the present paper we study slow-fast systems of coupled equations from fluid dynamics, where the fast component is perturbed by additive noise.
We prove that, under a suitable limit of infinite separation of scales, the slow component of the system converges in law to a solution of the initial equation perturbed with transport noise, and subject to the influence of an additional It\=o-Stokes drift.
The obtained limit equation is very similar to turbulent models derived heuristically. 
Our results apply to the Navier-Stokes equations in dimension $d=2,3$; the Surface Quasi-Geostrophic equations in dimension $d=2$; and the Primitive equations in dimension $d=2,3$.
\end{abstract}

\maketitle

\section{Introduction}
Let $T>0$ be fixed. 
In this work we consider fast-slow systems of coupled abstract equations of fluid dynamics
\begin{align} \label{eq:system}
\begin{cases}
du^\epsilon_t = Au^\epsilon_t dt + b(u^\epsilon_t,u^\epsilon_t) dt + b(v^\epsilon_t,u^\epsilon_t) dt,
\\
dv^\epsilon_t = \epsilon^{-1} Cv^\epsilon_t dt + Av^\epsilon dt + b(u^\epsilon_t,v^\epsilon_t) dt + b(v^\epsilon_t,v^\epsilon_t) dt + \epsilon^{-1} Q^{1/2}dW_t,
\end{cases}
\end{align}
where $t \in [0,T]$, $\epsilon\in(0,1)$ is a small parameter indicating separation of scales, and $Q^{1/2} dW$ is an additive Gaussian noise, white in time and coloured in space.
$A$ and $C$ are (possibly unbounded) negative definite linear operators on a separable Hilbert space $H$, and should be interpreted as dissipation terms. The map $b:H \times H \to H$ is bilinear and enjoys suitable properties detailed below.
In the present paper, as examples of physical equations for $u^\epsilon$, $v^\epsilon$ that can be coupled into \eqref{eq:system} and described within this formalism, we consider: the Navier-Stokes equations in dimension $d=2,3$; the Surface Quasi-Geostrophic equations in dimension $d=2$; the Primitive equations in dimension $d=2,3$. 
Our goal is to describe the asymptotic behaviour of the slow component $u^\epsilon$ as the scaling parameter $\epsilon\to 0$. 

System \eqref{eq:system} above aims to describe heuristically the dynamics of a coupled fast-slow system, where $u^\epsilon$ (resp. $v^\epsilon$) corresponds to the slow-varying, large-scale (resp. fast-varying, small-scale) component of the fluid.
It is not clear that such a separation of scales holds so strictly in fluids, however \eqref{eq:system} can be a physically sensible, yet mathematically challenging starting point in the investigation on turbulence in fluids. 
Indeed, \emph{"turbulent flows contain self-sustaining velocity fluctuations in addition to the main flow"} \cite[Chapter 26]{Pa13}, and additive noise in fluid dynamics equations has been widely used for decades to capture the statistical properties of turbulent fluids \cite{BeFe20,DPDe03,FlRo08,HaMa06}.
 
Under suitable conditions, we are able to prove the following result (for a precise statement thereof see \autoref{thm:main}):
\begin{thm*} 
Let $\{u^\epsilon\}_{\epsilon\in(0,1)}$ be a family of solutions to \eqref{eq:system} in the sense of \autoref{def:sol} below. Then $\{u^\epsilon\}_{\epsilon\in(0,1)}$ admits converging in law subsequences, and every weak accumulation point $u$ solves the equation with transport noise and It\=o-Stokes drift velocity $r$:
\begin{align} \label{eq:u_trans}
du_t = Au_t dt + b(u_t,u_t) dt + b((-C)^{-1}Q^{1/2}\circ dW_t,u_t) +  b(r,u_t) dt. 
\end{align}
In addition, if pathwise uniqueness holds for \eqref{eq:system} and \eqref{eq:u_trans} then the whole sequence $u^\epsilon$ converges to $u$ in probability. 
\end{thm*} 

The It\=o-Stokes drift velocity $r$, usually defined as the difference between Lagrangian and Eulerian average flows, has important consequences in wave-induced sediment transport and sandbar migration in the coastal zone, as well as transport of heat, salt and other natural or man-made tracers in the upper ocean layer \cite{BrBr18}.
Here we find a precise expression for $r$, related to the invariant measure $\mu$ of the linearised equation $dv_t=Cv_t dt + Q^{1/2}dW_t$  via the formula
\begin{align*}
r = \int (-C)^{-1} b(w,w) d\mu(w).
\end{align*}
Equation \eqref{eq:u_trans} contains a transport noise. 
There have been heuristic arguments developed to derive transport noise in Navier-Stokes and Euler equations by perturbing deterministic balance laws \cite{MiRo04}, variational principles \cite{Ho15} or homogenization techniques \cite{CoGoHo17}. Ideas of  \cite{MiRo04} have been extended and systematically used to derive stochastic fluid  models in a series of work initiated in \cite{Memin14}. Compared to the previous ones, the models in \cite{BaChChLiMe20} also contain the It\=o-Stokes drift, cf. (11b) therein. It is remarkable that we recover exactly the same expression for this It\=o-Stokes drift. 

More recently, in the particular case $A$ equals to the Laplace operator and $C=-Id$, \cite{FlPa21,FlPa22+} considered a similar problem as ours and were able to obtain similar results for fluids in dimension two by proving a Wong-Zakai-type convergence for the (stochastic) characteristics associated with \eqref{eq:system}. In these articles, the
It\=o-Stokes drift was not apparent in the limit model because of their assumptions on the spatial structure of the noise. Also the arguments do not seem to be applicable to more complicated models.

{
The present paper adds to this picture and goes much further. The transport noise  is justified rigorously for many fluid models through a very robust method, the adequat generalization of the classical perturbed test function method to partial differential equations. Also, up to our knowledge, it is the first rigorous derivation of a stochastic fluid model containing the It\=o-Stokes drift. 
In fact, the transport noise in the limit equation comes from a diffusion-approximation argument (instead of Wong-Zakai type results as in \cite{FlPa21,FlPa22+}), while the It\=o-Stokes drift is due to an average constraint similar to those appearing in homogenization theory. Both phenomena are handled by the perturbed test function method. 

In addition, as already mentioned our results are very general and we are able to study many different systems, some of them notably difficult to study in the Lagrangian formulation, highlighting the fundamental nature of transport noise in fluids. More precisely, we show that our method also applies to the Surface Quasi-Geostrophic equations and to the Primitive equations. 
In each case, we derive stochastic fluid models containing a transport noise and a It\=o-Stokes drift, they are very similar to the models obtained by E. M\'emin and his co-authors (\cite{BaChChLiMe20}), justifying their use through rigorous mathematical arguments. We restrict to these examples but, as argued in \cite{AHHS22}, \cite{AHHS22+}, \cite{FLL23}, \cite{LT22}, 
our arguments  can be used to study other climate models. Note also that the authors of \cite{Ag22}, \cite{AgVe22} mention our 
work as a way to justify their models for an application in a totally different context. More generally, we believe that many two-scales partial differential equations models can and will be treated thanks to suitable adaptations of our arguments. 

The perturbed test function method has been introduced in \cite{PaStVa77} and is now a standard method to study finite dimensional approximation-diffusion problems. It has been generalized to infinite dimensional problems recently (see for instance \cite{dBGa12,DedMVo16,DeVo12, DeVo21}). However the partial differential equations considered there were simpler and the present work is the first application of this method to nonlinear fluid models such as the ones considered here.
 
Also, the fast variable is $v^\epsilon$, solution of a nonlinear partial differential equation. This creates difficulties since, in the context of partial differential equations, the correctors have a more complex expressions than in the classical finite dimensional framework. They involve the fast variable and its spatial derivatives. It is well known that $v^\epsilon$ does not have much spatial smoothness for realistic three dimensional models such as the ones we want to handle. We see below that in fact we can use a linearization trick and replace  $v^\epsilon$ by an infinite dimensional Ornstein-Uhlenbeck process, solution of a linear problem. Also, a central tool of the method is the generator of the fast variable. Here, this generator is an infinite dimensional differential operator. We analyse this operator and prove that its domains contains the functions necessary to perform the method: prove tightness of $u^\epsilon$ and take the limit in the infinite dimensional problem thanks to the correctors. Among the various difficulties this implies, let us mention the uniform bounds on the correctors and the identification of the terms appearing at the limit, namely the transport noise, the associate It\=o-Stratonovitch corrector and the It\=o-Stokes drift. }

Besides that, in \autoref{sec:eddy} we restrict our attention to the Navier-Stokes system and discuss the interesting problem of the limit behaviour of $u^\epsilon$ when the covariance operator $Q=Q_\epsilon$ is depending on the scaling parameter $\epsilon\in(0,1)$.
Mostly inspired by \cite{FlLu21} (cf. also the series of works \cite{FlGaLu21+,FlGaLu22,FlLu20}), where the authors prove the convergence, in a certain scaling limit, of the solution of Navier-Stokes equations perturbed by transport noise to their deterministic counterpart with extra dissipation, we identify conditions under which  $u^\epsilon$ converges in law towards a process $u$, that solves the deterministic Navier-Stokes equations with an additional dissipative term $\kappa$ (\autoref{thm:deterministic_limit}):
\begin{align*}
du_t
&= 
A u_t dt + b(u_t,u_t) dt+ \kappa(u_t)dt.
\end{align*}
This is the first result showing that eddy viscosity in 3D Navier-Stokes equation can be created by additive noise. 
Notice that in the equation above we have neither the stochastic integral (because of the scaling) nor the It\=o-Stokes drift (because of the isotropy of the noise chosen in \cite{FlLu21}, see also part \emph{ii}) of \autoref{rmk:commuting+stokes}).

Fluid dynamics equations perturbed by noise of transport type have the subject of a lot of research recently:
\begin{align*}
du_t = Au_t dt + b(u_t,u_t) dt + b((-C)^{-1}Q^{1/2}\circ dW_t,u_t).
\end{align*} 
They stand out for often enjoying well-posedness \cite{BrFlMa16,BrCaFl92,BrSl20+,FlGaLu21c,FlGuPr10,FlLu21,FlMaNe14,FlOl18,MiRo05}, and manifest properties as enhanced dissipation \cite{FlGaLu21+,GeYa21+,Lu21+} and mixing \cite{FlGaLu21+,FlGaLu22}, typical of turbulent fluids.
Still, unlike the case of additive noise that is widely accepted as source of randomness, transport noise needs a more careful justification.
For instance, in \cite{MaKr99} the authors write: \emph{"In principle, the turbulent velocity which advects the passive scalar should be a solution to the Navier-Stokes equations [...] with some external stirring which maintains the fluid in a turbulent state. But the analytical representation of such solutions corresponding to complex, especially turbulent flows, are typically unwieldy or unknown.
We shall therefore instead utilize simplified velocity field models which exhibit some empirical features of turbulent or other flows, though these models may not be actual solutions to the
Navier-Stokes equations."}
More recently, passive scalars advected by solutions to the stochastic Navier-Stokes equations have been studied in the series of papers \cite{BeBlPu20,BeBlPu22b,BeBlPu22c,BeBlPu22}.

In view of this, and being aware of all the differences and limitations of our model, we look at \eqref{eq:system} as a more accurate way to model the evolution of a fluid $u^\epsilon$ subject to the influence of a turbulent component $v^\epsilon$ and believe that our results provides a good justification of transport noise. 
The choice of the parameter $\epsilon^{-1}$ in front of both noise and dissipation is appropriate when looking at some particular systems from to the point of view of a large-scale observer, see for instance \cite{FlPa21,FlPa22+}, and we adopt the same scaling in the present paper by analogy.
Also, the reader can notice a strong similarity with the theory of stochastic model reduction developed in \cite{MaTiVa01}, where the authors study finite-dimensional stochastic systems with quadratic nonlinearities and two widely separated time scales as an approximation of geophysical PDEs, with noise replacing the self-interaction of unresolved variables; however, their results were obtained by formal second-order perturbation theory following \cite{Ku73} and lack explicit error estimates, thus it is not clear whether they can be extended to infinite-dimensional frameworks like ours.

\subsection{Perturbed test function method} \label{ssec:perturbed_test_function}

Let us describe the main technique allowing us to determine the limiting behaviour of the slow process $u^\epsilon$ as $\epsilon \to 0$.
The method here presented is originally due to Papanicolaou, Stroock and Varadhan \cite{PaStVa77}. The reader may find new developments and a presentation in the book \cite{FGPS}. It has been recently extended to infinite dimension and partial differential equations, see for instance \cite{dBGa12,DedMVo16,DeVo12, DeVo21}.

Let us consider again system \eqref{eq:system}, and denote $y^\epsilon_t = \epsilon^{1/2} v^\epsilon_t$:
\begin{align} \label{eq:system_2}
\begin{cases}
du^\epsilon_t = Au^\epsilon_t dt + b(u^\epsilon_t,u^\epsilon_t) dt + \epsilon^{-1/2} b(y^\epsilon_t,u^\epsilon_t) dt,
\\
dy^\epsilon_t = \epsilon^{-1} Cy^\epsilon_t dt + A y^\epsilon_t dt + b(u^\epsilon_t,y^\epsilon_t) dt + \epsilon^{-1/2}b(y^\epsilon_t,y^\epsilon_t) dt + \epsilon^{-1/2} Q^{1/2} dW_t.
\end{cases}
\end{align}
{ In order to better present the ideas, let us suppose for the moment that $(u^\epsilon,y^\epsilon)$ is a Markov process whose evolution is described by its infinitesimal generator $\mathscr{L}^\epsilon$}, which takes the following form when applied to a suitable test function $\varphi$:
\begin{align*}
\mathscr{L}^\epsilon \varphi (u,y)
&=
\langle Au + b(u,u) , D_u \varphi \rangle
+
\epsilon^{-1/2} \langle b(y,u) , D_u \varphi \rangle
\\
&\quad
+
\langle Ay + b(u,y) , D_y \varphi \rangle
+
\epsilon^{-1/2} \langle b(y,y) , D_y \varphi \rangle
\\
&\quad+
\epsilon^{-1} \langle Cy , D_y \varphi \rangle
+
\frac{\epsilon^{-1}}{2} Tr (Q D^2_y \varphi).
\end{align*}
{ We shall see later that the hypothesis of $(u^\epsilon,y^\epsilon)$ being Markov is not strictly necessary, as one only needs the validity of a suitable It\=o Formula for functions of $(u^\epsilon,y^\epsilon)$, cf. \autoref{lem:ito}.}
Since we are interested in the limiting behaviour of $u^\epsilon$ as the parameter $\epsilon$ goes to zero, we add correctors to $\varphi$ in order to cancel out singular terms in the expression of $\mathscr{L}^\epsilon \varphi$, on the one hand, and simultaneously eliminate the dependence on $y$ in the terms of order one, on the other.
The motivation behind this procedure lies in the fact that we would like to obtain in the limit a closed equation in the sole variable $u^\epsilon$.

\subsubsection*{The classical method}
At this point, one could carry on the previous program in the following way. Consider the \emph{perturbed test function}
\begin{align*}
\varphi^\epsilon(u,y) = \varphi(u) + \epsilon^{1/2}\varphi_1(u,y) + \epsilon \varphi_2 (u,y),
\end{align*}
where $\varphi_1$ and $\varphi_2$ are suitable correctors.

Denoting $\mathscr{L}_y$ the operator 
\begin{align*}
\mathscr{L}_y 
= 
\langle Cy , D_y \cdot \rangle
+
\frac12 Tr (Q D^2_y \cdot),
\end{align*} 
it is immediately clear that terms of order $\epsilon^{-1}$ in the expression of $\mathscr{L}^\epsilon \varphi^\epsilon$ vanish: indeed, they are given by a factor $\epsilon^{-1}$ times $\mathscr{L}_y \varphi$, which equals zero since $\varphi$ does not depend on $y$.
Moving to terms of order $\epsilon^{-1/2}$, they are given by a factor $\epsilon^{-1/2}$ times the quantity 
\begin{align} \label{eq:order_eps-1/2}
\langle b(y,u) , D_u \varphi \rangle
+
\langle b(y,y) , D_y \varphi \rangle
+
\mathscr{L}_y \varphi_1.
\end{align}
Therefore, \eqref{eq:order_eps-1/2} cancels out if $\varphi_1$ is a solution to the Poisson equation:
\begin{align}\label{1.4bis} 
\mathscr{L}_y \varphi_1(u,y) 
= 
-\langle b(y,u) , D_u \varphi \rangle.
\end{align}

Finally, the terms of order $1$ in the expression of $\mathscr{L}^\epsilon \varphi^\epsilon$ equal:
\begin{align*}
\langle Au + b(u,u) , D_u \varphi \rangle
+
\langle b(y,u) , D_u \varphi_1 \rangle
+
\langle b(y,y) , D_y \varphi_1 \rangle
+
\mathscr{L}_y \varphi_2.
\end{align*}
As previously remarked, it is not necessary to require the previous quantity to be zero, but it is sufficient to just seek for $\varphi_2$ such that it does not depend on $y$, namely
\begin{align*}
\mathscr{L}^0 \varphi (u)
&=
\langle Au + b(u,u) , D_u \varphi \rangle
+
\langle b(y,u) , D_u \varphi_1 \rangle
\\
&\quad+ \nonumber
\langle b(y,y) , D_y \varphi_1 \rangle
+
\mathscr{L}_y \varphi_2(u,y)
\end{align*}
for some \emph{effective} generator $\mathscr{L}^0$.
In this way, one would formally get $\mathscr{L}^\epsilon \varphi^\epsilon (u^\epsilon,y^\epsilon)=\mathscr{L}^0 \varphi (u^\epsilon)$ and identify the limit behaviour of $\mathscr{L}^\epsilon \varphi(u^\epsilon)$ up to a correction that has to be shown to be infinitesimal as $\epsilon \to 0$.
Moreover, it is needed to properly justify the previous procedure for a sufficiently large class of test function $\varphi$.

\subsubsection*{Linearisation Trick}
The method described above is very powerful, and strictly speaking it does allow to rigorously understand the limiting behaviour of $u^\epsilon$, $\epsilon \to 0$ for the Navier-Stokes system in dimension $d=2,3$ for some sufficiently regularizing operator $C$; namely, it is possible to show that correctors $\varphi_1$ and $\varphi_2$ exist, and they are such that both $\varphi^\epsilon(u^\epsilon,y^\epsilon) - \varphi(u^\epsilon)$ and $\mathscr{L}^\epsilon \varphi^\epsilon (u^\epsilon,y^\epsilon) - \mathscr{L}^0 \varphi (u^\epsilon)$ are actually infinitesimal in the limit $\epsilon \to 0$, with respect to suitable topologies.
However, in general, checking these conditions requires a certain degree of regularity both for the solutions $(u^\epsilon,y^\epsilon)$ and the dynamics itself - namely the coefficients $A$, $C$ and $b$ cannot be too bad.
In particular, the previous method does not permit to study the limiting behaviour of other equations of interest, as the Surface Quasi-Geostrophic in dimension $d=2$ and the Primitive equations in dimension $d=2,3$, and requires strong assumptions on $C$ in the case of Navier-Stokes equations.

For this reason, also inspired by \cite{GaMe22,Zak-NLS22}, here we develop a modification of the classical method that permits to replace the process $y^\epsilon$ by its linear counterpart $Y^\epsilon$, satisfying:
\begin{align*}
dY^\epsilon_t = \epsilon^{-1} C Y^\epsilon_t dt + \epsilon^{-1/2} Q^{1/2} dW_t,
\quad Y^\epsilon_0 = 0.
\end{align*} 

Loosely speaking, the key observation is that one can rewrite
\begin{align*}
\langle b(y,u) , D_u \varphi_1(y) \rangle
&=
\langle b(Y,u) , D_u \varphi_1(Y) \rangle
+
\langle b(y-Y,u) , D_u \varphi_1(Y) \rangle
\\
&\quad+
\langle b(y,u) , D_u \varphi_1(y-Y) \rangle,
\\
\langle b(y,y) , D_y \varphi_1(u) \rangle
&=
\langle b(Y,Y) , D_y \varphi_1(u) \rangle
+
\langle b(y-Y,Y) , D_y \varphi_1(u) \rangle
\\
&\quad+
\langle b(y,y-Y) , D_y \varphi_1(u) \rangle,
\end{align*} 
and prove that the terms involving the difference $y-Y$ are infinitesimal as $\epsilon \to 0$ when evaluated at $y=y^\epsilon_t$, $Y=Y^\epsilon_t$, and integrated with respect to time.
Therefore, the actual terms of order one in the expression of $\mathscr{L}^\epsilon \varphi^\epsilon$ are given by
\begin{align}\label{1.4ter}
\langle Au + b(u,u) , D_u \varphi \rangle
+
\langle b(Y,u) , D_u \varphi_1(Y) \rangle
+
\langle b(Y,Y) , D_y \varphi_1(u) \rangle
+
\mathscr{L}_y \varphi_2,
\end{align}   
and thus we only need to find a corrector $\varphi_2 = \varphi_2(u,Y)$ such that the previous expression only depends on $u$, and control remainders.
This task is sensibly easier, due to the higher space regularity of $Y$.

\subsubsection*{Non-homogeneous Ornstein-Uhlenbeck generator}
On top of the linearisation trick described above, we can further sharpen our results if we replace $\mathscr{L}_y$ by its inhomogeneous counterpart 
\begin{align} \label{eq:L^eps_y}
\mathscr{L}_y^\epsilon
=
\langle C_\epsilon y, D_y \cdot \rangle 
+
\frac12 Tr(QD^2_y \cdot),
\quad
C_\epsilon = C+\epsilon A.
\end{align} 
Roughly speaking, this choice is of help because it allows us to trade additional space regularity for the correctors $\varphi_1=\varphi_1^\epsilon$ and $\varphi_2=\varphi_2^\epsilon$ with multiplicative factors $\epsilon^{-1}$ (in our examples the operator $C$ is less regularizing than $A$). As a particular instance of this trading, the reader could take a look at \autoref{prop:regularity_varphi1} in \autoref{sec:test}.
A little inconvenience, however, is that using $\mathscr{L}^\epsilon_y$ instead of $\mathscr{L}_y$ produces correctors $\varphi_1^\epsilon$, $\varphi_2^\epsilon$ and effective generator $\mathscr{L}^{0,\epsilon}$ that depend on $\epsilon$, although it is easy to see that this does not affect the limiting behaviour of $u^\epsilon$, i.e. the equation satisfied by $u = \lim_{\epsilon \to 0} u^\epsilon$ remains the same.

\subsection{Structure of the paper}
The paper is structured as follows.

In \autoref{sec:prelim} we introduce the necessary notation and preliminaries for our analysis. In particular, we introduce the abstract spaces and operators governing our system, and we identify their key properties. Working assumptions on the covariance operator $Q$ are declared in this section. Also, here we present the notion of bounded-energy family $\{(u^\epsilon,y^\epsilon)\}_{\epsilon\in(0,1)}$ of weak martingale solutions to our system, that is a family of solutions enjoying some uniform-in-$\epsilon$ bound on the energy.

In \autoref{sec:poisson} we introduce a class of test functions $\psi$ for which it is possible to solve implicitly the Poisson equation $\mathscr{L}^\epsilon_y \phi = -\psi$. The class consists of quadratic functions on $H$ that are continuous maps from some Sobolev space $H^\theta$ to $\R$. We also show that, depending on the regularity of $C$, solutions of the Poisson equations so constructed are more regular than the datum $\psi$, and recover bounds on the regularity of $\phi$ and its derivative in terms of the regularity of $\psi$ and $C$.

In \autoref{sec:test} we apply abstract results on the Poisson equation to carry on the program presented in the Introduction; we identify suitable correctors $\varphi_1^\epsilon$ and $\varphi_2^\epsilon$ to cancel out divergent terms in the expression of $\mathscr{L}^\epsilon \varphi$, and recover the limiting behaviour of the slow variable $u^\epsilon$ alone.

In \autoref{sec:conv}, we prove our main \autoref{thm:main} dividing the proof into three different steps: at first, we prove that the family (of the laws of) $\{u^\epsilon\}_{\epsilon\in(0,1)}$ is tight in a suitable space of functions; then, checking that the contribution due to correctors $\varphi_1^\epsilon$ and $\varphi_2^\epsilon$ is actually negligible as $\epsilon \to 0$,  we prove that every weak accumulation point $u$ is a solution of a limit closed equation; finally, we recognize the different terms in the equation solved by $u$ as the sum of the original slow dynamics, a Stratonovich transport noise and an It\=o-Stokes drift.

The subsequent \autoref{sec:eddy}, whose results are applied only to the Navier-Stokes system, link our results to the ones in \cite{FlLu21} to provide conditions under which a suitable scaling of the parameters $\epsilon,Q$ makes our system converge towards solutions of the deterministic Navier-Stokes equations with additional dissipation.
We provide an explicit example (due to \cite{FlLu21}) of parameters that gives in the limit a large multiple of the Laplace operator.  

Finally, in the main body of this work we prefer to illustrate only the arguments needed for the Navier-Stokes system in dimension $d=3$, for the sake of a clear presentation, and we devote \autoref{sec:models} to discuss all the necessary changes needed to take into account other models of interest.

\section{Preliminaries and assumptions} \label{sec:prelim}

In this section we collect all the necessary notation, assumptions, preliminaries and auxiliary results useful for our analysis.
As discussed in the Introduction, our theory covers at least three models of fundamental relevance in fluid dynamics, which are: the Navier-Stokes equations in dimension $d=2,3$; the Surface Quasi-Geostrophic equations in dimension $d=2$; the Primitive equations in dimension $d=2,3$.
However, each of these models requires \emph{ad hoc} analysis that takes into consideration peculiar features of the dynamics, and minor modifications in the arguments are needed to properly deal with each of them. 
Thus, for the sake of a clear and effective presentation, we decide to present our results first for the 3D Navier-Stokes equations; in the last section, we illustrate the differences arising when considering Surface Quasi-Geostrophic and Primitive equations, and only discuss how to adapt the arguments in the main body of the paper to the changed framework.

Throughout the paper, we use the notation $a \lesssim b$ if there exists an unimportant constant $c\in(0,\infty)$ such that $a \leq cb$.

\subsection{Coupled Navier-Stokes equations}
As a main example of application of the theory here developed, recall the Navier-Stokes system in velocity form, with additive noise and large dissipation at small scales:
\begin{align} \label{eq:system_NS}
\begin{cases}
du^\epsilon_t 
= 
\nu \Delta u^\epsilon_t dt
-
(u^\epsilon_t \cdot \nabla) u^\epsilon_t dt
-
(v^\epsilon_t \cdot \nabla) u^\epsilon_t dt
+
\nabla p^\epsilon_t dt,
\\
\mbox{div}\, u^\epsilon_t = 0,
\\
dv^\epsilon_t 
= 
\epsilon^{-1} C v^\epsilon_t dt
+
\nu \Delta v^\epsilon_t dt
-
(u^\epsilon_t \cdot \nabla) v^\epsilon_t dt
-
(v^\epsilon_t \cdot \nabla) v^\epsilon_t dt
+
\epsilon^{-1}
d\mathcal{W}_t
+
\nabla q^\epsilon_t dt,
\\
\mbox{div}\, v^\epsilon_t = 0,
\end{cases}
\end{align}
where $t \in [0,T]$, $u^\epsilon_t,v^\epsilon_t$ are unknown velocity fields belonging to the space $[L^2(\mathbb{T}^3)]^3$ of zero-mean square integrable velocity fields of the three-dimensional torus $\mathbb{T}^3 = (\R/\mathbb{Z})^3$; $p^\epsilon_t,q^\epsilon_t$ are unknown pressure fields in $L^2(\mathbb{T}^3)$ ensuring the validity of the divergence-free conditions $\mbox{div}\, u^\epsilon_t=\mbox{div}\, v^\epsilon_t =0$; $\mathcal{W}$ is a Wiener process on $[L^2(\mathbb{T}^3)]^3$; $\nu>0$ is the viscosity coefficient; and $\epsilon\in(0,1)$ is a small scaling parameter.
In dimension two a physically relevant choice for $C$ is $C=-Id$ (the identity operator on $[L^2(\mathbb{T}^2)]^2$), see also \cite{FlPa21,FlPa22+} for a justification of the model; in three dimensions we can not deal with the pure friction case and we need a more regularizing operator $C$ (cf. assumption (C2) below).

Denote $H \coloneqq \{ u \in [L^2(\mathbb{T}^3)]^3 , \mbox{div}\,u=0 \}$ the space of periodic, zero-mean, square integrable velocity fields $u$ with null divergence in the sense of distributions.  
Projecting \eqref{eq:system_NS} on $H$ via the Helmhotz projector $\Pi$ we get the equivalent system, without pressure terms:
\begin{align*} 
\begin{cases}
du^\epsilon_t 
= 
\nu \Delta u^\epsilon_t dt
-
\Pi(u^\epsilon_t \cdot \nabla) u^\epsilon_t dt
-
\Pi(v^\epsilon_t \cdot \nabla) u^\epsilon_t dt,
\\
dv^\epsilon_t 
= 
\epsilon^{-1} C v^\epsilon_t dt 
+
\nu \Delta v^\epsilon_t dt
-
\Pi(u^\epsilon_t \cdot \nabla) v^\epsilon_t dt
-
\Pi(v^\epsilon_t \cdot \nabla) v^\epsilon_t dt
+
\epsilon^{-1}
d\Pi \mathcal{W}_t.
\end{cases}
\end{align*}

Therefore, we can recast the system \eqref{eq:system_NS} in the more abstract setting \eqref{eq:system} defining:
\begin{align*}
Au = \nu\Delta u,
\quad
b(u,v)=-\Pi(u \cdot \nabla) v,
\quad
Q^{1/2} W = \Pi\mathcal{W}_t,
\end{align*}
with $W$ being a cylindrical Wiener process on $H$ and $Q$ a covariance operator.
Hereafter, \eqref{eq:system} will be accompanied by initial conditions $(u_0,y_0) \in H \times H$. In order to keep our analysis as simple as possible, we assume $u_0,y_0$ to be deterministic.

\subsection{Abstract spaces and operators}
\subsubsection*{The linear operator $A$ and Sobolev spaces}
The operator $A:D(A) \subset H \to H$ is unbounded, self-adjoint and negative definite.
For $s \in \R$, let us define the Sobolev space $H^s$ by the relation $H^s \coloneqq D((-A)^{s/2})$.
Sobolev spaces form a Hilbert scale in the sense of Krein-Petunin \cite{KrPe66}, with respect to the operator $(-A)^{1/2}$: 
\begin{align*}
\langle f, g\rangle_{H^s} = 
\langle (-A)^{s/2} f,(-A)^{s/2} g \rangle. 
\end{align*}
In particular, the map $(-A)^{s/2}: H^{r+s} \to H^r$ is an isomorphism for every $s,r \in \R$ and the following interpolation inequality holds between $H^{s_1}$ and $H^{s_2}$, for $s_1, s_2 \in \R$ and $\lambda \in (0,1)$:
\begin{align*}
\|f\|_{H^{s_\lambda}}
\leq
\|f\|_{H^{s_1}}^{\lambda}
\|f\|_{H^{s_2}}^{1-\lambda}, \quad
s_\lambda = \lambda s_1 + (1-\lambda) s_2.
\end{align*}
Sobolev spaces $H^s$ embed continuously into $L^p = L^p(\mathbb{T}^d)$ spaces provided $s>d/2$ and $p \in [1,\infty]$ or $s \in (0,d/2)$ and $p \leq 2^\star = \frac{2d}{d-2s}$.

For $\alpha \in (0,1)$, $p \geq 1$ and $s \in \R$, we define $W^{\alpha,p}([0,T],H^s)$ as the Sobolev-Slobodeckij space of all $u \in L^p([0,T],H^s)$ such that
\begin{align*}
\int_0^T \int_0^T \frac{\|u_t-u_s\|_{H^s}^p}{|t-s|^{1+\alpha p}} dt ds < \infty,
\end{align*}
endowed with the norm
\begin{align*}
\|u\|_{W^{\alpha,p}([0,T],H^s)}^p
&\coloneqq
\int_0^T \|u_t\|^p_{H^s} dt
+
\int_0^T \int_0^T \frac{\|u_t-u_s\|_{H^s}^p}{|t-s|^{1+\alpha p}} dt ds.
\end{align*}
We recall the following compactness criterium from \cite{Si86} due to Simon.
\begin{lem} \label{lem:Simon}
For $\sigma>0$, $\alpha > 1/p$ and $\beta \in (0,\sigma)$ we have the compact embeddings:
\begin{align*}
L^2([0,T],H^1) \cap W^{\alpha,p}([0,T],H^{-\sigma})
\subset
L^2([0,T],H);
\\
L^\infty([0,T],H) \cap W^{\alpha,p}([0,T],H^{-\sigma})
\subset
C([0,T],H^{-\beta}).
\end{align*}
\end{lem}
Denote $\mathscr{S} \coloneqq \cap_{s \in \R} H^s$ the class of smooth elements $h \in H$, and define
\begin{align*}
F &\coloneqq \left\{ \varphi : H \to \mathbb{R},
\quad \exists h \in \mathscr{S} \mbox{ such that } \varphi(u) = \langle u,h \rangle \right\}.
\end{align*}
Distributions on $H$ are elements of the space $\mathscr{S}' \coloneqq \cup_{s \in \R} H^s$. Every $\varphi \in F$ is continuous from $\mathscr{S}'$ to $\R$.

\subsubsection*{The linear operator $C$}

We assume
\begin{itemize}
\item [({\bf C1})]
The operator $C:D(C) \subset H \to H$ is self-adjoint and negative definite, with principal eigenvalue $-\lambda_0 < 0$;
\item [({\bf C2})]
There exist $\Gamma \geq \gamma>1/4$ such that $\|x\|_{H^{s+\beta\gamma}}^2 \lesssim \|(-C)^{\beta/2}x\|_{H^s}^2 \lesssim \|x\|_{H^{s+\beta\Gamma}}^2$ for every $s\in \R$, $\beta >0$. 
\end{itemize}
The previous assumptions imply that the operators $C$ and $C_\epsilon \coloneqq C + \epsilon A$ generate $C_0$-semigroups on $H$, that we denote respectively $e^{Ct}$ and $e^{C_\epsilon t}$, $t > 0$. 
Moreover, for every $s \in \R$ and $\beta_1>0$ it holds uniformly in $t>0$ and $\epsilon\in(0,1)$:
\begin{align*}
\|(-C_\epsilon)^{\beta_1}e^{C_\epsilon t}\|_{H^s \to H^s} 
&\lesssim \frac{e^{-\lambda_0 t/2}}{t^{\beta_1}};
\end{align*}
by interpolation, since the operators $(-C_\epsilon)^{-1} C$ and $(-C_\epsilon)^{-1} \epsilon A$ are bounded, we also have for every $\theta \in [\gamma,1]$, $\lambda = \frac{\lambda_0 (1-\theta)}{2(1-\gamma)}$: 
\begin{align*}
\|e^{C_\epsilon t}\|_{H^s \to H^{s+2\theta\beta_1}} 
&\lesssim 
\|(-C)^{\beta_1}e^{C_\epsilon t}\|_{H^s \to H^s}^{\frac{1-\theta}{1-\gamma}}
\|(-A)^{\beta_1}e^{C_\epsilon t}\|_{H^s \to H^s}^{\frac{\theta-\gamma}{1-\gamma}}
\lesssim
\epsilon^{-\beta_1 \frac{\theta-\gamma}{1-\gamma}} \frac{e^{-\lambda t}}{t^{\beta_1}}.
\end{align*}

In addition, for every $s \in \R$ and $\beta_2 \in [0,1]$ they hold:
\begin{align*}
\|(-C_\epsilon)^{-{\beta_2}}(e^{C_\epsilon t}-1)\|_{H^s \to H^s} &\lesssim t^{\beta_2},
\quad
\|e^{C_\epsilon t}-1\|_{H^s \to H^{s-2\Gamma\beta_2 }}
\lesssim
t^{\beta_2}
\end{align*}
uniformly in $t>0$ and $\epsilon\in(0,1)$, and moreover the difference of the semigroups $e^{C_\epsilon t} - e^{C t}$ satisfies\footnote{To see this, one can define $y_t \coloneqq e^{C_\epsilon t}x - e^{C t}x$, $x \in H$ and notice that $y_t=\int_0^t C y_s ds + \epsilon \int_0^t A e^{C_\epsilon s}x ds$; since $y_0=0$, Duhamel's Formula gives $y_t=\epsilon\int_0^t e^{C(t-s)}A e^{C_\epsilon s}x ds$, and using $\|\epsilon^{1-\beta_2}(-A)^{1-\beta_2}e^{C_\epsilon t}\|_{H^s \to H^{s}} \lesssim e^{-\lambda_0 t/2}t^{\beta_2-1}$ produces the desired inequality.} $\|e^{C_\epsilon t} - e^{C t}\|_{H^{\theta+2\beta_2} \to H^\theta} \lesssim  \epsilon^{\beta_2}$ uniformly in $t >0$. 

Finally, for every $\beta_2 \in [0,1]$ the operator $G_\epsilon \coloneqq (-C_\epsilon)^{-1}-(-C)^{-1} = \epsilon (-C)^{-1}A(-C_\epsilon)^{-1}$ satisfies $\|G_\epsilon\|_{H^s \to H^{s+2\gamma(1+\beta_2)-2\beta_2}} \lesssim \epsilon^{\beta_2}$.

\subsubsection*{The bilinear operator $b$}
Concerning the nonlinearity $b$ of the Navier-Stokes equations in velocity form, we have the following properties\footnote{That will involve different regularity when dealing with Primitive equations.} (take $d=3$ below): 
\begin{itemize}
\item[({\bf B1})]
$b:H^s \times H^{\theta_0} \to H^s$ is bilinear and continuous for every $s \in \R$, $s<d/2$, $\theta_0>1+d/2$;
\item[({\bf B2})]
$b:H^s \times H^{\theta_1} \to H^s$ is bilinear and continuous for every $s \in \R$, $s \geq d/2$ and $\theta_1>1+s$;
\item[({\bf B3})]
$b:H^s \times H^r \to H^{s+r-1-d/2}$ is bilinear and continuous if $s,r-1 \in (-d/2,d/2)$, $s+r>1$.
\end{itemize}
For every smooth $x_i \in \mathscr{S}$, $i=1,2,3$ it holds $\langle b(x_1,x_2),x_3 \rangle= -\langle b(x_1,x_3),x_2 \rangle$, by integration by parts. Therefore, if there exist $s_i \in \R$, $i=1,2,3$, such that either $|\langle b(x_1,x_2),x_3 \rangle| \lesssim \Pi_{i=1,2,3}\|x_i\|_{H^{s_i}}$ or $|\langle b(x_1,x_3),x_2 \rangle| \lesssim \Pi_{i=1,2,3}\|x_i\|_{H^{s_i}}$, we can extend the other one by continuity preserving the same bounds. We summarise this property with:
\begin{itemize}
\item[({\bf B4})]
$\langle b(x_1,x_2),x_3 \rangle= -\langle b(x_1,x_3),x_2 \rangle$ for every $x_i \in \mathscr{S}'$, $i=1,2,3$ such that either one of the scalar products is well-defined.
\end{itemize}
In the following, we will denote without explicit mention $\theta_0,\theta_1=\theta_1(s)$ any constants such that (B1) and (B2) hold.

\subsubsection*{The covariance operator $Q$}
We assume that the covariance operator $Q:H\to H$ satisfies the following properties: 
\begin{itemize}
\item[(\textbf{Q1})]
$Q$ is symmetric, positive semidefinite and commutes with $C$. The following operators on $H$ are trace-class for every $t \geq 0$:
\begin{align*}
e^{Ct}Q e^{Ct}, \quad
Q_\infty \coloneqq \int_0^\infty e^{Ct}Qe^{Ct} dt
=
\frac12 (-C)^{-1} Q;
\end{align*}
\item[(\textbf{Q2})]
Denoting $\mathcal{N}(0,Q_\infty)$ the Gaussian measure on $H$ with covariance $Q_\infty$ and $s_0 = \max \{\theta_0,2\Gamma \}$, it holds $\int_H \|w\|_{H^{s_0}}^2 d\mathcal{N}(0,Q_\infty)(w) < \infty$.  
\end{itemize}
In (Q2) above, $\theta_0$ can be any real number such that (B1) holds true and $\Gamma$ is as in (C2).
It is easy to see that under (Q1)-(Q2) the following hold true: $e^{C_\epsilon t}Q e^{C_\epsilon t}$, $Q_\infty^\epsilon \coloneqq \int_0^\infty e^{C_\epsilon t}Qe^{C_\epsilon t} dt$ are trace-class (although in general $Q^\epsilon_\infty \neq \frac12 (-C_\epsilon )^{-1} Q$ since we do not assume $A$ and $Q$ commuting) and $\int_H \|w\|_{H^{\theta_0}}^2 d\mathcal{N}(0,Q^\epsilon_\infty)(w) \leq 1+\int_H \|w\|_{H^{\theta_0}}^2 d\mathcal{N}(0,Q_\infty)(w) < \infty$ for every $\epsilon\in(0,1)$.

\subsection{Ornstein-Uhlenbeck semigroup} \label{ssec:OU} 
Assume (Q1)-(Q2).
For every $\epsilon\in(0,1)$ and $y \in H$ there exists a unique solution $Y^y=Y^y(\epsilon)$ of the Ornstein-Uhlenbeck equation
\begin{align*} 
dY^y_t = C_\epsilon Y^y_t dt + Q^{1/2}dW_t, \quad Y^y_0 = y,
\end{align*} 
that is explicitly given by the formula 
\begin{align*}
Y^y_t = e^{C_\epsilon t}y + W^{C_\epsilon,Q}_t,
\quad
W^{C_\epsilon,Q}_t = \int_0^t e^{C_\epsilon(t-s)} Q^{1/2} dW_s.
\end{align*}
The Ornstein-Uhlenbeck semigroup $P_t^\epsilon:C_b(H) \to C_b(H)$ is defined by
\begin{align*}
P_t^\epsilon \psi (y) \coloneqq \mathbb{E}\left[ \psi (Y_{t}^y)\right], 
\quad \psi \in C_b(H), \,\,\, y \in H,
\end{align*} 
and it is a semigroup by Markov property.
It can be extended uniquely to a strongly continuous semigroup of $1$-Lipschitz maps on $L^2(H,\mu^\epsilon)$, $\mu^\epsilon \coloneqq \mathcal{N}(0,Q^\epsilon_\infty)$, see \cite[Theorem 10.1.5]{DPZa02}.
The Gaussian measure $\mu^\epsilon$ is concentrated on $H^{\theta_0} \subset H$, and $\mu^\epsilon$ is invariant for $P_t^\epsilon$, i.e.
\begin{align*} 
\int_H P_t^\epsilon\psi(y) d\mu^\epsilon(y) = \int_H \psi(y) d\mu^\epsilon(y), \quad \forall \psi \in L^2(H,\mu^\epsilon). 
\end{align*}

The domain $D(\mathscr{L}_y^\epsilon)$ of the generator $\mathscr{L}_y^\epsilon : D(\mathscr{L}_y^\epsilon) \to L^2(H,\mu^\epsilon)$ is defined as the set
\begin{align*}
D(\mathscr{L}_y^\epsilon) 
\coloneqq
\left\{\psi \in L^2(H,\mu^\epsilon) : \exists \lim_{t \to 0^+} \frac{P_t^\epsilon \psi - \psi}{t} \in L^2(H,\mu^\epsilon)  \right\},
\end{align*} 
and $\mathscr{L}_y^\epsilon$ acts on $\psi \in D(\mathscr{L}_y^\epsilon)$ as $\mathscr{L}_y^\epsilon \psi \coloneqq \lim_{t \to 0^+} \frac{P_t^\epsilon \psi - \psi}{t}$. 
The generator $\mathscr{L}_y^\epsilon$ is a closed operator on $L^2(H,\mu^\epsilon)$.

\subsection{Stochastic framework}
Let us recall some basic notions and notations from stochastic analysis. For a more in-depth review of the general theory of stochastic equations we refer to \cite{DPZa14}.

Let $(\Omega, \mathcal{F}, \mathbb{P})$ be a probability space supporting a sequence of standard Wiener processes $\{W^k\}_{k\in \N}$ adapted to a common filtration $\{\mathcal{F}_t \}_{t \geq 0}$, assumed complete and right continuous.
Given any complete orthonormal system $\{ e_k \}_{k \in \N}$ of $H$, we define a \emph{cylindrical} Wiener process $W$ on $H$ to be equal to the formal series $W = \sum_{k \in\N} W^k e_k$. For every $t\geq 0$, $W_t$ is well-defined as a random variables taking values in a space of distribution $H^{-s}$, for some $s$ sufficiently large so that the embedding $H \subset H^{-s}$ is Hilbert-Schmidt.
We call any such $(\Omega, \mathcal{F}, \{\mathcal{F}_t \}_{t \geq 0}, \mathbb{P}, W)$ a \emph{stochastic basis}.

\subsection{Notion of solution and energy estimates}
In the following, $H_w$ denotes the space $H$ endowed with the weak topology { and $\mathcal{B}([0,T], H)$ the space of $H$-valued bounded functions (not necessarily continuous) endowed with the supremum norm. A similar notation is used for bounded functions on $[0,T]$ taking values in a general Banach space}.
\begin{defn} \label{def:sol}
We say that the family $\{(u^\epsilon,y^\epsilon)\}_{\epsilon\in(0,1)}$ is a \emph{bounded-energy family of weak martingale solutions} to \eqref{eq:system_2} if for every $\epsilon\in(0,1)$ there exists a stochastic basis $(\Omega, \mathcal{F}, \{\mathcal{F}_t \}_{t \geq 0}, \mathbb{P}, W)$ such that the following hold :
\begin{itemize}
\item[(\textbf{S1})]
$(u^\epsilon,y^\epsilon) : \Omega \times [0,T] \to  H \times H$ is $\{\mathcal{F}_t\}$-progressively measurable, with paths
$u^\epsilon, y^\epsilon \in C([0,T],H_w) \,\cap\, L^2([0,T],H^1)$, $\mathbb{P}$-almost surely;
\item[(\textbf{S2})]
for every $h \in \mathscr{S}$, the following equalities hold $\mathbb{P}$-almost surely for every $t \in [0,T]$:
\begin{align*}
\langle u^\epsilon_t, h\rangle 
&= 
\langle u_0, h\rangle
+
\int_0^t
\langle u^\epsilon_s, Ah \rangle
+ 
\int_0^t
\langle b(u^\epsilon_s,u^\epsilon_s) , h \rangle ds 
+ \epsilon^{-1/2} \int_0^t
\langle b(y^\epsilon_s,u^\epsilon_s) , h \rangle ds,
\\
\langle y^\epsilon_t, h\rangle 
&= 
\langle y_0, h\rangle
+
\epsilon^{-1} \int_0^t 
\langle y^\epsilon_s , C h \rangle ds 
+
\int_0^t 
\langle y^\epsilon_s , A h \rangle ds
+ 
\int_0^t
\langle b(u^\epsilon_s,y^\epsilon_s) , h \rangle ds
\\
&\quad 
+ \epsilon^{-1/2}\int_0^t
\langle b(y^\epsilon_s,y^\epsilon_s) , h \rangle ds + 
\epsilon^{-1/2} \langle Q^{1/2} W_t , h \rangle;
\end{align*}
\item[(\textbf{S3})]
the family $\{u^\epsilon\}_{\epsilon\in(0,1)}$ is uniformly bounded in
\begin{align*}
\mathcal{U} 
\coloneqq 
L^\infty(\Omega,{\mathcal{B}([0,T],H)}) \cap L^\infty(\Omega,L^2([0,T],H^1));
\end{align*}
\item[(\textbf{S4})]
for every fixed $p<\infty$, the family $\{y^\epsilon\}_{\epsilon\in(0,1)}$ is uniformly bounded in 
\begin{align*}
\mathcal{Y}
\coloneqq
{
\mathcal{B}([0,T],L^p(\Omega,H)) \cap L^p(\Omega,L^2([0,T],H^1))}.
\end{align*}
\end{itemize}
\end{defn}

It is worth to comment on the previous definition. 

First of all, since we are working on the intersection of two fields and to avoid any confusion with the terminology, let us specify that here we are working with \emph{analytically weak, probabilistically martingale} solutions.
Solutions are analytically weak since they solve \eqref{eq:system_2} only when tested against smooth test functions $h \in \mathscr{S}$. 
They are martingale solutions (sometimes referred to also as probabilistically weak solutions) since the stochastic basis in not given a priori (that would be called pathwise or probabilistically strong solutions). 
To avoid any misunderstanding we point out that hereafter the stochastic basis $\{(\Omega^\epsilon, \mathcal{F}^\epsilon, \{\mathcal{F}^\epsilon_t \}_{t \geq 0}, \mathbb{P}^\epsilon, W^\epsilon)\}_{\epsilon\in(0,1)}$ will be always dependent on $\epsilon$, but we shall drop the indices for notational simplicity.

Second, our solutions form a \emph{bounded-energy family} since in (S3)-(S4) we require suitable energy bounds to hold uniformly in $\epsilon\in(0,1)$.

In the classical theory of deterministic Navier-Stokes equations subject to external forcing $f \in L^1([0,T],H)$:
\begin{align*}
\begin{cases}
du_t + (u_t \cdot \nabla) u_t dt = \nu\Delta u_t dt + \nabla p_t dt + f_t dt,\\
\mbox{div}\, u_t = 0,
\end{cases}
\end{align*}
a very fundamental concept is that of \emph{Leray-Hopf} weak solutions, namely (analytically) weak solutions $u$ enjoying the energy inequality
\begin{align*}
\frac12 \|u_t \|_{H}^2 
+ 
\int_0^t \|u_s \|_{H^1}^2 ds 
\leq 
\frac12 \|u_0\|_H^2 
+
\int_0^t \langle u_s,f_s \rangle ds.
\end{align*}
In the stochastic setting the picture is more complicated since, when the external forcing $f = Q^{1/2} W$ is a stochastic process: \emph{i}) sensible bounds can only be obtained in expected value; and \emph{ii}) formally applying It\=o Formula to $\|u_t \|_{H}^2$ introduces an additional term $Tr(Q)dt$ on the right hand side of the estimate.

In \cite{FlRo08}, the authors propose a notion of solution which encodes the energy inequality in the requirement that the process
\begin{align*}
E^p_t
&\coloneqq
\frac12\|u_t\|_H^{2p}
+
p \int_0^t \|u_s\|_H^{2p-2} \|u_s \|_{H^1}^2 ds 
-
\frac12 \|u_0\|_H^{2p}
\\
&\quad-
\frac{p(2p-1)}{2} Tr(Q) \int_0^t \|u_s\|_H^{2p-2} ds
\end{align*}
be an \emph{almost sure super martingale} for every positive integer $p$, namely $\mathbb{E}\left[ E^p_t\right] < \infty$ for all $t \in [0,T]$ and there exists a Lebesgue measurable set $\mathcal{T} \subset (0,T]$, with null Lebesgue measure, such that $\mathbb{E}\left[ E^p_t \mathbf{1}_A\right] \leq \mathbb{E}\left[ E^p_s \mathbf{1}_A\right]$ for every $s \in  \mathcal{T}$, every $t \geq s$ and every $A \in \mathcal{F}_s$.

However, for our purposes there are some limitations in considering solutions satisfying some kind of energy inequality, since: \emph{i}) it does not seem immediate to recover uniform bounds in $\epsilon\in(0,1)$, and \emph{ii}) we do not need energy inequality but just energy bounds, and recents developments in convex integration suggest that the class of weak solutions to Navier-Stokes equations with bounded energy may be stricly larger than the class of Leray-Hopf weak solutions, see \cite{BuVi19} for a deterministic result and \cite{HoZhZh21+} for a stochastic one (even though the solutions constructed there are not known to satisfy $H^1$ bounds in the space variable).

In order to construct a bounded-energy family of weak martingale solutions to \eqref{eq:system_2}, one can make use of classical compactness arguments involving the Galerkin approximation scheme:
\begin{align} \label{eq:system_Gal}
\begin{cases}
du^{\epsilon,n}_t 
= 
Au^{\epsilon,n}_t dt 
+ 
\Pi_n b(u^{\epsilon,n}_t,u^{\epsilon,n}_t) dt 
+ 
\epsilon^{-1/2} \Pi_n b(y^{\epsilon,n}_t,u^{\epsilon,n}_t) dt,
\\
dy^{\epsilon,n}_t 
= 
\epsilon^{-1} Cy^{\epsilon,n}_t dt 
+
Ay^{\epsilon,n}_t dt 
+ 
\Pi_n b(u^{\epsilon,n}_t,y^{\epsilon,n}_t) dt 
+ 
\epsilon^{-1/2}\Pi_n b(y^{\epsilon,n}_t,y^{\epsilon,n}_t) dt 
\\
\qquad\qquad+ \epsilon^{-1/2} \Pi_n Q^{1/2} dW_t,
\end{cases}
\end{align}
where $\{\Pi_n\}_{n \in \N}$ is a family of Galerkin projectors and the initial condition is $(u^{\epsilon,n}_0,y^{\epsilon,n}_0)=(\Pi_n u_0,\Pi_n y_0)$.
{Indeed, since solutions of \eqref{eq:system_Gal} above are smooth in space, uniform energy estimates (S3)-(S4) can be rigorously proved for $(u^{\epsilon,n},y^{\epsilon,n})$ making use of It\=o Formula; then, for every fixed $\epsilon\in(0,1)$, it is possible to prove via Ascoli-Arzelà Theorem that there exist $u^\epsilon,y^\epsilon$ such that $u^{\epsilon,n} \to u^\epsilon$ and $y^{\epsilon,n} \to y^\epsilon$ with respect to a topology that permits to take the limit in the energy estimates (S3)-(S4)}, on the one hand, and in the weak formulation of the equation (S2), on the other (up to a possible change in the underlying stochastic basis, in order to gain adaptedness of the processes $u^\epsilon,y^\epsilon$).
\begin{prop} \label{prop:existence}
There exists at least one bounded-energy family $\{(u^\epsilon,y^\epsilon)\}_{\epsilon\in(0,1)}$ of weak martingale solutions to \eqref{eq:system_2}.
\end{prop}
Existence of a weak martingale solution for fixed $\epsilon\in(0,1)$ is known since \cite{FlGa95}. The only difference here is uniform in $\epsilon$ energy bounds, which require suitable estimates at the level of Galerkin truncations (cf. \autoref{lem:int_y} and \autoref{lem:E_y} below) and compactness arguments well-suited for the passage to the limit $n \to \infty$.
For the sake of completeness, we present a proof of \autoref{prop:existence} in the \autoref{sec:existence_sol}.
Here we limit ourselves to show the needed energy bounds.

\begin{rmk}
Notice that if \eqref{eq:system_2} admits pathwise uniqueness then one has existence of probabilistically strong solution, namely the stochastic basis can be taken independent of $\epsilon$. 
\end{rmk}

\begin{lem} \label{lem:int_y}
For every positive integer $p$ it holds
\begin{align*}
\sup_{\substack{\epsilon\in(0,1),\\ n \in \N}}\,
\int_0^T\mathbb{E}\left[\|y^{\epsilon,n}_s\|^{2p-2}_H \|y^{\epsilon,n}_s\|_{H^\gamma}^2 \right] ds
\lesssim 1.
\end{align*}
\end{lem}

\begin{proof}
Let $\epsilon\in(0,1)$ and $n \in \N$ be fixed, and take an arbitrary $t \in [0,T]$. 
Applying It\=o Formula to $\frac12 \|y^{\epsilon,n}_t\|_H^{2p}$ we get
\begin{align*} 
\frac12 \|y^{\epsilon,n}_t\|_H^{2p} 
&
+ 
\epsilon^{-1} p \int_0^t\|y^{\epsilon,n}_s\|_H^{2p-2} \|(-C)^{1/2}y^{\epsilon,n}_s\|_H^2 ds
+
p \int_0^t\|y^{\epsilon,n}_s\|_H^{2p-2} \|y^{\epsilon,n}_s\|_{H^1}^2 ds
\\
&= \nonumber
\frac12 \|\Pi_ny_0\|_H^{2p} 
+ 
\epsilon^{-1/2} p \int_0^t \|\Pi_n y^\epsilon_s\|_H^{2p-2} \langle y^{\epsilon,n}_s,\Pi_n Q^{1/2}dW_s \rangle
\\
&\quad+ 
\epsilon^{-1} \frac{p(2p-1)}{2} Tr(\Pi_n Q \Pi_n) \int_0^t \|y^{\epsilon,n}_s\|_H^{2p-2} ds.
\end{align*} 
Taking expectations in the expression above with $p=1$  we obtain
\begin{align*} 
\epsilon^{-1}
\int_0^t\mathbb{E}\left[\|y^{\epsilon,n}_s\|_{H^\gamma}^2 \right] ds
&\leq
\frac1{2M} \|y_0\|^2 
+
\epsilon^{-1} \frac{Tr(Q)}{2M} t,
\end{align*}
where we have used $\|(-C)^{1/2}y^{\epsilon,n}_s\|_H^2 \geq M \|y^{\epsilon,n}_s\|_{H^\gamma}^2$ for some unimportant constant $M$; thus we deduce
\begin{align*}
\sup_{\substack{\epsilon\in(0,1),\\ n \in \N}}\,
\int_0^T\mathbb{E}\left[\|y^{\epsilon,n}_s\|_{H^\gamma}^2 \right] ds
\leq
\frac{\epsilon}{2M} 
\|y_0\|_H^2
+
\frac{Tr(Q)T}{2M}
\lesssim 1.
\end{align*}
For $p>1$, we argue as follows: first, recalling $\|y\|^2_{H^\gamma} \geq \nu_1^\gamma \|y\|_H^2$ for some $\nu_1>0$ (the principal eigenvalue of the operator $-A$), for every $t \in [0,T]$ we have
\begin{align*} 
\int_0^t&\mathbb{E}\left[\|y^{\epsilon,n}_s\|_H^{2p-2}\|y^{\epsilon,n}_s\|_{H^\gamma}^2  \right] ds
\\
&\leq
\frac{\epsilon}{2pM}\|y_0\|_H^{2p} 
+
\frac{2p-1}{2M}Tr(Q) \int_0^t\mathbb{E}\left[\|y^{\epsilon,n}_s\|_H^{2p-2} \right] ds
\\
&\leq
\frac{\epsilon}{2pM}\|y_0\|_H^{2p} 
+
\frac{2p-1}{2M}\nu_1^\gamma Tr(Q) \int_0^t\mathbb{E}\left[\|y^{\epsilon,n}_s\|_H^{2p-4} \|y^{\epsilon,n}_s\|_{H^\gamma}^2\right] ds;
\end{align*}
then, since $p-1$ is a positive integer, by induction we have the desired inequality uniformly in $\epsilon\in(0,1)$ and $n \in \N$.
\end{proof}

\begin{lem} \label{lem:E_y}
For every $p \geq 2$ it holds
\begin{align*}
\sup_{\substack{\epsilon\in(0,1),\\ n \in \N}}\,
\sup_{t \in [0,T]}
\left(
\mathbb{E}\left[\|y^{\epsilon,n}_t\|_H^p \right]
+
\int_0^t
\mathbb{E}\left[\|y^{\epsilon,n}_s\|_H^{p-2}\|y^{\epsilon,n}_s\|_{H^1}^2\right] ds \right)
\lesssim 1.
\end{align*}
\end{lem}

\begin{proof}
As in the previous \autoref{lem:int_y}, it is sufficient to prove the result for every positive even integer $p$.
Let us introduce the auxiliary process $Y^{\epsilon,n}_t$ solution of
\begin{align*}
dY^{\epsilon,n}_t
&=
\epsilon^{-1}C_\epsilon Y^{\epsilon,n}_t dt
+
\epsilon^{-1/2} \Pi_nQ^{1/2}dW_s,
\qquad
Y^{\epsilon,n}_0=0,
\end{align*}
so that, by It\=o Formula, the difference process $\zeta^{\epsilon,n}_t \coloneqq y^{\epsilon,n}_t-Y^{\epsilon,n}_t$ satisfies, for every $t \in [0,T]$ and $p \geq 2$
\begin{align} \label{eq:energy_zeta_p}
\|\zeta^{\epsilon,n}_t\|_H^{p}
&+
\epsilon^{-1}pM \int_0^t\|\zeta^{\epsilon,n}_s\|_H^{p-2}\|\zeta^{\epsilon,n}_s\|_{H^\gamma}^2 ds
+
p \int_0^t\|\zeta^{\epsilon,n}_s\|_H^{p-2}\|\zeta^{\epsilon,n}_s\|_{H^1}^2 ds
\\
&\leq \nonumber
\|y_0\|_H^p
+
p \int_0^t
\|\zeta^{\epsilon,n}_s\|_H^{p-2}
\langle b(u^{\epsilon,n}_s , Y^{\epsilon,n}_s),\zeta^{\epsilon,n}_s \rangle ds
\\
&\quad+ \nonumber
\epsilon^{-1/2} p
\int_0^t
\|\zeta^{\epsilon,n}_s\|_H^{p-2} 
\langle b(y^{\epsilon,n}_s,Y^{\epsilon,n}_s),\zeta^{\epsilon,n}_s \rangle ds
\\
&\leq\|y_0\|^p
+ \nonumber
M_1\int_0^t
\|\zeta^{\epsilon,n}_s\|_H^{p-1} \|u^{\epsilon,n}_s\|_H \|Y^{\epsilon,n}_s\|_{H^{\theta_0}}ds
\\
&\quad+ \nonumber
\epsilon^{-1/2}
M_1\int_0^t
\|\zeta^{\epsilon,n}_s\|_H^{p-1} \|y^{\epsilon,n}_s\|_H \|Y^{\epsilon,n}_s\|_{H^{\theta_0}} ds,
\end{align}
where $M_1$ is another unimportant constant.
By Young inequality
\begin{align*}
\|Y^{\epsilon,n}_s\|_{H^{\theta_0}}
\|\zeta^{\epsilon,n}_s\|_H^{p-1}
\leq
\frac{\|Y^{\epsilon,n}_s\|_{H^{\theta_0}}^p}{p}
+
\frac{p-1}{p}\|\zeta^{\epsilon,n}_s\|_H^p,
\end{align*}
and, for every positive constant $c$:
\begin{align*}
\epsilon^{-1/2}
\|y^{\epsilon,n}_s\|_H
\|Y^{\epsilon,n}_s\|_{H^{\theta_0}}
\|\zeta^{\epsilon,n}_s\|_H^{p-1}
&\leq
\frac{c^{-p}}{2p}\|y^{\epsilon,n}_s\|_H^{2p}
+
\frac{c^{-p}}{2p}\|Y^{\epsilon,n}_s\|_{H^{\theta_0}}^{2p}
\\
&\quad+
\epsilon^{-\frac{p}{2(p-1)}}
\frac{(p-1)c^{\frac{p}{p-1}}}{p}\|\zeta^{\epsilon,n}_s\|_H^p.
\end{align*}
Choosing $c=\left(\frac{p^2 M}{2(p-1)\nu_1^\gamma M_1} \right)^{\frac{p-1}{p}}$, the previous inequalities can be plugged into \eqref{eq:energy_zeta_p} to get
\begin{align}  \label{eq:zeta_pwr_p}
\|\zeta^{\epsilon,n}_t\|_H^p
&+
\epsilon^{-1}\int_0^t\|\zeta^{\epsilon,n}_s\|_H^{p-2}\|\zeta^{\epsilon,n}_s\|_{H^\gamma}^2 ds
+
\int_0^t\|\zeta^{\epsilon,n}_s\|_H^{p-2}\|\zeta^{\epsilon,n}_s\|_{H^1}^2 ds
\\
&\lesssim \nonumber
\|y_0\|_H^p
+
\int_0^t
\|Y^{\epsilon,n}_s\|^p_{H^{\theta_0}}ds
+
\int_0^t \|y^{\epsilon,n}_s\|_H^{2p} ds
+
\int_0^t \|Y^{\epsilon,n}_s\|_{H^{\theta_0}}^{2p} ds. \nonumber
\end{align}
Since $\mathbb{E}\left[\|Y^{\epsilon,n}_s\|_{H^{\theta_0}}^{2p}\right]$ is bounded uniformly in $\epsilon,n$ and $s$ by assumption (Q2), and invoking previous \autoref{lem:int_y}, the previous inequality produces the bounds for $\zeta^{\epsilon,n}$:
\begin{gather} \label{eq:bound_zeta_H_p1}
\sup_{\substack{\epsilon\in(0,1),\\ n \in \N}}\,
\sup_{t \in [0,T]}
\left(
\mathbb{E}\left[\|\zeta^{\epsilon,n}_t\|_H^p \right]
+
\int_0^t
\mathbb{E}\left[\|\zeta^{\epsilon,n}_s\|_H^{p-2}\|\zeta^{\epsilon,n}_s\|_{H^1}^2\right] ds \right)
\lesssim 1,
\\
\sup_{\substack{\epsilon\in(0,1),\\ n \in \N}} \label{eq:bound_zeta_H_p2}
\epsilon^{-1}
\int_0^T\mathbb{E}\left[\|\zeta^{\epsilon,n}_s\|_H^{p-2}\|\zeta^{\epsilon,n}_s\|_{H^\gamma}^2 \right] ds
\lesssim 1.
\end{gather}
Since $y^{\epsilon,n} = Y^{\epsilon,n} + \zeta^{\epsilon,n}$, from \eqref{eq:bound_zeta_H_p1} we deduce the thesis.
\end{proof}

The previous \autoref{lem:E_y} gives the uniform-in-$\epsilon$ bounds necessary for the proof of \autoref{prop:existence}.
As a by-product of the previous proof we have also obtained \eqref{eq:bound_zeta_H_p2}, that permits to control the difference between the Galerkin approximation of the small scale process $y^{\epsilon,n}$ and its linearised counterpart $Y^{\epsilon,n}$ in the Sobolev space $H^\gamma$. Recall that we have assumed $\gamma>1/4$, and we can assume $\gamma \in (1/4,1]$ without loss of generality.
A close inspection of the proof of \autoref{prop:existence} in the \autoref{sec:existence_sol} shows that the bound \eqref{eq:bound_zeta_H_p2} is stable under passage to the limit $n \to \infty$; therefore, we can deduce the following
\begin{prop} \label{prop:y-Y}
Let $\{(u^\epsilon,y^\epsilon)\}_{\epsilon\in(0,1)}$ be a bounded-energy family of weak martingale solutions to \eqref{eq:system_2}.
For every $\epsilon\in(0,1)$ let $Y^\epsilon$ be the unique strong solution of
\begin{align} \label{eq:Yeps}
dY^\epsilon_t
&=
\epsilon^{-1}C_\epsilon Y^\epsilon_t dt
+
\epsilon^{-1/2} Q^{1/2}dW_s,
\qquad
Y^\epsilon_0=0.
\end{align}
Then
\begin{align*}
\sup_{\epsilon\in(0,1)}
\epsilon^{-1} 
\int_0^T\mathbb{E}\left[\|y^\epsilon_s-Y^\epsilon_s\|_{H^\gamma}^2 \right] ds
\lesssim 1.
\end{align*}
\end{prop}
The previous result will be fundamental in performing the linearisation trick presented in \autoref{ssec:perturbed_test_function}; cf. also \autoref{prop:linearisation}.

\section{Quadratic functions and solution to the Poisson equation} \label{sec:poisson}

Recall that, as already discussed in \autoref{ssec:perturbed_test_function}, we shall define correctors $\varphi_1^\epsilon$, $\varphi_2^\epsilon$ as solutions to certain Poisson equations $\mathscr{L}_y^\epsilon \phi = -\psi$.
In this section we develop the technology needed to solve the Poisson equation for a class of functions $\psi$ that is large enough for our purposes, namely the class of quadratic functions on Sobolev spaces $H^\theta$.
Moreover, we also provide partial regularity estimates for the so obtained solution $\phi$ in terms of analogous bounds on $\psi$, showing improved regularity (see \autoref{cor:regularity_phi}).
To avoid any confusion for the reader, we point out that all the estimates in the present section are uniform in $\epsilon\in(0,1)$. 

\subsection{Quadratic functions}
Denote $\mathcal{E}_\theta \subset L^2(H,\mu^\epsilon)$, $\theta 
\in (-\infty,\theta_0]$ the space of quadratic functions $\psi:H^\theta \to \mathbb{R}$, namely $\psi \in \mathcal{E}_\theta$ if there exist $a_0 \in \mathbb{R}$, $a_1 : H^\theta \to \mathbb{R}$ linear and bounded, and $a_2 : H^\theta \times H^\theta \to \mathbb{R}$ bilinear, symmetric and bounded such that $\psi(y) = a_0 + a_1(y) + a_2(y,y)$ for every $y \in H^\theta$.
The inclusion in $L^2(H,\mu^\epsilon)$ holds true by (Q2).
Notice that every $\psi \in \mathcal{E}_\theta$ admits a unique rewriting as $\psi = a_0 + a_1 + a_2$: indeed $\psi(ry)=a_0 + r a_1(y) + r^2 a_2(y,y)$, and therefore
\begin{align*}
a_0 = \psi(0),
\quad
a_1(y) = \left. \frac{d}{dr} \psi(ry) \right|_{r=0},
\end{align*}
and by taking the difference $a_2(y,y) = \psi(y)-a_1(y)-a_0$, defining uniquely the quadratic form $a_2(y,y)$ and also its associated symmetric bilinear map, via polarization formula.

For future purposes define
\begin{align*}
\|\psi\|_{\mathcal{E}_\theta}
&\coloneqq
|a_0|
+
\|a_1\|_{H^\theta \to \R}
+
\|a_2\|_{H^\theta \times H^\theta \to \R}
\\
&=
|a_0|
+
\sup_{\substack{y \in H^\theta,\\ \|y\|_{H^\theta} = 1}} |a_1(y)|
+
\sup_{\substack{y,y' \in H^\theta,\\ \|y\|_{H^\theta} = \|y'\|_{H^\theta} = 1}} |a_2(y,y')|.
\end{align*}
The space $\mathcal{E}_\theta$ is Banach when endowed with the norm $\|\cdot\|_{\mathcal{E}_\theta}$, and $\mathcal{E}_{\theta} \subset \mathcal{E}_{\theta'}$ with continuous embedding if $\theta \leq \theta'$. As a notational convention, denote $\mathcal{E} \coloneqq \mathcal{E}_0$.

\begin{lem} \label{lem:commuting}
Let $\psi \in \mathcal{E}$, then for every $t \geq 0$ it holds $P_t^\epsilon \psi \in D(\mathscr{L}_y^\epsilon)$ and $\mathscr{L}_y^\epsilon P_t^\epsilon \psi = P_t^\epsilon \mathscr{L}_y^\epsilon \psi$.
\end{lem}
\begin{proof}
By Markov property $P_s^\epsilon P_t^\epsilon \psi = P_{t+s}^\epsilon \psi$. Recalling that $Y^y$ is a strong solution of $dY^y_t = C_\epsilon Y^y_t dt + Q^{1/2}dW_t$, $Y^y_0 = y$, by It\=o Formula we have
\begin{align*}
P_s^\epsilon P_t^\epsilon \psi - P_t^\epsilon \psi
&=
\mathbb{E}\left[ \psi (Y_{t+s}^\cdot)\right]-\mathbb{E}\left[ \psi (Y_{t}^\cdot)\right]
\\
&=
\int_t^{t+s}
\mathbb{E} \left[ \langle C_\epsilon Y_r^\cdot , D_y \psi(Y_r^\cdot) \rangle + \frac12 Tr(Q D^2_y \psi(Y_r^\cdot)) \right] dr.
\end{align*}
Let $\psi(y)=a_0 + a_1(y) + a_2(y,y)$ be the canonical decomposition of $\psi \in \mathcal{E}$. Since $\int_H \|C_\epsilon
y\|_H\|y\|_H d\mu^\epsilon(y) < \infty$ uniformly in $\epsilon$ by our assumptions on $C$ and $Q$, we have
\begin{align*}
\bar{\psi}
&\coloneqq 
\langle C_\epsilon
y , D_y \psi(y) \rangle + \frac12 Tr(Q D^2_y \psi) 
\\
&=
a_1(C_\epsilon y) + 2a_2(C_\epsilon y,y) + \frac12 Tr(Q D^2_y a_2)
\in L^2(H,\mu^\epsilon).
\end{align*}
In addition, the semigroup $P_\cdot \bar{\psi}$ is right continuous at time $t$, with respect to the $L^2(H,\mu^\epsilon)$ topology, and therefore we have in the limit $s \to 0^+$
\begin{align*}
\left\| \frac{P_s^\epsilon P_t^\epsilon \psi - P_t^\epsilon \psi}{s} - P_t^\epsilon \bar{\psi} \right\|_{L^2(H,\mu^\epsilon)}
&=
\left\| \frac1s \int_t^{t+s}
P_r^\epsilon \bar{\psi} dr - P_t^\epsilon \bar{\psi} \right\|_{L^2(H,\mu^\epsilon)}
\\
&\leq
\frac1s \int_t^{t+s} 
\left\| P_r^\epsilon \bar{\psi}-P_t^\epsilon \bar{\psi} \right\|_{L^2(H,\mu^\epsilon)} dr \to 0.
\end{align*}
In particular, this means $P_t^\epsilon \psi \in D(\mathscr{L}_y^\epsilon)$ and $\mathscr{L}_y^\epsilon P_t^\epsilon \psi =  P_t^\epsilon \bar{\psi}$. Finally, taking $t=0$ in the previous formula we deduce $\mathscr{L}_y^\epsilon \psi = \bar{\psi}$, that yields $\mathscr{L}_y^\epsilon P_t^\epsilon \psi =  P_t^\epsilon \bar{\psi} = P_t^\epsilon \mathscr{L}_y^\epsilon \psi$.
\end{proof}

\begin{lem} \label{lem:exp_mixing}
The semigroup $P_t^\epsilon$ is exponentially mixing at zero when restricted to $\mathcal{E}_\theta$, $\theta \in (-\infty,\theta_0]$, namely for every $\psi \in \mathcal{E}_\theta$ and $t\geq 0$
\begin{align*}
\left| P_t^\epsilon \psi (0) - \int_H \psi(w) d\mu^\epsilon(w)\right| \lesssim \| \psi \|_{\mathcal{E}_\theta} e^{-\lambda_0 t}. 
\end{align*}
\end{lem}
\begin{proof}
Let $\psi \in \mathcal{E}_\theta$ be given by $\psi(y) = a_0 + a_1(y) + a_2(y,y)$, $y \in H^\theta$. 
Recall
\begin{align*}
P_t^\epsilon \psi (y) 
&= 
a_0 + \mathbb{E} \left[ a_2(W^{C_\epsilon,Q}_t,W^{C_\epsilon,Q}_t)\right] + a_1(e^{C_\epsilon t}y) + a_2(e^{C_\epsilon t}y,e^{C_\epsilon t}y)
\\
&=
P_t^\epsilon \psi (0) + a_1(e^{C_\epsilon t}y) + a_2(e^{C_\epsilon t}y,e^{C_\epsilon t}y),
\end{align*}
and therefore 
\begin{align*}
\left| P_t^\epsilon \psi (0)-P_t^\epsilon \psi (y)\right|
&\leq
\left| a_1(e^{C_\epsilon t}y)\right|
+
\left| a_2(e^{C_\epsilon t}y,e^{C_\epsilon t}y)\right|
\\
&\lesssim
\|\psi\|_{\mathcal{E}_\theta} e^{-\lambda_0 t} \|y\|_{H^\theta}
+
\|\psi\|_{\mathcal{E}_\theta} e^{-2\lambda_0 t} \|y\|_{H^\theta}^2.
\end{align*}
Since $\mu^\epsilon$ is invariant for $P_t^\epsilon$,
\begin{align*}
\left| P_t^\epsilon \psi (0) - \int_H \psi(w) d\mu^\epsilon(w)\right| 
&=
\left| P_t^\epsilon \psi (0) - \int_H P_t^\epsilon\psi(w) d\mu^\epsilon(w)\right| 
\\
&\leq
\int_H \left| P_t^\epsilon \psi (0) - P_t^\epsilon\psi(w)\right| d\mu^\epsilon(w) 
\\
&\lesssim
\|\psi\|_{\mathcal{E}_\theta} e^{-\lambda_0 t}\int_H \|w\|_{H^\theta}d\mu^\epsilon(w) 
\\
&\quad+
\|\psi\|_{\mathcal{E}_\theta} e^{-2\lambda_0 t}\int_H\|w\|_{H^\theta}^2d\mu^\epsilon(w) 
\\
&\lesssim
\|\psi\|_{\mathcal{E}_\theta} e^{-\lambda_0 t}.
\end{align*}
\end{proof}

\subsection{Solution to the Poisson equation}

Recall that the operator $\mathscr{L}_y^\epsilon$ is closed under our assumptions, and therefore the space $D(\mathscr{L}_y^\epsilon)$ is complete when endowed with the graph norm
\begin{align*}
\| \psi \|_{D(\mathscr{L}_y^\epsilon)}
&\coloneqq
\| \psi \|_{L^2(H,\mu^\epsilon)}
+
\| \mathscr{L}_y^\epsilon \psi \|_{L^2(H,\mu^\epsilon)}.
\end{align*}
Also, by \autoref{lem:commuting} we have $\mathcal{E} \subset D(\mathscr{L}_y^\epsilon)$ with continuous embedding, in virtue of the following inequality:
\begin{align*}
\| \psi \|_{D(\mathscr{L}_y^\epsilon)}^2
&\lesssim
\int_H |\psi(w)|^2 d\mu^\epsilon(w)
+
\int_H \left| \langle C_\epsilon w, D_y \psi (w) \rangle \right|^2 d\mu^\epsilon(w) + Tr(QD^2_y \psi)^2
\\
&\lesssim \|\psi\|_\mathcal{E}^2 \left(
\int_H \|w\|_H^4 d\mu^\epsilon(w)
+
\int_H \|w\|_{H^{2\Gamma}}^2 (1+\|w\|_H)^2 d\mu^\epsilon(w) + Tr(Q)^2 \right)
\\
&\lesssim \|\psi\|_\mathcal{E}^2 \left(
\int_H \|w\|_H^2 d\mu^\epsilon(w)
+
\int_H \|w\|_{H^{2\Gamma}}^2 d\mu^\epsilon(w) + Tr(Q) \right)^2
\\
&\lesssim \|\psi\|_\mathcal{E}^2.
\end{align*}
{In the third inequality above we have used that, for Gaussian measures on $H$, the fourth moment is controlled with the square of the second moment (up to unimportant multiplicative constants).}

\begin{lem} \label{lem:time_integral_theta1}
Let $\psi \in \mathcal{E}_\theta$, $\theta \in [0,\theta_0)$, be given by $\psi(y) = a_1(y)$. Then for every $T>0$ we have $\psi(e^{C_\epsilon T}\cdot) \in \mathcal{E}$, $\int_{1/T}^T \psi(e^{C_\epsilon t}\cdot) dt \in \mathcal{E}$, and
\begin{align*}
\lim_{T \to \infty}
\int_{1/T}^T \psi(e^{C_\epsilon t}\cdot) dt
=
\psi((-C_\epsilon)^{-1}\cdot),
\end{align*}
the limit being understood with respect to the $D(\mathscr{L}_y^\epsilon)$ topology.
\end{lem}

\begin{proof}
First of all, there exists a vector $\mathbf{a}_1 \in H^{-\theta}$ such that $a_1(y)=\langle y, \mathbf{a}_1 \rangle$ for every $y \in H^\theta$. 
Hence, for every $T>0$ we have $\psi(e^{C_\epsilon T}\cdot) \in \mathcal{E}$ since
$|\psi(e^{C_\epsilon T}y)|=|\langle e^{C_\epsilon T}y, \mathbf{a}_1 \rangle| \leq \|e^{C_\epsilon T}y\|_{H^\theta} \|\mathbf{a}_1 \|_{H^{-\theta}} \lesssim e^{-\lambda_0 T/2\gamma} T^{-\theta/2} \|y\|_H \|\mathbf{a}_1 \|_{H^{-\theta}}$. 
As a consequence, $\int_{1/T}^T \psi(e^{C_\epsilon t}\cdot)dt \in \mathcal{E} \subset D(\mathscr{L}_y^\epsilon)$ as well, since
\begin{align*}
\left\| \int_{1/T}^T \psi(e^{C_\epsilon t}\cdot)dt \right\|_{\mathcal{E}}
&\leq
\int_{1/T}^T \|\psi(e^{C_\epsilon t}\cdot)\|_{\mathcal{E}} dt 
< \infty.
\end{align*}

Let us finally check $\int_{1/T}^T \psi(e^{C_\epsilon t}\cdot) dt \to \psi((-C_\epsilon)^{-1}\cdot)$ in $D(\mathscr{L}_y^\epsilon)$ as $T \to \infty$.
For every $T>0$ and $y \in H$ we have
\begin{align*}
\int_{1/T}^T \psi(e^{C_\epsilon t}y) dt
&=
\int_{1/T}^T \langle e^{C_\epsilon t}y, \mathbf{a}_1 \rangle  dt
=
\langle \left(\int_{1/T}^T e^{C_\epsilon t}  dt \right) y, \mathbf{a}_1\rangle
\\
&=
\langle (e^{C_\epsilon/T}-e^{C_\epsilon T})(-C_\epsilon)^{-1} y , \mathbf{a}_1  \rangle,
\end{align*} 
and therefore we only have to check $\langle (e^{C_\epsilon/T}-e^{C_\epsilon T}-1)(-C_\epsilon)^{-1} \cdot , \mathbf{a}_1  \rangle \to 0$ and $\langle (e^{C_\epsilon/T}-e^{C_\epsilon T}-1) \cdot , \mathbf{a}_1  \rangle \to 0$ in $L^2(H,\mu^\epsilon)$ (recall that $D^2_y \psi = 0$ since $\psi$ is linear).
We only prove the second convergence, the former being easier.
\begin{align*}
\int_H \left| \langle (e^{C_\epsilon/T}-e^{C_\epsilon T}-1) w , \mathbf{a}_1  \rangle \right|^2 d\mu^\epsilon(w)
&\lesssim
\int_H \left| \langle (e^{C_\epsilon/T}-1) w , \mathbf{a}_1  \rangle \right|^2 d\mu^\epsilon(w)
\\
&\quad+
\int_H \left| \langle e^{C_\epsilon T} w , \mathbf{a}_1  \rangle \right|^2 d\mu^\epsilon(w)
\\
&\leq
\int_H \| (e^{C_\epsilon/T}-1) w \|_{H^\theta}^2 \|\mathbf{a}_1\|_{H^{-\theta}}^2 d\mu^\epsilon(w)
\\
&\quad+
\int_H \| e^{C_\epsilon T} w \|_{H^\theta}^2 \| \mathbf{a}_1 \|_{H^{-\theta}}^2 d\mu^\epsilon(w)
\\
&\lesssim T^{(\theta-\theta_0)/\Gamma} \|\mathbf{a}_1\|_{H^{-\theta}}^2 \int_H \| w \|_{H^{\theta_0}}^2 d\mu^\epsilon(w)
\\
&\quad+
e^{-2\lambda_0T} \| \mathbf{a}_1 \|_{H^{-\theta}}^2 \int_H \|  w \|_{H^\theta}^2  d\mu^\epsilon(w) \to 0.
\end{align*}
\end{proof}

\begin{lem} \label{lem:time_integral_theta2}
Let $\psi \in \mathcal{E}_\theta$, $\theta \in [0,\theta_0)$, be given by $\psi(y) = a_2(y,y)$. Then for every $T>0$ we have $\psi(e^{C_\epsilon T}\cdot) \in \mathcal{E}$, $\int_{1/T}^T \psi(e^{C_\epsilon t}\cdot) dt \in \mathcal{E}$, and there exists $\phi \in \mathcal{E}_{\theta-\delta} \cap D(\mathscr{L}_y^\epsilon)$ for every $\delta\in(0,\gamma)$ satisfying $\| \phi \|_{\mathcal{E}_{\theta-\delta}} \lesssim \|\psi\|_{\mathcal{E}_\theta}$, such that
\begin{align*}
\lim_{T \to \infty}
\int_{1/T}^T \psi(e^{C_\epsilon t}\cdot) dt
=
\phi,
\end{align*}
the limit being understood with respect to the $D(\mathscr{L}_y^\epsilon)$ topology. 
\end{lem}

\begin{proof}
First of all, there exists a linear bounded operator $A_2:H^\theta \to H^{-\theta}$ such that $a_2(y,v)=\langle y, A_2 v \rangle = \langle  A_2 y,v \rangle$ for every $y,v \in H^\theta$, and $\|\psi\|_{H^\theta} = \|A_2\|_{H^\theta \to H^{-\theta}}$. 
Then, for every $T>0$ we have $\psi(e^{C_\epsilon T}\cdot) \in \mathcal{E}$ since
$|\psi(e^{C_\epsilon T}y)| \leq \|\psi\|_{H^\theta} \|e^{C_\epsilon T}y\|^2_{H^\theta} \lesssim \|\psi\|_{H^\theta} e^{-2\lambda_0 T} T^{-\theta/\gamma} \|y\|_H^2$, and similarly $\int_{1/T}^T \psi(e^{C_\epsilon t}\cdot) dt \in \mathcal{E}$.
In addition, for every $T>0$ and $y \in H^\theta$
\begin{align*}
\int_{1/T}^T \psi(e^{C_\epsilon t}y) dt
&=
\int_{1/T}^T \langle e^{C_\epsilon t}y, A_2 e^{C_\epsilon t}y \rangle  dt
=
\langle y , \left(\int_{1/T}^T e^{C_\epsilon t} A_2 e^{C_\epsilon t} dt \right) y \rangle,
\end{align*} 
and since for every $\delta_1,\delta_2 \geq 0$ satisfying $\delta_1+\delta_2=\delta<\gamma$ it holds
\begin{align*}
\int_0^\infty 
\left\| e^{C_\epsilon t} A_2 e^{C_\epsilon t} \right\|_{H^{\theta-2\delta_1} \to H^{2\delta_2-\theta}} dt 
&\lesssim  
\|A_2\|_{H^{\theta} \to H^{-\theta}} \int_0^\infty \frac{e^{-\lambda_0 t}}{t^{\delta/\gamma}} dt  
< \infty,
\end{align*}
there exists a linear bounded operator $A_2^\infty:H^{\theta-2\delta_1} \to H^{2\delta_2-\theta}$ such that
\begin{align*}
\int_{1/T}^T e^{C_\epsilon t} A_2 e^{C_\epsilon t} dt \to A_2^\infty
\end{align*}
strongly as $T \to \infty$ for every $\delta_1+\delta_2=\delta<\gamma$, and $\langle y, A_2^\infty v \rangle = \langle  A_2^\infty y,v \rangle$ for every $y,v \in H^{\theta-\delta}$.
In particular, using $\delta_1=\delta_2=\delta/2$, we can define $\phi \in \mathcal{E}_{\theta-\delta}$ given by 
\begin{align*}
\phi(y) = \langle y , A_2^\infty y \rangle, 
\quad
y \in H^{\theta-\delta},
\end{align*}
which of course satisfies $\|\phi\|_{\mathcal{E}_{\theta-\delta}} = \|A_2^\infty\|_{H^{\theta-\delta} \to H^{\delta-\theta}} \lesssim \|A_2\|_{H^{\theta} \to H^{-\theta}}= \|\psi\|_{\mathcal{E}_\theta}$.
Let us now check $\phi \in D(\mathscr{L}_y^\epsilon)$: we have for every $y \in H^{\theta}$
\begin{align*}
\langle C_\epsilon y , D_y \phi (y) \rangle 
&=
2
\langle C_\epsilon y , A_2^\infty y \rangle
=
-\langle y , A_2 y \rangle,
\end{align*}
where the last equality comes from an integration by parts.
Also, given a complete orthonormal system $\{e_k\}_{k \in\N}$ of $H$ and choosing $\delta \in (0,\gamma)$ such that $\theta-\delta \leq \theta_0-\gamma$, by (Q2) it holds
\begin{align*}
\frac12 Tr(QD^2_y \phi) 
&= 
\sum_{k \in \N} \langle Q^{1/2} e_k , A_2^\infty Q^{1/2}e_k \rangle
\\
&\leq
\|A_2^\infty\|_{H^{\theta-\delta} \to H^{\delta-\theta}}\sum_{k \in \N} \|Q^{1/2} e_k\|^2_{H^{\theta-\delta}}
<
\infty.
\end{align*}
Putting all together,
\begin{align*}
\|\phi\|_{D(\mathscr{L}_y)}^2 
&\lesssim
\int_H |\phi(w)|^2 d\mu^\epsilon(w)
+
\int_H |\langle w , A_2 w \rangle|^2 d\mu^\epsilon(w)
+
Tr(QD^2_y \phi)^2
\\
&\lesssim
\|A_2^\infty\|_{H^{\theta-\delta} \to H^{\delta-\theta}}^2
\left(
\int_H \|w\|^2_{H^{\theta-\delta}} d\mu^\epsilon(w) \right)^2
\\
&\quad+
\|A_2\|_{H^\theta \to H^{-\theta}}^2
\left( 
\int_H \|w\|^2_{H^\theta} d\mu^\epsilon(w)
\right)^2
\\
&\quad+
Tr(QD^2_y \phi)^2
< \infty.
\end{align*}

Let us finally prove $\lim_{T \to \infty} \int_{1/T}^T \psi(e^{C_\epsilon t}\cdot) dt = \phi$ in the $D(\mathscr{L}_y^\epsilon)$ topology. To ease the notation, denote $A_2^T = \int_{1/T}^T e^{C_\epsilon t} A_2 e^{C_\epsilon t} dt$, so that $\int_{1/T}^T \psi (e^{C_\epsilon t}y) dt = \langle y, A_2^T y \rangle$ for every $y \in H$.
First, we have
\begin{align*}
\left\|\int_{1/T}^T \psi (e^{C_\epsilon t}\cdot) dt - \phi\right\|_{L^2(H,\mu^\epsilon)}^2
&=
\int_H |\langle w, (A_2^T-A_2^\infty) w \rangle|^2 d\mu^\epsilon(w)
\\
&\lesssim
\|A_2^T-A_2^\infty\|_{H^{\theta-\delta} \to H^{\delta-\theta}}^2 \left(  \int_H \|w\|_{H^{\theta-\delta}}^2 d\mu^\epsilon(w) \right)^2 \to 0,
\end{align*}
as $T \to \infty$, since $A_2^T \to A_2^\infty$ strongly. 
Second, the following identities hold true:
\begin{align*}
\mathscr{L}_y^\epsilon \left( \int_{1/T}^T \psi(e^{C_\epsilon t}\cdot) dt\right)(y)
&=
2 \langle C_\epsilon y, A_2^T y \rangle
+
Tr(Q A_2^T)
\\
&=
\langle y, (e^{C_\epsilon T} A_2 e^{C_\epsilon T} - e^{C_\epsilon /T} A_2 e^{C_\epsilon /T}) y \rangle 
+
Tr(Q A_2^T);
\\
\mathscr{L}_y^\epsilon \phi(y)
&=
- \langle y, A_2 y \rangle
+
Tr(Q A_2^\infty),
\end{align*}
from which we get
\begin{align*}
&\left\| \mathscr{L}_y^\epsilon \left( \int_{1/T}^T \psi(e^{C_\epsilon t}\cdot) dt\right)- \mathscr{L}^\epsilon_y \phi \right\|_{L^2(H,\mu^\epsilon)}^2
\\
&\quad
\lesssim
\int_H
|\langle w, (e^{C_\epsilon T} A_2 e^{C_\epsilon T} + A_2- e^{C_\epsilon/T} A_2 e^{C_\epsilon/T}) w \rangle|^2 d\mu^\epsilon(w)
\\
&\quad\quad
+
|Tr(Q A_2^T)- Tr(Q A_2^\infty)|^2
\\
&\quad\lesssim \int_H
|\langle w, e^{C_\epsilon T} A_2 e^{C_\epsilon T} w \rangle|^2 d\mu^\epsilon(w)
+
\int_H
|\langle w, A_2(1-e^{C_\epsilon/T}) w \rangle|^2 d\mu^\epsilon(w)
\\
&\quad\quad
+
\int_H
|\langle w, (1-e^{C_\epsilon/T}) A_2 e^{C_\epsilon/T} w \rangle|^2 d\mu^\epsilon(w)
\\
&\quad\quad
+
\left|\sum_{k \in \N}\langle Q^{1/2} e_k , (A_2^T-A_2^\infty) Q^{1/2}e_k \rangle \right|^2
\\
&\lesssim e^{-\lambda_0 T} \|A_2\|_{H^\theta \to H^{-\theta}}
\int_H  \|w\|^2_{H^\theta} d\mu^\epsilon(w)
\\
&\quad
+
T^{(\theta-\theta_0)/2\Gamma} \|A_2\|_{H^\theta \to H^{-\theta}} \left(1+e^{-\lambda_0 T} \right)
\int_H \|w\|_{H^\theta} \|w\|_{H^{\theta_0}} d\mu^\epsilon(w)
\\
&\quad
+
\|A_2^T-A_2^\infty\|^2_{H^{\theta-\delta} \to H^{\delta-\theta}}
\left|\sum_{k \in \N}\|Q^{1/2} e_k\|^2_{H^{\theta-\delta}}\right|^2
\to 0.
\end{align*}
\end{proof}

\begin{cor} \label{cor:regularity_phi}
Under the hypotheses \autoref{lem:time_integral_theta2}, let $\phi=\lim_{T \to \infty} \int_{1/T}^T \psi(e^{C_\epsilon t}\cdot) dt$. Then for every $\delta_1,\delta_2 \geq 0$, $\delta_1+\delta_2<\gamma$ it holds $\langle D_y \phi(\cdot) , v \rangle \in \mathcal{E}_{\theta-2\delta_1}$ for every $v \in H^{\theta-2\delta_2}$, with $\|\langle D_y \phi(\cdot) , v \rangle\|_{\mathcal{E}_{\theta-2\delta_1}} \lesssim \|\psi\|_{\mathcal{E}_\theta} \|v\|_{H^{\theta-2\delta_2}}$. 
\end{cor}

\begin{proof}
It is sufficient to recall the expression $\phi(y)=\langle y, A^\infty_2 y \rangle$, valid for $y \in \mathscr{S}$, and notice that for every $v \in H^{\theta-2\delta_2}$ it holds $\langle v, D_y \phi(y) \rangle = 2 \langle v , A_2^\infty y \rangle$.
To conclude, notice that $|\langle v, D_y \phi(y) \rangle|\lesssim \|y\|_{H^{\theta-2\delta_1}} \|\psi\|_{\mathcal{E}_\theta} \|v\|_{H^{\theta-2\delta_2}}$ because $A^\infty_2 : H^{\theta-2\delta_1} \to H^{2\delta_2-\theta}$ is continuous with $\|A^\infty_2\|_{H^{\theta-2\delta_1} \to H^{2\delta_2-\theta}} \lesssim \|\psi\|_{\mathcal{E}_\theta}$, and therefore the identity $\langle v, D_y \phi(y) \rangle = 2 \langle v , A_2^\infty y \rangle$ extends to every $y \in H^{\theta-2\delta_1}$. 
\end{proof}

The next proposition permits to solve the Poisson equation $\mathscr{L}_y^\epsilon \phi = -\psi$ in the unknown $\phi$, under suitable assumptions on the datum $\psi$. We need, in particular, $\psi$ to be a quadratic function on some Sobolev space $H^\theta$, $\theta \in [0,\theta_0)$, with zero average with respect to the invariant measure $\mu^\epsilon$, namely $\int_H \psi(w) d\mu^\epsilon(w) = 0$. This latter condition being necessary is clear from invariance of $\mu^\epsilon$ under the Ornstein-Uhlenbeck semigroup $P_t^\epsilon$ and 
\begin{align*}
\left|\int_H \mathscr{L}_y^\epsilon \phi(w) d\mu^\epsilon(w)\right|
&\leq
\left|\int_H \left( \mathscr{L}_y^\epsilon \phi(w) - \frac{1}{t} \left(P_t^\epsilon \phi - \phi\right)(w) \right) d\mu^\epsilon(w)\right|
\to 0
\end{align*}
as $t \to 0^+$; by the proposition, the zero-average condition on $\psi$ is also sufficient, at least when we restrict ourselves to $\psi \in \mathcal{E}_\theta$.
Finally, notice that the solution of the Poisson equation is more regular than the datum, namely: if $\psi \in \mathcal{E}_\theta$, then $\phi \in \mathcal{E}_{\theta-\delta}$ for every $\delta\in(0,\gamma)$, and $\langle D_y \phi(\cdot) , v \rangle \in \mathcal{E}_{\theta-2\delta_1}$ for every $v \in H^{\theta-2\delta_2}$, $\delta_1,\delta_2 \geq 0$, $\delta_1+\delta_2<\gamma$.

\begin{prop} \label{prop:inverse_theta}
Let $\psi \in \mathcal{E}_\theta$, $\theta \in [0,\theta_0)$ be such that $\int_H \psi(w) d\mu^\epsilon(w) = 0$. 
Then there exists $\phi \in \mathcal{E}_{\theta-\delta} \cap D(\mathscr{L}_y^\epsilon)$ for every $\delta\in(0,\gamma)$ satisfying $\| \phi \|_{\mathcal{E}_{\theta-\delta}} \lesssim \|\psi\|_{\mathcal{E}_\theta}$, such that
\begin{align*}
\phi = \lim_{T \to \infty} \int_{1/T}^T P_t^\epsilon \psi dt
\end{align*}
with respect to topology of $D(\mathscr{L}_y^\epsilon)$, and $\mathscr{L}_y^\epsilon \phi = -\psi$. 
Moreover, $\langle D_y \phi(\cdot) , v \rangle \in \mathcal{E}_{\theta-2\delta_1}$ for every $v \in H^{\theta-2\delta_2}$, $\delta_1,\delta_2 \geq 0$, $\delta_1+\delta_2<\gamma$, with $\|\langle D_y \phi(\cdot) , v \rangle\|_{\mathcal{E}_{\theta-2\delta_1}} \lesssim \|\psi\|_{\mathcal{E}_\theta} \|v\|_{H^{\theta-2\delta_2}}$.
\end{prop}

\begin{proof}
First we prove that the limit exists. Let $\psi(y) = a_0 + a_1(y) + a_2(y,y)$. We have
\begin{align*}
P_t \psi (y) = P_t \psi (0) + a_1(e^{C_\epsilon t}y) + a_2(e^{C_\epsilon t}y,e^{C_\epsilon t}y),
\end{align*}
and by \autoref{lem:time_integral_theta1} and \autoref{lem:time_integral_theta2}, the quantity $\int_{1/T}^T \left( a_1(e^{C_\epsilon t}y) + a_2(e^{C_\epsilon t}y,e^{C_\epsilon t}y) \right)dt$ converges with respect to the $D(\mathscr{L}_y^\epsilon)$  topology to some $\phi_\star \in \mathcal{E}_{\theta-\delta} \cap D(\mathscr{L}_y^\epsilon)$ for every $\delta\in(0,\gamma)$. 
Moreover, by \autoref{lem:exp_mixing}
\begin{align*}
|P_t^\epsilon \psi (0)| = |P_t^\epsilon \psi (0)-\int_H \psi(w) d\mu^\epsilon(w)| \lesssim
\|\psi\|_{\mathcal{E}_\theta} e^{-\lambda_0 t}
\end{align*}
is integrable with respect to time, and so it converges with respect to the $D(\mathscr{L}_y^\epsilon)$ topology to a constant $\phi_0$. Putting all together,
\begin{align*}
\lim_{T \to \infty} \int_{1/T}^T P_t^\epsilon \psi dt = 
\phi_0 + \phi_\star \eqqcolon \phi \in \mathcal{E}_{\theta-\delta} \cap D(\mathscr{L}_y^\epsilon).
\end{align*}

Let us show that $\phi$ is indeed a solution of the Poisson equation $\mathscr{L}_y^\epsilon \phi = -\psi$.
Notice that $\mathscr{L}_y^\epsilon:D(\mathscr{L}_y^\epsilon) \to L^2(H,\mu^\epsilon)$ is bounded, and therefore by continuity we have 
\begin{align*}
\mathscr{L}_y^\epsilon\phi 
=
\mathscr{L}_y^\epsilon \left( \lim_{T \to \infty}\int_{1/T}^T P_t^\epsilon \psi dt \right) 
= 
\lim_{T \to \infty} \mathscr{L}_y^\epsilon \left( \int_{1/T}^T P_t^\epsilon \psi dt \right),
\end{align*}
where the first limit is understood with respect to the $D(\mathscr{L}_y^\epsilon)$ topology, and the second one with respect to the $L^2(H,\mu^\epsilon)$ topology. 
Since $\int_{1/T}^T \|P_t^\epsilon \psi\|_\mathcal{E} dt < \infty$ for every $T>0$ we have $\int_{1/T}^T P_t^\epsilon \psi dt \in \mathcal{E} \subset D(\mathscr{L}_y^\epsilon)$ for every $T>0$, and by \autoref{lem:commuting} we have
\begin{align*}
\mathscr{L}_y^\epsilon \left( \int_{1/T}^T P_t^\epsilon \psi dt \right)
&=
\int_{1/T}^T \mathscr{L}_y^\epsilon  P_t^\epsilon \psi dt
=
\int_{1/T}^T P_t^\epsilon  \mathscr{L}_y  \psi dt
=
P_T^\epsilon \psi - P_{1/T}^\epsilon \psi.
\end{align*}
In particular,
\begin{align*}
\mathscr{L}_y^\epsilon \phi
=
\lim_{T \to \infty} \left( P_T^\epsilon \psi - P_{1/T}^\epsilon \psi \right).
\end{align*}
Since we already know that $P_{1/T}^\epsilon \psi \to \psi \in L^2(H,\mu^\epsilon)$ as $T \to \infty$ by continuity of the semigroup, we are left to check $\lim_{T \to \infty} P_T^\epsilon \psi = 0 \in L^2(H,\mu^\epsilon)$. We have:
\begin{align*}
|P_T^\epsilon \psi (y)| 
&\leq
|P_T^\epsilon \psi (0)| + |a_1(e^{C_\epsilon T}y)| + |a_2(e^{C_\epsilon T}y,e^{C_\epsilon T}y)|
\\
&\leq
|P_T^\epsilon \psi (0)| +
\|\psi\|_{\mathcal{E}_\theta} (\|e^{C_\epsilon T}y\|_{H^\theta}+\|e^{C_\epsilon T}y\|_{H^\theta}^2) 
\\
&\lesssim 
\|\psi\|_{\mathcal{E}_\theta} e^{-\lambda_0 T}
+
\|\psi\|_{\mathcal{E}_\theta} (e^{-\lambda_0 T}\|y\|_{H^\theta}+e^{-2\lambda_0 T}\|y\|_{H^\theta}^2)
\to 0
\end{align*}
as $T \to \infty$ in $L^2(H,\mu^\epsilon)$. Finally, the assertion about the derivative $D_y \phi$ follows by the explicit construction of \autoref{lem:time_integral_theta1} and \autoref{cor:regularity_phi}.
\end{proof}

\section{Perturbed test function method} \label{sec:test}

Let us move back to the problem of identifying $\varphi_1^\epsilon$, $\varphi_2^\epsilon$ in the expression of the test function $\varphi^\epsilon$. 
Recall we are looking for a perturbation of $\varphi = \varphi(u)$ of the following form
\begin{align*}
\varphi^\epsilon(u,y) = \varphi(u) + \epsilon^{1/2} \varphi_1^\epsilon(u,y) + \epsilon \varphi_2^\epsilon(u,Y).
\end{align*}
For our purposes it is sufficient to consider $\varphi \in F$, namely $\varphi(u) = \langle u ,h \rangle$ for some given smooth test function $h \in \mathscr{S}$.
With this choice of $\varphi$, we have in particular
\begin{align*}
D_u \varphi = h,
\quad
D^2_u \varphi = 0.
\end{align*}

\subsection{Finding $\varphi_1$}
Recalling \eqref{1.4bis}, the first corrector $\varphi_1^\epsilon$ needs to solve the Poisson equation
\begin{align*}
\mathscr{L}_y^\epsilon \varphi_1^\epsilon(u,y) = - \langle b(y,u),h \rangle.
\end{align*}
For every fixed $u \in H$, we can apply \autoref{prop:inverse_theta} to the function $\psi_u = \langle b(\cdot,u) , h \rangle \in \mathcal{E}$. Indeed $\int_H \psi_u(w) d\mu^\epsilon(w)=0$, therefore there exists $\phi_u^\epsilon \in \mathcal{E}$ such that $\mathscr{L}_y^\epsilon \phi_u^\epsilon = -\psi_u$. 
Moreover, since $\psi_u$ is linear in $y$, following the construction of \autoref{lem:time_integral_theta1} it is easy to check
\begin{align*}
\phi_u^\epsilon 
=
\langle b((-C_\epsilon)^{-1} \cdot, u ) , h \rangle.
\end{align*}
Finally, we define:
\begin{align} \label{eq:varphi1}
\varphi_1^\epsilon(u,y)=\phi_u^\epsilon(y) &= \langle b((-C_\epsilon)^{-1}y,u),h\rangle.
\end{align}
Notice that for every $v \in \mathscr{S}$:
\begin{align*}
\langle D_y \varphi_1^\epsilon(u) , v \rangle &= \langle b((-C_\epsilon)^{-1}v,u),h\rangle
\\ \nonumber
\langle D_u \varphi_1^\epsilon(y) , v \rangle &= \langle b((-C_\epsilon)^{-1}y,v),h\rangle.
\end{align*}

\begin{prop} \label{prop:regularity_varphi1}
For every $u,y \in H^s$, $s\in \R$, we have $D_y \varphi_1^\epsilon(u),D_u \varphi_1^\epsilon(y) \in H^{2\theta+s}$ for every $\theta \in [\gamma,1]$, with
\begin{align*}
\|D_y \varphi_1^\epsilon(u)\|_{H^{2\theta+s}} 
\lesssim 
\epsilon^{-\frac{\theta-\gamma}{1-\gamma}}\|h\|_{H^{\theta_1}} \|u\|_{H^s},
\quad
\|D_u \varphi_1^\epsilon(y)\|_{H^{2\theta+s}} \lesssim \epsilon^{-\frac{\theta-\gamma}{1-\gamma}}\|h\|_{H^{\theta_1}} \|y\|_{H^s},
\end{align*}
for some $\theta_1=\theta_1(s)$ sufficiently large.
\end{prop}
\begin{proof}
Take $\theta_1=\theta_1(s)$ such that (B2) holds true.
For every $v \in H^{2\gamma+s}$ we have
\begin{align*}
|\langle D_y \varphi_1^\epsilon(u), (-A)^{\gamma+s/2} v \rangle|
&=
|\langle b((-C_\epsilon)^{-1}(-A)^{\gamma+s/2}v,u),h \rangle|
\\
&\lesssim
\|(-C_\epsilon)^{-1}(-A)^{\gamma+s/2}v\|_{H^{-s}}
\|h\|_{H^{\theta_1}} \|u\|_{H^s} 
\\
&\lesssim
\|(-C)^{-1}(-A)^{\gamma+s/2}v\|_{H^{-s}}
\|h\|_{H^{\theta_1}} \|u\|_{H^s}
\\
&\lesssim 
\|v\|_H \|h\|_{H^{\theta_1}} \|u\|_{H^s} ,
\end{align*}
and similarly for every $v \in H^{2+s}$
\begin{align*}
|\langle D_y \varphi_1^\epsilon(u), (-A)^{1+s/2} v \rangle|
&=
|\langle b((-C_\epsilon)^{-1}(-A)^{1+s/2}v,u),h \rangle|
\\
&\lesssim
\|(-C_\epsilon)^{-1}(-A)^{1+s/2}v\|_{H^{-s}}
\|h\|_{H^{\theta_1}} \|u\|_{H^s} 
\\
&\lesssim
\|(-\epsilon A)^{-1}(-A)^{1+s/2}v\|_{H^{-s}}
\|h\|_{H^{\theta_1}} \|u\|_{H^s} 
\\
&\lesssim 
\epsilon^{-1}
\|v\|_H \|h\|_{H^{\theta_1}} \|u\|_{H^s}.
\end{align*} 
Since $v$ is arbitrary, by interpolation we deduce
\begin{align*}
\|D_y \varphi_1^\epsilon(u)\|_{H^{2\theta+s}} 
\lesssim
\|D_y \varphi_1^\epsilon(u)\|_{H^{2\gamma+s}}^{\frac{1-\theta}{1-\gamma}}
\|D_y \varphi_1^\epsilon(u)\|_{H^{2+s}}^{\frac{\theta-\gamma}{1-\gamma}}
\lesssim
\epsilon^{-\frac{\theta-\gamma}{1-\gamma}}\|h\|_{H^{\theta_1}} \|u\|_{H^s}.
\end{align*}
The argument is similar for the term involving $D_u \varphi_1^\epsilon$: first, for every $v \in H^{2\gamma+s}$
\begin{align*}
|\langle D_u \varphi_1^\epsilon(y), (-A)^{\gamma+s/2}v \rangle|
&\lesssim
|\langle b((-C_\epsilon)^{-1}y,(-A)^{\gamma+s/2}v),h \rangle|
\\
&\lesssim
\|(-C_\epsilon)^{-1}y\|_{H^{2\gamma+s}} \|h\|_{H^{\theta_1}} \|(-A)^{\gamma+s/2}v \|_{H^{-2\gamma-s}}
\\
&\lesssim
\|(-C)^{-1}y\|_{H^{2\gamma+s}} \|h\|_{H^{\theta_1}} \|(-A)^{\gamma+s/2}v \|_{H^{-2\gamma-s}}
\\
&\lesssim 
\|y\|_{H^s} \|h\|_{H^{\theta_1}} \|v\|_H,
\end{align*}
whereas for every $v \in H^{2+s}$
\begin{align*}
|\langle D_u \varphi_1^\epsilon(y), (-A)^{1+s/2}v \rangle|
&\lesssim
|\langle b((-C_\epsilon)^{-1}y,(-A)^{1+s/2}v),h \rangle|
\\
&\lesssim
\|(-C_\epsilon)^{-1}y\|_{H^{2+s}} \|h\|_{H^{\theta_1}} \|(-A)^{1+s/2}v \|_{H^{-2-s}}
\\
&\lesssim
\|(-\epsilon A)^{-1}y\|_{H^{2+s}} \|h\|_{H^{\theta_1}} \|(-A)^{1+s/2}v \|_{H^{-2-s}}
\\
&\lesssim 
\epsilon^{-1}
\|y\|_{H^s} \|h\|_{H^{\theta_1}} \|v\|_H.
\end{align*}
The thesis follows by interpolation.
\end{proof}

\subsection{Finding $\varphi_2$}
Let us move to equation \eqref{1.4ter} for the second corrector:
\begin{align*}
\langle Au + b(u,u), h \rangle 
+ 
\langle b(y,u), D_u \varphi_1^\epsilon(y) \rangle
+
\langle b(y,y), D_y \varphi_1^\epsilon(u) \rangle
+
\mathscr{L}_y^\epsilon \varphi_2^\epsilon (u,Y).
\end{align*}
As discussed in \autoref{ssec:perturbed_test_function}, as a preliminary step to our analysis it is useful to manipulate the previous expression to replace the small-scale process $y$ with its linearised counterpart $Y$. We point out that at this level of generality the variables $y$ and $Y$ only represent variables in the Hilbert space $H$, but we prefer to keep the notational difference for the sake of clarity, since we intend to eventually evaluate the previous expressions at $y=y^\epsilon$, $Y=Y^\epsilon$.

Having said this, let $\zeta \in H$ indicate the difference $\zeta=y-Y$. We can prove the following:
\begin{prop} \label{prop:linearisation}
For every $\delta \in (0, \gamma-1/4)$, $u \in H^{1-\delta}$, $y \in H^\gamma$ and $Y \in H^{\theta_0-\gamma}$ it holds:
\begin{align*}
|\langle b(y,u), D_u \varphi_1^\epsilon(y) \rangle - \langle b(Y,u), D_u \varphi_1^\epsilon(Y) \rangle|
&\lesssim
\|\zeta\|_H \|h\|_{H^{\theta_1}} \|Y\|_{H^{\theta_0-\gamma}} \|u\|_H
\\
&\quad+ 
\epsilon^{-\frac{1-\gamma-\delta}{2(1-\gamma)}} 
\|\zeta\|_{H^\gamma}\|h\|_{H^{\theta_1}} \|y\|_{H}\|u\|_{H^{1-\delta}},
\\
|\langle b(y,y) , D_y \varphi_1^\epsilon(u) \rangle-\langle b(Y,Y) , D_y \varphi_1^\epsilon(u) \rangle|
&\lesssim
\|\zeta\|_H \|h\|_{H^{\theta_1}} \|Y\|_{H^{\theta_0-\gamma}} \|u\|_H
\\
&\quad+ 
\epsilon^{-\frac{1-\gamma-\delta}{2(1-\gamma)}} 
\|\zeta\|_{H^\gamma}\|h\|_{H^{\theta_1}} \|y\|_{H}\|u\|_{H^{1-\delta}}. 
\end{align*}
\end{prop}

\begin{proof}
Recall
\begin{align*}
\langle b(y,u) , D_u \varphi_1^\epsilon(y) \rangle
&-
\langle b(Y,u) , D_u \varphi_1^\epsilon(Y) \rangle
=
\langle b(\zeta,u) , D_u \varphi_1^\epsilon(Y) \rangle
+
\langle b(y,u) , D_u \varphi_1^\epsilon(\zeta) \rangle,
\\
\langle b(y,y) , D_y \varphi_1^\epsilon(u) \rangle
&-
\langle b(Y,Y) , D_y \varphi_1^\epsilon(u) \rangle
=
\langle b(\zeta,Y) , D_y \varphi_1^\epsilon(u) \rangle
+
\langle b(y,\zeta) , D_y \varphi_1^\epsilon(u) \rangle,
\end{align*} 
and by (B3) and \autoref{prop:regularity_varphi1} with $\theta = \frac{1+\gamma-\delta}{2}$ the following estimates hold:
\begin{align*}
|\langle b(\zeta,u), D_u \varphi_1^\epsilon(Y) \rangle|
&\lesssim
\|\zeta\|_H \|h\|_{H^{\theta_1}} \|Y\|_{H^{\theta_0-\gamma}} \|u\|_H,
\\
|\langle b(y,u) , D_u \varphi_1^\epsilon(\zeta) \rangle|
&\lesssim
\epsilon^{-\frac{1-\gamma-\delta}{2(1-\gamma)}} 
\|\zeta\|_{H^\gamma}\|h\|_{H^{\theta_1}} \|y\|_{H}\|u\|_{H^{1-\delta}},
\\
|\langle b(\zeta,Y) , D_y \varphi_1^\epsilon(u) \rangle|
&\lesssim
\|\zeta\|_H \|h\|_{H^{\theta_1}}\|Y\|_{H^{\theta_0-\gamma}}   \|u\|_H,
\\
|\langle b(y,\zeta) , D_y \varphi_1^\epsilon(u) \rangle|
&\lesssim
\epsilon^{-\frac{1-\gamma-\delta}{2(1-\gamma)}}
\|\zeta\|_{H^\gamma} \|h\|_{H^{\theta_1}} \|y\|_H \|u\|_{H^{1-\delta}} ,
\end{align*}
where we have used $b:H \times H^{2\gamma+1-\delta} \to H^{\delta-1}$ and $b:H \times H^{2+\gamma-2\delta} \to H^{-\gamma}$ continuous by (B3) and our choice of $\delta$.
\end{proof}

\begin{rmk}
The previous proposition will be used in \autoref{sec:conv} to check rigorously that we can actually replace the small-scale process $y^\epsilon$ with $Y^\epsilon$, up to a correction which is infinitesimal as $\epsilon$. Indeed, since $\frac{1-\gamma-\delta}{2(1-\gamma)} < \frac12$ we can compensate diverging factors in $\epsilon$ in the previous proposition with a factor $\epsilon^{1/2}$ coming from 
\autoref{prop:y-Y}, having taken expectation and time integral.
\end{rmk}

Thus, in view of the discussion above and as already described in \autoref{ssec:perturbed_test_function}, the terms of order one in the expression of $\mathscr{L}^\epsilon \varphi$ are actually given by
\begin{align} \label{eq:order1}
\langle Au + b(u,u), h \rangle 
+ 
\langle b(Y,u), D_u \varphi_1^\epsilon(Y) \rangle
+
\langle b(Y,Y), D_y \varphi_1^\epsilon(u) \rangle
+
\mathscr{L}_y^\epsilon \varphi_2^\epsilon (u,Y),
\end{align}
and thus our goal is to find $\varphi_2^\epsilon = \varphi_2^\epsilon (u,Y)$ such that the previous quantity is independent of $Y$.

Let then $u \in H$ be fixed. The idea is again to apply \autoref{prop:inverse_theta} to 
\begin{align*}
\psi_u ^\epsilon
= \langle b(\cdot,u), D_u \varphi_1^\epsilon(\cdot) \rangle + \langle b(\cdot,\cdot), D_y \varphi_1^\epsilon(u) \rangle.
\end{align*}
For every $\theta > \frac{5}{4} - \gamma$ it holds $\psi_u^\epsilon \in \mathcal{E}_{\theta}$, with $\|\psi_u^\epsilon\|_{\mathcal{E}_{\theta}} \lesssim \|h\|_{H^{\theta_1}} \|u\|_{H}$.
However, $\psi_u^\epsilon$ does not satisfy the hypotheses of that proposition: indeed, it has not necessarily zero average with respect to the invariant measure $\mu^\epsilon$.
To deal with this issue, let us consider instead
\begin{align*}
\Psi_u^\epsilon = \psi_u^\epsilon - \int_H \psi_u^\epsilon (w) d\mu^\epsilon(w).
\end{align*}
With this choice of $\Psi_u^\epsilon$ we have $\Psi_u^\epsilon \in \mathcal{E}_{\theta}$ and $\int_H \Psi_u^\epsilon(w) d\mu^\epsilon(w)=0$, thus \autoref{prop:inverse_theta} applies.
Given $u \in H$ and $\Phi_u^\epsilon \in \mathcal{E}_{\theta'} \cap D(\mathscr{L}^\epsilon_y)$, $\theta'> \frac54 -2\gamma$ such that $\mathscr{L}_y^\epsilon \Phi_u^\epsilon = -\Psi_u^\epsilon$, we finally define 
\begin{align} \label{eq:varphi2}
\varphi_2^\epsilon(u,Y)=\Phi_u^\epsilon(Y),
\qquad
\|\varphi_2^\epsilon(u,\cdot)\|_{\mathcal{E}_{\theta'}} 
\lesssim 
\|\Psi_u^\epsilon\|_{\mathcal{E}_{\theta}} 
\lesssim \|h\|_{H^{\theta_1}} \|u\|_H.
\end{align}

With this choice of $\varphi_2^\epsilon$, \eqref{eq:order1} can be rewritten as
\begin{align} \label{eq:L0}
\langle Au + b(u,u), h \rangle 
+
\int_H \psi_u^\epsilon (w) d\mu^\epsilon(w)
\eqqcolon
\mathscr{L}^{0,\epsilon} \varphi(u),
\end{align}
which is indeed a function of the sole variable $u$.

In the following, specifically when computing $\mathscr{L}^\epsilon \varphi^\epsilon = \mathscr{L}^\epsilon \left( \varphi + \epsilon^{1/2}\varphi_1^\epsilon + \epsilon \varphi_2^\epsilon \right)$, we will need control over the derivatives $D_u \varphi_2^\epsilon$, $D_Y \varphi_2^\epsilon$ to check that the corrections we impose on the test function $\varphi$ do not change the underlying dynamics in the limit $\epsilon \to 0$, i.e. $\mathscr{L}^\epsilon \varphi^\epsilon$ is close to $\mathscr{L}^{0,\epsilon} \varphi$ in a suitable sense. 
Control over $D_Y \varphi_2^\epsilon$ is already encoded in the statement of \autoref{prop:inverse_theta}, as a straightforward consequence of \autoref{cor:regularity_phi}: indeed, for every $\theta > \frac{5}{4}-\gamma$, $\delta_1,\delta_2 \geq 0$, $\delta_1 + \delta_2 < \gamma$, and $v \in H^{\theta-2\delta_2}$ it holds $\langle D_Y \varphi_2^\epsilon(u,\cdot) , v \rangle \in \mathcal{E}_{\theta-2\delta_1}$, with uniform-in-$\epsilon$ bound:
\begin{align} \label{eq:bound_DY_varphi2}
\|\langle D_Y \varphi_2^\epsilon(u,\cdot) , v \rangle\|_{\mathcal{E}_{\theta-2\delta_1}}
\lesssim
\|\Psi_u^\epsilon\|_{\mathcal{E}_\theta} \|v\|_{H^{\theta-2\delta_2}}
\lesssim
\|h\|_{H^{\theta_1}} \|u\|_H \|v\|_{H^{\theta-2\delta_2}}.
\end{align}

On the other hand, to control $D_u \varphi_2^\epsilon$ we need the following preliminary lemma:
\begin{lem}
For every $\theta < \theta_0+2\gamma-1$ there exists $\theta_1$ sufficiently large such that $\langle v , D_u \Psi_u^\epsilon(\cdot) \rangle \in \mathcal{E}_{\theta_0}$ for every $v \in H^{-\theta}$ with
\begin{align*}
\|\langle v , D_u \Psi_u^\epsilon(\cdot) \rangle\|_{\mathcal{E}_{\theta_0}}
&\lesssim 
\|h\|_{H^{\theta_1}} \|v\|_{H^{-\theta}}.
\end{align*}
\end{lem}

\begin{proof}
By (B4) and (B2), for every $\theta < \theta_0+2\gamma-1$ we have $b:H^{\theta_0} \times H^{\theta_0+2\gamma} \to H^{\theta}$ and $b:H^{\theta_0} \times H^{\theta_0} \to H^{\theta-2\gamma}$ continuous, hence by \autoref{prop:regularity_varphi1} it holds for every $v \in H^{-\theta}$
\begin{align*}
|\langle b(Y,v), D_u \varphi_1^\epsilon(Y) \rangle|
&\lesssim
\|b(Y,D_u \varphi_1^\epsilon(Y))\|_{H^\theta} \|v\|_{H^{-\theta}}
\lesssim
\|Y\|_{H^{\theta_0}}^2 \|h\|_{H^{\theta_1}} \|v\|_{H^{-\theta}},
\end{align*} 
and,
\begin{align*}
|\langle v , \langle b(Y,Y), D_u D_y \varphi_1^\epsilon(u) \rangle\rangle|
&=
|\langle b((-C_\epsilon)^{-1}b(Y,Y),v),h\rangle|
\\
&\lesssim
\|b(Y,Y)\|_{H^{\theta-2\gamma}} \|v\|_{H^{-\theta}} \|h\|_{\theta_1}
\\
&\lesssim
\|h\|_{\theta_1} \|Y\|_{H^{\theta_0}}^2 \|v\|_{H^{-\theta}}.
\end{align*} 
Thus, the desired result is true for if we replace $\Psi_u^\epsilon$ by $\psi_u^\epsilon$. To conclude the proof, just notice that by the same computation as above,
\begin{align*}
\int_H |\langle v , D_u \Psi_u^\epsilon(w)| d\mu^\epsilon(w)
\lesssim
\|h\|_{\theta_1} \|v\|_{H^{-\theta}} \int_H \|w\|_{H^{\theta_0}}^2 d\mu^\epsilon(w)
\lesssim
\|h\|_{\theta_1} \|v\|_{H^{-\theta}},
\end{align*}
since the integral is finite by (Q2).
\end{proof}

\begin{prop} \label{prop:regularity_varphi2}
For every $\theta <\theta_0+2\gamma-1$ and $v \in H^{-\theta}$ it holds $\langle v , D_u \varphi_2^\epsilon(\cdot) \rangle \in \mathcal{E}_{\theta_0-\delta}$ for every $\delta \in (0,\gamma)$ with
\begin{align*}
\|\langle v , D_u\varphi_2^\epsilon(\cdot) \rangle\|_{\mathcal{E}_{\theta_0-\delta}}
&\lesssim 
\|h\|_{H^{\theta_1}} \|v\|_{H^{-\theta}}.
\end{align*}
\end{prop}
\begin{proof}
Recalling 
\begin{align*}
\psi_u^\epsilon(Y) 
&= 
\langle b(Y,u),D_u \varphi_1^\epsilon(Y) \rangle
+
\langle b(Y,Y),D_y \varphi_1^\epsilon(u) \rangle,
\\
\Psi_u^\epsilon(Y) 
&= 
\psi_u^\epsilon(Y) - \int_H \psi_u^\epsilon(w) d\mu^\epsilon(w),
\end{align*}
we have for every $v \in \mathscr{S}$
\begin{align*} 
\langle D_u\psi_u^\epsilon(Y),v \rangle 
&=
\langle b(Y,v),D_u \varphi_1^\epsilon(Y) \rangle
+
\langle b(Y,Y),\langle D_uD_y \varphi_1^\epsilon,v \rangle \rangle,
\\
\langle D_u\Psi_u^\epsilon(Y),v \rangle 
&=
\langle D_u\psi_u^\epsilon(Y),v \rangle 
-
\int_H \langle D_u\psi_u^\epsilon(w),v \rangle d\mu(w),
\end{align*}
and the previous quantity is independent of $u \in H$.
Recall that we have defined $\varphi_2^\epsilon(u,\cdot) = \Phi_u^\epsilon = \lim_{T \to \infty} \int_{1/T}^T P_t^\epsilon \Psi_u^\epsilon dt $, where the limit is taken in $D(\mathscr{L}_y^\epsilon)$; we prove now that for every $v \in H$ we have
\begin{align*}
\langle D_u \varphi_2^\epsilon(\cdot) , v \rangle
&=
\lim_{T \to \infty} \int_{1/T}^T P_t^\epsilon \langle D_u\Psi_u^\epsilon,v \rangle dt.
\end{align*}
Denote $\varphi_2^{\epsilon, T}(u,\cdot)  = \int_{1/T}^T P^\epsilon_t \Psi_u^\epsilon dt \in \mathcal{E}$, and consider for $r \in \mathbb{R}$
\begin{align*}
\varphi_2^{\epsilon, T}(u+rv,\cdot)-\varphi_2^{\epsilon, T}(u,\cdot)
&=
\int_{1/T}^T P_t^\epsilon \Psi_{u+rv}^\epsilon dt - 
\int_{1/T}^T P_t^\epsilon \Psi_u^\epsilon dt 
\\
&=
\int_{1/T}^T P_t^\epsilon  \left(
 \Psi_{u+rv}^\epsilon - \Psi_u^\epsilon \right) dt
\\
&=
\int_{1/T}^T P_t^\epsilon  \left(
r\langle D_u \Psi_u^\epsilon , v \rangle \right) dt
\\
&=
r\int_{1/T}^T P_t^\epsilon 
\langle D_u \Psi_u^\epsilon , v \rangle dt,
\end{align*}
where we have used linearity of $P_t^\epsilon$ and the fact that $\Psi_u^\epsilon$ is linear in $u$.
The map $y \mapsto \langle D_u \Psi_u^\epsilon(y) , v \rangle$ satisfies the assumptions of \autoref{prop:inverse_theta}, and therefore we can take the limit in $D(\mathscr{L}_y^\epsilon)$ of the previous expression, as $T \to \infty$, to obtain
\begin{align*}
\varphi_2^\epsilon(u+rv,\cdot)-\varphi_2^\epsilon(u,\cdot)
&=
r \lim_{T \to \infty}\int_{1/T}^T P_t^\epsilon 
\langle D_u \Psi_u^\epsilon , v \rangle dt.
\end{align*}
Finally, rearranging and taking the limit as $r \to 0$ we get
\begin{align*}
\langle D_u \varphi_2^\epsilon(\cdot) , v \rangle
&=
\lim_{r \to 0} \frac1r \left(
\varphi_2^\epsilon(u+rv,\cdot)-\varphi_2^\epsilon(u,\cdot) \right)
=
\lim_{T \to \infty} \int_{1/T}^T P_t^\epsilon \langle D_u\Psi_u^\epsilon,v \rangle dt.
\end{align*}

Since it holds $\langle D_u \Psi_u^\epsilon(\cdot),v \rangle \in \mathcal{E}_{\theta_0}$ with $\|\langle v , D_u \Psi_u^\epsilon(\cdot) \rangle\|_{\mathcal{E}_{\theta_0}} \lesssim \|h\|_{H^{\theta_1}} \|v\|_{H^{-\theta}}$ for every $v \in H^{-\theta}$ and $\theta <\theta_0+2\gamma-1$, by \autoref{prop:inverse_theta} we have $\langle D_u \varphi_2^\epsilon(\cdot) , v  \rangle \in \mathcal{E}_{\theta_0-\delta}$ for every $\delta \in (0,\gamma)$ with 
\begin{align*}
\|\langle D_u \varphi_2^\epsilon(\cdot) , v  \rangle\|_{\mathcal{E}_{\theta_0-\delta}}
\lesssim
\|\langle v , D_u \Psi_u^\epsilon(\cdot) \rangle\|_{\mathcal{E}_{\theta_0}} \lesssim
\|h\|_{H^{\theta_1}} \|v\|_{H^{-\theta}}.
\end{align*}
Finally, the previous bound extends to all $v \in H^{-\theta}$ by continuity. The proof is complete. 
\end{proof}

\section{Convergence to transport noise} \label{sec:conv}
In this section we state and prove convergence of $u^\epsilon$. We first give the precise formulation of our result:
\begin{thm} \label{thm:main}
Assume (Q1)-(Q2) and $u_0 , y_0 \in H$ be given. 
Let $\{u^\epsilon\}_{\epsilon\in(0,1)}$ be a family of solutions to \eqref{eq:system} in the sense of \autoref{def:sol}, which exists on a family of stochastic basis $\{(\Omega^\epsilon, \mathcal{F}^\epsilon, \{\mathcal{F}^\epsilon_t \}_{t \geq 0}, \mathbb{P}^\epsilon, W^\epsilon)\}_{\epsilon\in(0,1)}$ by \autoref{prop:existence}. 
Then for every $\beta>0$, the laws of the processes $\{u^\epsilon\}_{\epsilon\in(0,1)}$ are tight as probability measures on the space $L^2([0,T],H) \cap C([0,T],H^{-\beta})$, and every weak accumulation point $(u,Q^{1/2} W)$ of $(u^\epsilon,Q^{1/2} W^\epsilon)$, $\epsilon \to 0$, is an analytically weak solution of the equation with transport noise and It\=o-Stokes drift velocity $r = \int_H (-C)^{-1} b(w,w) d\mu(w)$:
\begin{align*} 
du_t = Au_t dt + b(u_t,u_t) dt + b((-C)^{-1}Q^{1/2}\circ dW_t,u_t) dt +  b(r,u_t) dt. 
\end{align*} 
If in addition pathwise uniqueness holds for the limit equation then the whole sequence converges in law; moreover, convergence in $\mathbb{P}$-probability holds true if solutions to \eqref{eq:system_2} are probabilistically strong.
\end{thm}
With analytically weak solution we refer to the fact that the identity holds almost surely when tested against smooth functions $h \in \mathscr{S}$ and integrated with respect to time.
Also, the weak convergence $Q^{1/2}W^\epsilon \to Q^{1/2}W$ is meant as random variables taking values in $C([0,T],H)$, although many other spaces would do since $\{Q^{1/2}W^\epsilon\}_{\epsilon\in(0,1)}$ are identically distributed.

The proof is split in three parts. 
In the first place, invoking Simon compactness criterium (\autoref{lem:Simon}), we prove that the laws of the processes $\{u^\epsilon\}_{\epsilon \in (0,1)}$ are tight as probability measures on the space $L^2([0,T],H) \cap C([0,T],H^{-\beta})$ for every $\beta>0$ - cf. \autoref{prop:tightness} below; next, in subsequent \autoref{prop:identification} we show that every weak accumulation point $(u,Q^{1/2}W)$ is an analytically weak solution of the equation with effective generator $\mathscr{L}^0$ and It\=o transport noise (cf. \eqref{eq:u} below); finally, we check that the generator $\mathscr{L}^0$ can be split into the sum of the deterministic Navier-Stokes dynamics, the It\=o-to-Stratonovich corrector (which together with the It\=o integral gives the Stratonovich transport noise) and the It\=o-Stokes drift, as made explicit by \eqref{eq:ito_integral}, \eqref{eq:corrector1}, \eqref{e5.14}.

The last statement of Theorem \ref{thm:main} is classical and follows from Lemma 1.1 in \cite{GyKr96}. We omit the proof of this point.

\begin{rmk} \label{rmk:selection,energy bounds}
When the limit equation does not admit uniqueness, we do not know whether or not different subsequences can converge towards different solutions of the limit equation.
It might well be that, notwithstanding the fact that the limit equation admits multiple solutions, this approximating procedure only selects some special solution which enjoys additional properties. 
However, we are not able to answer this question at the moment: we can only provide a partial selection criterium based on the fact that every selected solution $u$ must satisfy the same energy bounds (S3) of the approximating sequence $\{u^\epsilon\}_{\epsilon\in(0,1)}$ (this latter property can be deduced first on Galerkin projections $\{\Pi_m u\}_{m \in \N}$, and then checked to be uniform in $m \in \N$).
This is of particular interest if we start with solutions satisfying the energy inequality as in \cite{FlRo08}.
\end{rmk}

As a preliminary step towards the proof of \autoref{thm:main}, we need a version of the celebrated It\=o Formula suited for our solution processes $(u^\epsilon,y^\epsilon)$. Indeed, since \eqref{eq:system_2} only holds in analytically weak sense (S2), the classical It\=o Formula \cite[Theorem 4.32]{DPZa02} is not a priori applicable to the process $\Phi(u^\epsilon_t,y^\epsilon_t)$ unless $\Phi$ only depends on a finite number of projections $\langle u^\epsilon_t, h_i \rangle$, $\langle y^\epsilon_t , k_i \rangle$, for some $h_i, k_i \in \mathscr{S}$.
Thus, our approach consists in applying first the classical It\=o Formula to Galerkin projections $\Pi_n u^\epsilon$, $\Pi_n y^\epsilon$, and then pass to the limit as $n \to \infty$, under suitable controls over $D_u \Phi, D_y \Phi$.

\begin{lem}[It\=o Formula] \label{lem:ito}
Let $\Phi: H \times H \to \R$ be such that, for every fixed  $u,y \in H$, it holds $\Phi(u,\cdot),\Phi(\cdot,y) \in \mathcal{E}$ { with $\|\Phi(u,\cdot)\|_\mathcal{E} \lesssim \|u\|_H$ and $\|\Phi(\cdot,y)\|_\mathcal{E} \lesssim \|y\|_H$}, and moreover
\begin{align} \label{eq:assumptions_Ito}
\|D_u \Phi(u,y) \|_{H^1} \lesssim 1+\|u\|_H+\|y\|_H,
\quad
\|D_y \Phi(u,y)\|_{H^1} \lesssim 1+\|u\|_H+\|y\|_H.
\end{align}
Let $(u^\epsilon,y^\epsilon)$ be a solution to \eqref{eq:system_2} in the sense of \autoref{def:sol}.
{
Then for every fixed $\epsilon\in(0,1)$ and $t \in [0,T]$ the following It\=o Formula holds $\PP$-a.s.}
\begin{align*}
\Phi(u^\epsilon_t,y^\epsilon_t)
&=
\Phi(u_0,y_0)
+
\int_0^t \mathscr{L}^{\epsilon}\Phi(u^\epsilon_s,y^\epsilon_s) ds
+
\int_0^t \langle D_y \Phi(u^\epsilon_s,y^\epsilon_s) , Q^{1/2}dW_s \rangle.
\end{align*}
\end{lem}
\begin{proof}
Let $\{\Pi_n\}_{n \in \N}$ be a family of Galerkin projectors and let $h \in H$ be fixed. Since $\Pi_n h \in \mathscr{S}$ for every $n \in \N$, by (S2) we have $\PP$-a.s. for every $t \in [0,T]$ : 
\begin{align*}
\langle u^\epsilon_t, \Pi_n h\rangle 
&= 
\langle u_0, \Pi_n h\rangle
+
\int_0^t
\langle u^\epsilon_s, A \Pi_n h \rangle
+ 
\int_0^t
\langle b(u^\epsilon_s,u^\epsilon_s) ,\Pi_n h \rangle ds 
\\
&\quad+ 
\epsilon^{-1/2} \int_0^t
\langle b(y^\epsilon_s,u^\epsilon_s) ,\Pi_n h \rangle ds,
\\
\langle y^\epsilon_t,\Pi_n h\rangle 
&= 
\langle y_0,\Pi_n h\rangle
+
\epsilon^{-1} \int_0^t 
\langle y^\epsilon_s , C_\epsilon\Pi_n h \rangle ds 
+ 
\int_0^t
\langle b(u^\epsilon_s,y^\epsilon_s) ,\Pi_n h \rangle ds
\\
&\quad 
+ \epsilon^{-1/2}\int_0^t
\langle b(y^\epsilon_s,y^\epsilon_s) ,\Pi_n h \rangle ds + 
\epsilon^{-1/2} \langle Q^{1/2} W_t ,\Pi_n h \rangle.
\end{align*}
Letting $h$ freely vary in $H$ in the previous expression, we deduce that the process $(\Pi_n u^\epsilon, \Pi_n y^\epsilon)$ is an It\=o process satisfying
\begin{align*}
\Pi_n u^\epsilon_t
&= 
\Pi_n u_0
+
\int_0^t
\Pi_n A  u^\epsilon_s ds
+ 
\int_0^t
\Pi_n b(u^\epsilon_s,u^\epsilon_s) ds 
+ 
\epsilon^{-1/2} \int_0^t
\Pi_n b(y^\epsilon_s,u^\epsilon_s) ds,
\\
\Pi_n y^\epsilon_t
&= 
\Pi_n y_0 
+
\epsilon^{-1} \int_0^t 
\Pi_n C_\epsilon y^\epsilon_s ds 
+ 
\int_0^t
\Pi_n b(u^\epsilon_s,y^\epsilon_s) ds
\\
&\quad 
+ \epsilon^{-1/2}\int_0^t
\Pi_n b(y^\epsilon_s,y^\epsilon_s) ds + 
\epsilon^{-1/2} \Pi_n Q^{1/2} W_t
\end{align*}
in strong analytical sense. In particular, by (S1) and classical It\=o Formula, the following a.s. identity holds for every $t \in [0,T]$:
\begin{align*} 
\Phi(\Pi_n u^\epsilon_t,\Pi_n y^\epsilon_t)
&=
\Phi(\Pi_n u_0,\Pi_n y_0)
+
\int_0^t \mathscr{L}^{\epsilon,n}_{u^\epsilon_s,y^\epsilon_s}\Phi(\Pi_n u^\epsilon_s,\Pi_n y^\epsilon_s) ds
\\
&\quad+
\int_0^t \langle D_y \Phi(\Pi_n u^\epsilon_s,\Pi_n y^\epsilon_s) , \Pi_n Q^{1/2}dW_s \rangle,
\end{align*}
where $\mathscr{L}^{\epsilon,n}_{u^\epsilon_s,y^\epsilon_s}\Phi(\Pi_n u^\epsilon_s,\Pi_n y^\epsilon_s)$ is given by
\begin{align*}
\mathscr{L}^{\epsilon,n}_{u^\epsilon_s,y^\epsilon_s} \Phi (\Pi_n u^\epsilon_s,\Pi_n y^\epsilon_s)
&=
\langle \Pi_n Au^\epsilon_s + \Pi_n b(u^\epsilon_s,u^\epsilon_s) , D_u \Phi(\Pi_n u^\epsilon_s,\Pi_n y^\epsilon_s) \rangle
\\
&\quad+
\epsilon^{-1/2} \langle \Pi_n b(y^\epsilon_s,u^\epsilon_s) , D_u \Phi (\Pi_n u^\epsilon_s,\Pi_n y^\epsilon_s) \rangle
\\
&\quad+
\langle \Pi_n b(u^\epsilon_s,y^\epsilon_s) , D_y \Phi (\Pi_n u^\epsilon_s,\Pi_n y^\epsilon_s) \rangle
\\
&\quad+
\epsilon^{-1/2} \langle \Pi_n b(y^\epsilon_s,y^\epsilon_s) , D_y \Phi (\Pi_n u^\epsilon_s,\Pi_n y^\epsilon_s) \rangle
\\
&\quad+
\epsilon^{-1} \langle \Pi_n C_\epsilon y^\epsilon_s , D_y \Phi (\Pi_n u^\epsilon_s,\Pi_n y^\epsilon_s) \rangle
\\
&\quad+
\frac{\epsilon^{-1}}{2} Tr (\Pi_n Q \Pi_n D^2_y \Phi (\Pi_n u^\epsilon_s)).
\end{align*}
Because by assumption { $u^\epsilon,y^\epsilon \in C([0,T],H_w)$ for every fixed $\epsilon$ and} $\Phi \in \mathcal{E}$ whenever any of its two argument is fixed, it is easy to check {that for every fixed $t \in [0,T]$ the convergences $\Phi(\Pi_n u^\epsilon_t,\Pi_n y^\epsilon_t) \to \Phi(u^\epsilon_t,y^\epsilon_t)$ and $\Phi(\Pi_n u_0,\Pi_n y_0) \to \Phi(u_0,y_0)$ hold true $\PP$-a.s. as $n \to \infty$.}
{ Indeed, with probability one $u^\epsilon_t , y^\epsilon_t \in H$ for every $t \in [0,T]$, and thus $\Pi_n u^\epsilon_t \to u^\epsilon_t$ and $\Pi_n y^\epsilon_t \to y^\epsilon_t$ strongly in $H$.}
By (S3), \eqref{eq:assumptions_Ito} and Lebesgue dominated convergence, we have, up to subsequences, { for every $t \in [0,T]$ fixed $\PP$-a.s.} 
\begin{align*}
\int_0^t
\mathscr{L}^{\epsilon,n} \Phi (\Pi_n u^\epsilon_s,\Pi_n y^\epsilon_s) ds
\to
\int_0^t
\mathscr{L}^\epsilon \Phi (u^\epsilon_s,y^\epsilon_s) ds.
\end{align*}
Similarly, since $\int_0^t \|\Pi_n  D_y \Phi(\Pi_n u^\epsilon_s,\Pi_n y^\epsilon_s)\|^2 ds \to \int_0^t \|D_y \Phi(u^\epsilon_s,y^\epsilon_s)\|^2 ds $ a.s. as $n \to \infty$, the following convergence in probability holds true
\begin{align*}
\int_0^t \langle D_y \Phi(\Pi_n u^\epsilon_s,\Pi_n y^\epsilon_s) , \Pi_n Q^{1/2}dW_s \rangle
\to
\int_0^t \langle D_y \Phi(u^\epsilon_s,y^\epsilon_s) ,Q^{1/2}dW_s \rangle,
\end{align*}
and the convergence is almost sure up to extracting a subsequence, concluding the proof.
\end{proof}

\begin{rmk}
\emph{i})
As a consequence of \autoref{lem:ito} and the definition of corrector $\varphi_1^\epsilon$ from the previous section, we immediately deduce that $\varphi_1^\epsilon$ belongs to the domain of the generator $\mathscr{L}^\epsilon$ and the It\=o Formula holds for the process $\varphi_1^\epsilon(u^\epsilon,y^\epsilon)$ for every $\epsilon>0$.
However, strictly speaking we do not actually need such a strong result.
For instance, it would have been sufficient to show the existence of a \emph{generalized corrector} $\tilde{\varphi}_1^\epsilon$ and an adapted process $H^{\epsilon}$ such that for every $t \in [0,T]$:
\begin{align*}
\tilde{\varphi}_1^\epsilon(u^\epsilon_t,y^\epsilon_t)
&=
\tilde{\varphi}_1^\epsilon(u_0,y_0)
+
\int_0^t H^{\epsilon}_s ds
+
\int_0^t \langle D_y \tilde{\varphi}_1^\epsilon(u^\epsilon_s,y^\epsilon_s) , Q^{1/2}dW_s \rangle,
\end{align*}
$\PP$-a.s., with some suitable additional requirements that allow to identify a limit equation for $u^\epsilon$ as done here in the following.
In particular, it is not necessary that the generalized corrector $\tilde{\varphi}_1^\epsilon$ is in the domain of $\mathscr{L}^\epsilon$.

On the other hand, whenever the arguments of previous sections work and produce a corrector $\varphi_1^\epsilon$ within the domain of $\mathscr{L}^\epsilon$, it is natural to choose it to apply the perturbed function method. 
Moreover, we can not avoid proving the validity of \emph{some} It\=o Formula for the process $\varphi_1^\epsilon(u^\epsilon,y^\epsilon)$, since it does not descend directly from our notion of solution (whereas an It\=o Formula for $\varphi(u^\epsilon)$ does); thus, proving previous \autoref{lem:ito} is a fully justified effort.

\emph{ii})
Since $Y^\epsilon$ is regular, it is possible to consider functions $\Phi_1(u,y,Y) = \Phi_1(\Phi(u,y),Y)$ and prove an analogous It\=o Formula for the process $\Phi_1(u^\epsilon_t,y^\epsilon_t,Y^\epsilon_t)$. 
\end{rmk}

\subsection{Tightness}
In this subsection we prove that the laws of the processes $\{u^\epsilon\}_{\epsilon \in (0,1)}$ are tight as probability measures on the space $L^2([0,T],H) \cap C([0,T],H^{-\beta})$ for every $\beta>0$.
The idea is to apply Simon compactness criterium \autoref{lem:Simon}. 
In order to do so, we need estimates on the increments $\mathbb{E} \left[\|u^\epsilon_t-u^\epsilon_s\|^p_{H^{-\sigma}} \right]$, $s,t \in [0,T]$, where $p>2$ and $\sigma>0$ are suitable parameters. 
Making use of the formula
\begin{align} \label{eq:sobolev_norm}
\|u^\epsilon_t-u^\epsilon_s\|^2_{H^{-\sigma}}
&=
\sum_{k \in \N}
\frac{\left(\varphi^k(u^\epsilon_t) - \varphi^k(u^\epsilon_s)\right)^2}{\nu_k^{2\sigma}},
\quad
\varphi^k(u) = \langle u,e_k\rangle,
\end{align}
for $\{e_k\}_{k \in \N}$ a complete orthonormal system in $H$ of eigenfunctions of $-A$ with associated eigenvalues $\{ \nu_k \}_{k \in \N}$, we reduce the problem to providing suitable estimates for the quantity $\mathbb{E} \left[ (\varphi^k(u^\epsilon_t)-\varphi^k(u^\epsilon_s))^p \right]$, which can be obtained applying It\=o Formula \autoref{lem:ito} to the test function $\varphi^{k,\epsilon}(u^\epsilon,y^\epsilon) = \varphi^k(u^\epsilon) + \epsilon^{1/2} \varphi_1^{k,\epsilon}(u^\epsilon,y^\epsilon)$, with $\varphi_1^{k,\epsilon}$ being given by \eqref{eq:varphi1} with $h=e_k$.

We prove first the following auxiliary lemma, consisting of an estimate on some negative Sobolev norm of the time increments $u^\epsilon_t-u^\epsilon_s$ and $y^\epsilon_t-y^\epsilon_s$.
\begin{lem} \label{lem:increments}
Let $\{(u^\epsilon,y^\epsilon)\}_{\epsilon\in(0,1)}$ be a bounded-energy family of weak martingale solutions to \eqref{eq:system_2}, and for every $\epsilon\in(0,1)$ let $Y^\epsilon$ be the unique strong solution of \eqref{eq:Yeps}.
Then for every $p \geq 1$ and $\theta=\max\{\theta_0,\Gamma\}$ the following estimates hold:
\begin{align*}
\mathbb{E}\left[\|u^\epsilon_t - u^\epsilon_s \|_{H^{-{\theta_0}}}^p\right]
&\lesssim 
\epsilon^{-p/2}|t-s|^p;
\\
\mathbb{E}\left[\|Y^\epsilon_t - Y^\epsilon_s \|_{H^{-{\theta_0}}}^p\right]
&\lesssim 
\epsilon^{-p/2}|t-s|^{p/2};
\\
\mathbb{E}\left[ \|y^\epsilon_t - y^\epsilon_s - (Y^\epsilon_t - Y^\epsilon_s )\|^p_{H^{-\theta}}\right]
&\lesssim 
\epsilon^{-p}|t-s|^{p}.
\end{align*}
\end{lem}
\begin{proof}
Let us start from the estimate on $u^\epsilon$.
We have for every $h \in H^{\theta_0}$
\begin{align*}
\langle u^\epsilon_t - u^\epsilon_s , h\rangle
&=
\int_s^t
\langle u^\epsilon_r, Ah \rangle dr
+ 
\int_s^t
\langle b(u^\epsilon_r,u^\epsilon_r) , h \rangle dr 
+
\epsilon^{-1/2} \int_s^t
\langle b(y^\epsilon_r,u^\epsilon_r) , h \rangle dr,
\end{align*}
hence, using (B1)
\begin{align*}
|\langle u^\epsilon_t - u^\epsilon_s , h\rangle|
&\lesssim
\int_s^t
\| u^\epsilon_r\|_{H} \|h\|_{H^2} dr
+ 
\int_s^t
\|u^\epsilon_r\|_{H}^2 \|h\|_{H^{\theta_0}} dr 
+ 
\epsilon^{-1/2} \int_s^t
\|y^\epsilon_r\|_{H} \|u^\epsilon_r\|_{H} \|h\|_{H^{\theta_0}} dr.
\end{align*}
Therefore, taking the supremum over $h \in H^{\theta_0}$ with $\|h\|_{H^{\theta_0}}=1$, elevating to the $p$-th power and taking expectations:
\begin{align*}
\mathbb{E}\left[\|u^\epsilon_t - u^\epsilon_s \|_{H^{-{\theta_0}}}^p\right]
&\lesssim 
\epsilon^{-p/2}|t-s|^{p}.
\end{align*}

In order to get the estimate on $Y^\epsilon$, we preliminarily rewrite the increment $Y^\epsilon_t - Y^\epsilon_s$ using the mild formulation of \eqref{eq:Yeps}:
\begin{align*}
Y^\epsilon_t - Y^\epsilon_s
&=
\left( e^{\epsilon^{-1}C_\epsilon(t-s)} - 1 \right) Y^\epsilon_s
+
\epsilon^{-1/2}\int_s^t e^{\epsilon^{-1}C_\epsilon(t-r)} Q^{1/2}dW_r,
\end{align*}
from which we are able to deduce, on the one hand
\begin{align*}
\mathbb{E} \left[ \left\|\left( e^{\epsilon^{-1}C_{\epsilon}(t-s)} - 1 \right) Y^\epsilon_s \right\|_{H^{-\theta_0}}^p\right]
&\lesssim
\epsilon^{-p/2} |t-s|^{p/2} \mathbb{E} \left[ \|Y^\epsilon_s\|_{H^{\Gamma-{\theta_0}}}^p\right]
\lesssim
\epsilon^{-p/2}|t-s|^{p/2},
\end{align*}
and, applying It\=o Isometry:
\begin{align*}
\mathbb{E} \left[ \left\| \epsilon^{-1/2}\int_s^t e^{\epsilon^{-1}C_\epsilon(t-r)} Q^{1/2}dW_r \right\|_{H^{-{\theta_0}}}^p \right]
\lesssim
|t-s|^{p/2},
\end{align*}
on the other.

Let us move to the estimate on $y^\epsilon$.
First, since $Y^\epsilon$ is a strong solution of \eqref{eq:Yeps}, for every fixed $h \in H^{\theta_0}$ and $s,t \in [0,T]$, $s<t$, we have the following weak reformulation of \eqref{eq:Yeps}:
\begin{align*}
\langle Y^\epsilon_t - Y^\epsilon_s , h\rangle
&=
\epsilon^{-1} \int_s^t \langle Y^\epsilon_r, C_\epsilon h \rangle dr
+
\epsilon^{-1/2} \langle Q^{1/2} (W_t-W_s) , h \rangle,
\end{align*}
so that putting the previous expression together with (S2) we get
\begin{align*}
\langle y^\epsilon_t - y^\epsilon_s- (Y^\epsilon_t - Y^\epsilon_s) , h \rangle
&=
\epsilon^{-1} \int_s^t \langle y^\epsilon_r - Y^\epsilon_r, C_\epsilon h \rangle dr
+ 
\int_s^t
\langle b(u^\epsilon_r,y^\epsilon_r) , h \rangle dr 
\\
&\quad+
\epsilon^{-1/2} \int_s^t
\langle b(y^\epsilon_r,y^\epsilon_r) ,h \rangle dr.
\end{align*}
Hence, arguing as with $u^\epsilon_t-u^\epsilon_s$ we obtain 
\begin{align*}
\mathbb{E}\left[ \|y^\epsilon_t - y^\epsilon_s - (Y^\epsilon_t - Y^\epsilon_s )\|^p_{H^{-\theta}}\right]
\lesssim
\epsilon^{-p}|t-s|^{p} .
\end{align*}
\end{proof}

We move now to the main computation of this subsection. We have:

\begin{lem} \label{lem:Sobolev_increments}
There exists $\alpha>0$ depending only on $\gamma$, $\Gamma$ and $\theta_0$ such that the following holds. For every $p>2$ there exists $\sigma>0$ such that, for every $s,t \in [0,T]$: 
\begin{align*}
\mathbb{E} \left[
\|u^\epsilon_t-u^\epsilon_s\|^p_{H^{-\sigma}} \right]
&\lesssim |t-s|^{\alpha p}.
\end{align*}
\end{lem}

\begin{proof}
Let us consider as test function $\varphi^{k,\epsilon} = \varphi^k + \epsilon^{1/2} \varphi_1^{k,\epsilon}$ as above, namely $\varphi^k(u)=\langle u,e_k \rangle$ and $\{e_k\}_{k \in \N}$ an eigenbasis of $-A$, and $\varphi_1^{k,\epsilon}$ given by \eqref{eq:varphi1} with $h=e_k$.
With this choice of $\varphi^k$ we have $D_u \varphi^k = e_k$ and
\begin{align*}
\varphi_1^{k,\epsilon}(u^\epsilon,y^\epsilon) &= \langle b((-C_\epsilon)^{-1}y^\epsilon,u^\epsilon),e_k \rangle,
\\
D_u \varphi_1^{k,\epsilon}(y^\epsilon) &= -b((-C_\epsilon)^{-1}y^\epsilon,e_k),
\\
D_y \varphi_1^{k,\epsilon}(u^\epsilon) &= \langle b((-C_\epsilon)^{-1}\cdot,u^\epsilon),e_k \rangle.
\end{align*}
Applying It\=o Formula to $\varphi^{k,\epsilon}(u^\epsilon,y^\epsilon)$ we get almost surely for any given $s,t \in [0,T]$, $s<t$:
\begin{align} \label{eq:increments_k}
\varphi^k(u^\epsilon_t) - \varphi^k(u^\epsilon_s)
&=
\epsilon^{1/2} (\varphi_1^{k,\epsilon}(u^\epsilon_s,y^\epsilon_s) - \varphi_1^{k,\epsilon}(u^\epsilon_t,y^\epsilon_t))
\\
&\quad + \nonumber
\int_s^t \mathscr{L}^\epsilon\varphi^{k,\epsilon}(u^\epsilon_r,y^\epsilon_r) dr
+
\int_s^t \langle D_y \varphi_1^{k,\epsilon} (u^\epsilon_r), Q^{1/2}dW_r \rangle.
\end{align} 
Therefore, using \eqref{eq:sobolev_norm} and H\"older inequality, for every $\sigma>0$ satisfying $\sum_{k \in \N} \nu_k^{-2\sigma}< \infty$ we get the following inequality
\begin{align} \label{eq:aux003}
\mathbb{E} \left[
\|u^\epsilon_t-u^\epsilon_s\|^p_{H^{-\sigma}} \right]
&=
\mathbb{E} \left[
\left( \sum_{k\in\mathbb{N}} \frac{\left(\varphi^k(u^\epsilon_t)-\varphi^k(u^\epsilon_s)\right)^2}{\nu_k^{2\sigma}} \right)^{p/2} \right]
\\
&\leq \nonumber
\left( \sum_{k\in\mathbb{N}} \frac{1}{\nu_k^{2\sigma}} \right)^{\frac{p-2}{p}}
\mathbb{E} \left[
\sum_{k\in\mathbb{N}} \frac{\left(\varphi^k(u^\epsilon_t)-\varphi^k(u^\epsilon_s)\right)^p}{\nu_k^{2\sigma}} \right]
.
\end{align}

Let us estimate the summands on the right-hand-side of \eqref{eq:increments_k} to complete the proof of the proposition.
We start from the terms involving the time increment $\varphi_1^{k,\epsilon}(u^\epsilon_s,y^\epsilon_s)-\varphi_1^{k,\epsilon}(u^\epsilon_t,y^\epsilon_t)$: for every $s,t \in [0,T]$, $s<t$ it holds
\begin{align*}
|\varphi_1^{k,\epsilon}(u^\epsilon_s,y^\epsilon_s)-\varphi_1^{k,\epsilon}(u^\epsilon_t,y^\epsilon_t)|
&\leq
|\varphi_1^{k,\epsilon}(u^\epsilon_s-u^\epsilon_t,y^\epsilon_s)|
+
|\varphi_1^{k,\epsilon}(u^\epsilon_t,y^\epsilon_s-y^\epsilon_t)|
\\
&=
|\langle b((-C_\epsilon)^{-1}y^\epsilon_s,e_k),u^\epsilon_s-u^\epsilon_t \rangle|
+
|\langle b((-C_\epsilon)^{-1}(y^\epsilon_s-y^\epsilon_t),e_k),u^\epsilon_t \rangle|
\\
&\lesssim
\|u^\epsilon_s-u^\epsilon_t\|_{H^{-2\gamma}} \|y^\epsilon_s\|_H \|e_k\|_{H^{\theta_1}}
+
\|y^\epsilon_s-y^\epsilon_t\|_{H^{-2\gamma}}\|u^\epsilon_t\|_H\|e_k\|_{H^{\theta_0}}.
\end{align*}
We can invoke \autoref{lem:increments} and interpolation inequality to estimate the $H^{-2\gamma}$ norm of the time increments $u^\epsilon_s-u^\epsilon_t$ and $y^\epsilon_s-y^\epsilon_t$, and get (without loss of generality we assume $\gamma \leq \theta_0/4$)
\begin{align*}
\epsilon^{p/2} &\, \mathbb{E} \left[
\left( \varphi_1^{k,\epsilon}(u^\epsilon_s,y^\epsilon_s) - \varphi_1^{k,\epsilon}(u^\epsilon_t,y^\epsilon_t) 
\right)^p \right]
\\
&\lesssim
\epsilon^{p/2}\|e_k\|_{H^{\theta_1}}^p\, 
\mathbb{E} \left[
\|u^\epsilon_t-u^\epsilon_s\|_{H^{-\theta_0}}^{2 \gamma p / \theta_0} 
\|u^\epsilon_t-u^\epsilon_s\|_{H}^{(1-2\gamma/\theta_0) p} 
\|y^\epsilon_t\|_H^p
\right]
\\
&\quad+
\epsilon^{p/2} \|e_k\|_{H^{\theta_0}}^p \, 
\mathbb{E} \left[
\|y^\epsilon_t-y^\epsilon_s\|_{H^{-\theta}}^{2 \gamma p / \theta} 
\|y^\epsilon_t-y^\epsilon_s\|_{H}^{(1-2\gamma/\theta) p}
\|u^\epsilon_s\|^p_H \right]
\\
&\lesssim
\|e_k\|_{H^{\theta_1}}^p
\epsilon^{p/2}\epsilon^{-2\gamma p/\theta_0}
|t-s|^{2\gamma p/\theta_0}
+
\|e_k\|_{H^{\theta_0}}^p
\epsilon^{p/2}\epsilon^{-2\gamma p/\theta}
|t-s|^{\gamma p/\theta}
\\
&\lesssim
\|e_k\|_{H^{\theta_1}}^p
|t-s|^{2\gamma p/\theta_0}
+
\|e_k\|_{H^{\theta_0}}^p
|t-s|^{\gamma p/\theta},
\end{align*}
where we recall $\theta=\max\{\theta_0,\Gamma\}$.
Let us move now the term with the time integral of $\mathscr{L}^\epsilon \varphi^{k,\epsilon} (u^\epsilon_r,y^\epsilon_r)$. We conveniently rewrite this term as $\mathscr{L}^\epsilon \varphi^{k,\epsilon}=\Phi^{k,\epsilon} + \epsilon^{1/2}\Phi^{k,\epsilon}_1$, where for $r \in [0,T]$, $\Phi^{k,\epsilon},\Phi^{k,\epsilon}_1$ are implicitly given by
\begin{align*}
\mathscr{L}^\epsilon \varphi^{k,\epsilon} (u^\epsilon_r,y^\epsilon_r)
&=
\langle Au^\epsilon_r + b(u^\epsilon_r,u^\epsilon_r) , D_u \varphi^k \rangle
+
\langle b(y^\epsilon_r,u^\epsilon_r) , D_u \varphi_1^{k,\epsilon}(y^\epsilon_r) \rangle
+
\langle b(y^\epsilon_r,y^\epsilon_r) , D_y \varphi_1^{k,\epsilon}(u^\epsilon_r) \rangle
\\
&\quad+
\epsilon^{1/2} \left( 
\langle Au^\epsilon_r + b(u^\epsilon_r,u^\epsilon_r), D_u \varphi_1^{k,\epsilon}(y^\epsilon_r) \rangle 
+ 
\langle b(u^\epsilon_r,y^\epsilon_r), D_y \varphi_1^{k,\epsilon}(u^\epsilon_r) \rangle \right)
\\
&\eqqcolon \Phi^{k,\epsilon}(u^\epsilon_r,y^\epsilon_r) + \epsilon^{1/2} \Phi^{k,\epsilon}_1(u^\epsilon_r,y^\epsilon_r).
\end{align*}
We have the inequalities
\begin{align*}
|\langle Au^\epsilon_r + b(u^\epsilon_r,u^\epsilon_r) , D_u \varphi^k \rangle|
&=
|\langle Au^\epsilon_r + b(u^\epsilon_r,u^\epsilon_r) , e_k \rangle|
\\
&\lesssim
\|u^\epsilon_r\|_H \|e_k\|_{H^2} + \|u^\epsilon_r\|_{H}^2 \|e_k\|_{H^{\theta_0}},
\\
|\langle b(y^\epsilon_r,u^\epsilon_r) , D_u \varphi_1^{k,\epsilon}(y^\epsilon_r) \rangle|
&\lesssim 
\|e_k\|_{H^{\theta_1}}\|y^\epsilon_r\|_H \|y^\epsilon_r\|_{H^{3/2-2\gamma}} \|u^\epsilon_r\|_{H^1},
\\ 
|\langle b(y^\epsilon_r,y^\epsilon_r) , D_y \varphi_1^{k,\epsilon}(u^\epsilon_r) \rangle|
&\lesssim 
\|e_k\|_{H^{\theta_1}}\|y^\epsilon_r\|_{H^1} \|y^\epsilon_r\|_{H^{3/2-2\gamma}} \|u^\epsilon_r\|_H,
\end{align*}
so that the time integral of $\Phi^{k,\epsilon}(u^\epsilon_r,y^\epsilon_r)$ satisfies:
\begin{align*}
\mathbb{E} \left[ \left( \int_s^t \Phi^{k,\epsilon}(u^\epsilon_r,y^\epsilon_r) dr \right)^p\right]
&\lesssim
\|e_k\|_{H^{\theta_1}}^p |t-s|^{(\gamma-1/4)p}
\end{align*}
uniformly in $\epsilon$.
Similarly, 
\begin{align*}
|\langle Au^\epsilon_r , D_u \varphi_1^{k,\epsilon}(y^\epsilon_r) \rangle |
&\lesssim 
\|e_k\|_{H^{\theta_1}} \|u^\epsilon_r\|_{H^1} \|y^\epsilon_r\|_{H^{1-2\gamma}},
\\
|\langle b(u^\epsilon_r,u^\epsilon_r) , D_u \varphi_1^{k,\epsilon}(y^\epsilon_r) \rangle |
&\lesssim
\|e_k\|_{H^{\theta_1}} \|u^\epsilon_r\|_{H^1} \|u^\epsilon_r\|_{H^{3/2-2\gamma}} \|y^\epsilon_r\|_H,
\\
|\langle b(u^\epsilon_r,y^\epsilon_r), D_y \varphi_1^{k,\epsilon}(u^\epsilon_r) \rangle |
&\lesssim
\|e_k\|_{H^{\theta_1}} \|y^\epsilon_r\|_{H^1} \|u^\epsilon_r\|_{H^{3/2-2\gamma}} \|u^\epsilon_r\|_H,
\end{align*}
and we can bound the time integral of $\Phi^{k,\epsilon}_1(u^\epsilon_r,y^\epsilon_r)$  with
\begin{align*}
\epsilon^{p/2}
\mathbb{E} \left[ \left( \int_s^t \Phi^{k,\epsilon}_1(u^\epsilon_r,y^\epsilon_r) dr \right)^p\right]
&\lesssim
\epsilon^{p/2}
\|e_k\|_{H^{\theta_1}}^p |t-s|^{(\gamma-1/4)p}.
\end{align*}

The last term remaining is the stochastic integral; we have by It\=o Isometry
\begin{align*}
\mathbb{E}&\left[ \left(
\int_s^t \langle D_y \varphi_1^{k,\epsilon} (u^\epsilon_r), Q^{1/2}dW_r \rangle \right)^p \right]
\\
&=
\mathbb{E}\left[ \left(
\int_s^t \langle b((-C_\epsilon)^{-1}Q^{1/2}dW_r,u^\epsilon_r), e_k \rangle \right)^p \right]
\lesssim \|e_k\|^p_{H^{\theta_0}} |t-s|^{p/2}.
\end{align*}

Putting all together, we finally arrive to the following bound, uniform in $\epsilon$ and valid for every $k \in \N$, $s,t \in [0,T]$, $s<t$ and for every $p>2$:
\begin{align} \label{eq:aux002}
\mathbb{E} \left[ 
\left( \varphi^k(u^\epsilon_t) - \varphi^k(u^\epsilon_s)\right)^p \right]
&\lesssim
\|e_k\|_{H^{\theta_1}}^p |t-s|^{\alpha p},
\quad
\alpha = \min \left\{ 2\gamma/\theta_0 , \gamma/\theta, \gamma-1/4 \right\}.
\end{align}

Recall that in order to estimate the $H^{-\sigma}$ norm of $u^\epsilon_t-u^\epsilon_s$ we have to sum \eqref{eq:aux002} above over all $k \in \N$; for this reason, we further require that $\sigma$ is such that
\begin{gather*}
\sum_{k\in\mathbb{N}} \frac{\|e_k\|^p_{H^{\theta_1}}}{\nu_k^{2\sigma}}<\infty,
\end{gather*}
so that by Fubini-Tonelli Theorem, \eqref{eq:aux003} and \eqref{eq:aux002}
\begin{align*}
\mathbb{E} \left[
\|u^\epsilon_t-u^\epsilon_s\|^p_{H^{-\sigma}} \right]
&=
\mathbb{E} \left[
\left( \sum_{k\in\mathbb{N}} \frac{\left(\varphi^k(u^\epsilon_t)-\varphi^k(u^\epsilon_s)\right)^2}{\nu_k^{2\sigma}} \right)^{p/2} \right]
\\
&\leq
\left( \sum_{k\in\mathbb{N}} \frac{1}{\nu_k^{2\sigma}} \right)^{\frac{p-2}{p}}
\mathbb{E} \left[
\sum_{k\in\mathbb{N}} \frac{\left(\varphi^k(u^\epsilon_t)-\varphi^k(u^\epsilon_s)\right)^p}{\nu_k^{2\sigma}} \right]
\lesssim |t-s|^{\alpha p}.
\end{align*}
\end{proof}

Thus, we are ready to prove the first part of \autoref{thm:main}, that is:
\begin{prop} \label{prop:tightness}
For every $\beta>0$, the laws of the processes $\{u^\epsilon\}_{\epsilon \in (0,1)}$ are tight as probability measures on the space $L^2([0,T],H) \cap C([0,T],H^{-\beta})$.
\end{prop}
\begin{proof}
Let $\alpha_0$ be given by previous \autoref{lem:Sobolev_increments}, and take $\alpha \in (0,\alpha_0)$, $p>1/\alpha$ and $\sigma>0$ such that the lemma holds.
By the aforementioned lemma, $\mathbb{E}\left[\|u^\epsilon\|_{W^{\alpha,p}([0,T],H^{-\sigma})}\right]$ is bounded uniformly in $\epsilon$; since in addition $\mathbb{E}\left[ \|u^\epsilon\|_{L^\infty([0,T],H)} \right]$ and $\mathbb{E}\left[ \|u^\epsilon\|_{L^2([0,T],H^1)} \right]$ are also bounded uniformly in $\epsilon$ by assumption (S3), Simon compactness criterium \autoref{lem:Simon} yields tightness of the sequence of laws of the processes $\{u^\epsilon\}_{\epsilon \in (0,1)}$ in the space $L^2([0,T],H)$ $\cap$ $C([0,T],H^{-\beta})$ for every $\beta > 0$.
\end{proof}

\subsection{Identification of the limit}
\label{s5.2}
Let $\varphi=\langle \cdot,h\rangle \in F$ be a test function, and denote $\varphi^\epsilon(u,y,Y) = \varphi(u) + \epsilon^{1/2} \varphi_1^\epsilon(u,y) + \epsilon \varphi_2^\epsilon(u,Y)$, where $\varphi_1^\epsilon$ and $\varphi_2^\epsilon$ are given by \eqref{eq:varphi1} and \eqref{eq:varphi2} respectively.
Let us also introduce the homogeneous corrector $\varphi_1$ via the formula
\begin{align} \label{eq:varphi1_lim}
\varphi_1(u,y) 
\coloneqq 
\langle b((-C)^{-1}y,u) , h \rangle,
\end{align}
and the limiting effective generator $\mathscr{L}^0$ by
\begin{align} \label{eq:L0_lim}
\mathscr{L}^0 \varphi (u)
=
\langle Au +b(u,u),h \rangle
+
\int_H \psi_u(w) d\mu(w),
\end{align}
where $\psi_u(w)=\langle b(w,u) , D_u \varphi_1(w)\rangle + \langle b(w,w) , D_y \varphi_1(u) \rangle$ and $\mu = \mathcal{N}(0,Q_\infty)$.

Since $(u^\epsilon,y^\epsilon)$ is a weak solution of system \eqref{eq:system_2} and $Y^\epsilon$ is a strong solution to \eqref{eq:Yeps}, by It\=o Formula \autoref{lem:ito} we have almost surely for every $t \in [0,T]$:
\begin{align*}
\varphi^\epsilon(u^\epsilon_t,y^\epsilon_t,Y^\epsilon_t)
&=
\varphi^\epsilon(u_0,y_0,0)
+
\int_0^t \mathscr{L}^\epsilon \varphi^\epsilon(u^\epsilon_s,y^\epsilon_s,Y^\epsilon_s) ds
\\
&\quad+
\epsilon^{-1/2}\int_0^t \langle D_y \varphi^\epsilon(u^\epsilon_s,y^\epsilon_s,Y^\epsilon_s) , Q^{1/2} dW_s \rangle,
\end{align*}
or equivalently, 
\begin{align} \label{eq:002}
\varphi(u^\epsilon_t) 
&= 
\varphi(u_0)
+
\int_0^t \mathscr{L}^{0} \varphi(u^\epsilon_s) ds
+
\int_0^t \langle b((-C)^{-1}Q^{1/2} dW_s,u^\epsilon_s) , h\rangle
\\
&\quad+ \nonumber
\int_0^t \left( \mathscr{L}^{0,\epsilon} \varphi(u^\epsilon_s) - \mathscr{L}^{0} \varphi(u^\epsilon_s) \right) ds
+ 
\int_0^t \langle b(\left((-C_\epsilon)^{-1}-(-C)^{-1}\right) Q^{1/2} dW_s,u^\epsilon_s) , h\rangle
\\
&\quad \nonumber
+ \epsilon^{1/2} \left( \varphi_1^\epsilon(u_0,y_0)-\varphi_1^\epsilon(u^\epsilon_t,y^\epsilon_t)\right)
+ \epsilon \left( \varphi_2^\epsilon(u_0,0)-\varphi_2^\epsilon(u^\epsilon_t,Y^\epsilon_t)\right)
\\
&\quad \nonumber
+ \int_0^t \Phi_0^\epsilon(u^\epsilon_s,y^\epsilon_s,Y^\epsilon_s)ds
+ \epsilon^{1/2} \int_0^t \Phi_1^\epsilon(u^\epsilon_s,y^\epsilon_s,Y^\epsilon_s)ds
+ \epsilon \int_0^t \Phi_2^\epsilon(u^\epsilon_s,Y^\epsilon_s)ds
\\
&\quad+ \nonumber
\epsilon^{1/2} 
\int_0^t \langle D_Y \varphi_2^\epsilon(u^\epsilon_s,Y^\epsilon_s) ,Q^{1/2} dW_s \rangle, 
\end{align}
where $\mathscr{L}^0$ is the limiting effective generator defined by \eqref{eq:L0_lim}, $\mathscr{L}^{0,\epsilon}$ is the effective generator defined by \eqref{eq:L0} and we have denoted for notational simplicity
\begin{align*}
\Phi_0^\epsilon(u,y,Y)
&=
\langle b(y,u), D_u \varphi_1^\epsilon(y) \rangle 
+
\langle b(y,y), D_y \varphi_1^\epsilon(u) \rangle
\\
&\quad-
\langle b(Y,u), D_u \varphi_1^\epsilon(Y) \rangle 
+
\langle b(Y,Y), D_y \varphi_1^\epsilon(u) \rangle,
\\
\Phi_1^\epsilon (u,y,Y)
&=
\langle Au + b(u,u), D_u \varphi_1^\epsilon(y) \rangle 
+ 
\langle b(u,y), D_y \varphi_1^\epsilon(u) \rangle
\\
&\quad+ 
\langle b(y,u), D_u \varphi_2^\epsilon(Y) \rangle ,
\\
\Phi_2^\epsilon (u,Y)
&=
\langle Au + b(u,u), D_u \varphi_2^\epsilon(Y) \rangle 
.
\end{align*}

Equation \eqref{eq:002} clearly indicates the candidate limit dynamics - first line of the equation - and the remainder terms - lines second to fifth. 
Our aim consists in proving, on the one hand the convergence of the first line to the same quantity evaluated at $u^\epsilon=u$ (for a possibly different Wiener process $W$; recall that at this stage the stochastic basis is still dependent on $\epsilon$), and  on the other hand the convergence of all remainders to zero.

In order to conveniently pass to the limit $\epsilon \to 0$, we invoke a standard combination of Prokhorov Theorem and Skorokhod Theorem. 
Indeed, since the family of laws of the processes $\{u^\epsilon\}_{\epsilon\in(0,1)}$ is tight on the space $L^2([0,T],H)$ $\cap$ $C([0,T],H^{-\beta})$ for every $\beta>0$, and $\{Q^{1/2}W=Q^{1/2}W^\epsilon\}_{\epsilon\in(0,1)}$ is a family of identically distributed $C([0,T],H)$-valued random variables, by Prokhorov Theorem there exists a subsequence $\epsilon_n \to 0$ such that $(u^{\epsilon_n},Q^{1/2}W^{\epsilon_n})$ converges in distribution as $n \to \infty$ towards a process $(u,Q^{1/2}W)$\footnote{Recall that any Wiener process with covariance operator $Q$ can be written as $Q^{1/2} W$ for some cylindrical Wiener process $W$ on $H$.} taking values in the space: 
\begin{align*}
\mathcal{X} \coloneqq \left( L^2([0,T],H) \cap C([0,T],H^{-\beta}) \right) \times C([0,T],H).
\end{align*}

Then, given \emph{any} subsequence such that $(u^{\epsilon_n},Q^{1/2}W^{\epsilon_n}) \to (u,Q^{1/2}W)$ in distribution (not necessarily that one provided by Prokhorov Theorem), in virtue of Skorokhod Theorem there exists a new probability space $(\tilde{\Omega},\tilde{\mathcal{F}},\tilde{\mathbb{P}})$ supporting $\mathcal{X}$-valued random variables $(\tilde{u},Q^{1/2}\tilde{W}) \sim (u,Q^{1/2}W)$ and $(\tilde{u}^n,Q^{1/2}\tilde{W}^n ) \sim (u^{\epsilon_n},Q^{1/2} W^{\epsilon_n})$ for every $n \in \N$ such that $(\tilde{u}^n,Q^{1/2}\tilde{W}^n ) \to (\tilde{u},Q^{1/2}\tilde{W})$ $\tilde{\mathbb{P}}$-almost surely as random variables in $\mathcal{X}$.
Of course, as usually done in these situations we drop the tildes in what follows.

\begin{prop} \label{prop:identification}
Let $(u^n,Q^{1/2}W^n) \to (u,Q^{1/2}W)$ as above.
Then for every $\varphi \in F$ we have the almost sure identity
\begin{align} \label{eq:u}
\varphi(u_t) 
&= 
\varphi(u_0)
+
\int_0^t \mathscr{L}^0 \varphi(u_s) ds
+
\int_0^t \langle b((-C)^{-1} Q^{1/2} dW_s, u_s), h \rangle,
\quad
\forall t \in [0,T].
\end{align}
\end{prop}
\begin{proof}
We divide the proof in three steps. First, we show that the remainder terms are infinitesimal in mean square as $n \to \infty$;
second, we prove that the deterministic effective dynamics is a continuous function of the path $\xi \in C([0,T],H^{-\beta}) \cap L^2([0,T],H)$; finally, we invoke a martingale representation theorem to identify the limit behaviour of the martingale term in \eqref{eq:u}.

\emph{Step 1}. Let us focus on the remainder terms in the right-hand-side of \eqref{eq:002}.
They are of several kinds: \emph{i}) terms involving the differences 
\begin{align*}
\int_0^t &\mathscr{L}^{0,\epsilon_n} \varphi (u^{\epsilon_n}_s) ds 
-
\int_0^t \mathscr{L}^{0} \varphi (u^{\epsilon_n}_s) ds,
\qquad
\int_0^t \langle b(G_{\epsilon_n} Q^{1/2} dW_s,u^{\epsilon_n}_s) , h\rangle,
\end{align*} 
where the operator $G_{\epsilon_n} \coloneqq (-C_{\epsilon_n})^{-1}-(-C)^{-1} = {\epsilon_n} (-C)^{-1}A(-C_\epsilon)^{-1}$, which are controlled using the bounds $\|G_{\epsilon_n}\|_{H^s \to H^{s+2\gamma(1+\beta)-2\beta}} \lesssim \epsilon_n^\beta$ and  $\|e^{C_{\epsilon_n} t} - e^{C t}\|_{H^{\theta+2\beta} \to H^\theta} \lesssim \epsilon_n^\beta $ for every $\beta \in [0,1]$, uniformly in $t \in [0,\infty)$, and go to zero in mean square as $n \to \infty$ (and ${\epsilon_n} \to 0$);
\emph{ii}) terms of the form $\epsilon_n^{1/2} \varphi_1^{\epsilon_n}(u^n_t,y^n_t)$ or $\epsilon_n \varphi_2^{\epsilon_n}(u^n_t,Y^n_t)$, $t \in [0,T]$, $Y^n \coloneqq Y^{\epsilon_n}$, which can be easily shown to converge to zero in mean square as $n \to \infty$ as a consequence of energy bounds for $(u^\epsilon,y^\epsilon)$, the bound $\|\varphi_2^{\epsilon_n}(u^n_t,\cdot)\|_{\mathcal{E}^{\theta'}} \lesssim \|h\|_{H^{\theta_1}} \|u^n_t\|_H$ for some $\theta'>5/4-2\gamma$ and
\begin{align*}
|\varphi_1^{\epsilon_n}(u^n_t,y^n_t)| 
&\lesssim
\|y^n_t\|_H \|u^n_t\|_H \|h\|_{H^{\theta_0}},
\\
|\varphi_2^{\epsilon_n}(u^n_t,Y^n_t)| 
&\lesssim 
(1+\|Y^n_t\|^2_{H^{\theta'}}) \|h\|_{H^{\theta_1}} \|u^n_t\|_H;
\end{align*} 
\emph{iii}) the term $\int_0^t \Phi_0^{\epsilon_n}(u^n_s,y^n_s,Y^n_s) ds$, which is infinitesimal in mean square by \autoref{prop:linearisation} and \autoref{prop:y-Y};
\emph{iv}) the terms involving the time integrals of $\Phi_1^{\epsilon_n}(u^n_s,y^n_s,Y^n_s)$ and $\Phi_2^{\epsilon_n}(u^n_s,Y^n_s)$, which are controlled by \autoref{prop:regularity_varphi1}, \autoref{prop:regularity_varphi2} and the estimates:
\begin{align*}
|\Phi_1^{\epsilon_n} (u^n_s,y^n_s,Y^n_s)|
&\lesssim
|\langle Au^n_s + b(u^n_s,u^n_s), D_u \varphi_1^{\epsilon_n}(y^n_s) \rangle| 
+ 
|\langle b(u^n_s,y^n_s), D_y \varphi_1^{\epsilon_n}(u^n_s) \rangle|
\\
&\quad
+ 
|\langle b(y^n_s,u^n_s),D_u \varphi_2^{\epsilon_n}(Y^n_s) \rangle| 
\\
&\lesssim 
\|u^n_s\|_H (1+\|u^n_s\|_{H^1}) \|y^n_s\|_H \|h\|_{H^{\theta_1}}
\\
&+
\|u^n_s\|_H \|y^n_s\|_H \|Y^n_s\|_{H^{\theta_0}}^2 \|h\|_{H^{\theta_1}},
\\
|\Phi_2^{\epsilon_n} (u^n_s,Y^n_s)|
&\lesssim 
\|u^n_s\|_H(1+\|u^n_s\|_H) \|Y^n_s\|_{H^{\theta_0}}^2 \|h\|_{H^{\theta_1}};
\end{align*}
and finally, \emph{v}) the stochastic integral $\epsilon_n^{1/2} \int_0^t \langle D_Y \varphi_2^{\epsilon_n}(u^n_s,Y^n_s), Q^{1/2} dW^n_s \rangle$, which by It\=o Isometry satisfies
\begin{align*}
\epsilon_n \mathbb{E} \left[ 
\left(\int_0^t \langle D_Y \varphi_2^{\epsilon_n}(u^n_s,Y^n_s), Q^{1/2} dW^n_s \rangle\right)^2
\right]
&=
\epsilon_n \mathbb{E} \left[ 
\int_0^t \| Q^{1/2}  D_Y \varphi_2^{\epsilon_n}(u^n_s,Y^n_s)\|_H^2 ds \right]
\to 0.
\end{align*}
Thus all the remainders converge to zero in mean square.

\emph{Step 2}. 
Let us consider on the path space $\mathcal{X}$ equipped with its Borel sigma field $\mathcal{B}$ the pushforward probability measures
\begin{align*}
\mathbb{Q}^n \coloneqq \mathbb{P} \circ (u^n,Q^{1/2}W^n)^{-1},
\quad
\mathbb{Q} \coloneqq \mathbb{P} \circ (u,Q^{1/2}W)^{-1}.
\end{align*}
Of course $\mathbb{Q}^n$ weakly converges towards $\mathbb{Q}$ as $n \to \infty$.
Let $\mathcal{A}$ be the $\mathbb{Q}$-completion of $\mathcal{B}$, and let $\{\mathcal{A}_t\}_{t \in [0,T]}$ be the smallest filtration of $\mathcal{A}$ that satisfies the usual conditions with respect to $\mathbb{Q}$ and such that the coordinate process $(\xi,\omega)$ on $\mathcal{X}$ is adapted. Introduce $\mathcal{A}^n$ and $\{\mathcal{A}^n_t\}_{t \in [0,T]}$ similarly.
Define the process
\begin{align} \label{eq:rho}
\rho_t 
\coloneqq 
\varphi(\xi_t)-\varphi(\xi_0)-\int_0^t \mathscr{L}^0 \varphi(\xi_s) ds, \quad t \in [0,T].
\end{align}
Let us show that $\rho_t$ is a continuous function of $\xi$.
 
First of all, notice that every $\varphi \in F$, $\varphi(\xi)=\langle \xi,h \rangle$ for some $h \in \mathscr{S}$, is a continuous function from $H^{-\beta}$ to $\mathbb{R}$. 
Therefore if $\xi^n \to \xi$ in $C([0,T],H^{-\beta})$ we have $\varphi(\xi^n) \to \varphi(\xi)$ in $C([0,T])$ as well.
Let us now consider the term involving the effective generator $\mathscr{L}^0$.
Recall
\begin{align*}
\mathscr{L}^0 \varphi (\xi)
=
\langle A\xi + b(\xi,\xi), h \rangle
+
\int_H \psi_\xi (w) d\mu(w).
\end{align*}
Let us show that the map $L^2([0,T],H) \ni \xi \mapsto \int_0^\cdot \mathscr{L}^0 \varphi(\xi_s) ds \in C([0,T])$ is sequentially continuous, or equivalently
\begin{align} \label{eq:conv_L0}
\int_0^t \mathscr{L}^0 \varphi(\xi^n_s) ds
\to
\int_0^t \mathscr{L}^0 \varphi(\xi_s) ds
\end{align}
in $C([0,T])$. 
For every $s \in [0,t]$, rewrite
\begin{align*}
\mathscr{L}^0 \varphi(\xi^n_s) -  \mathscr{L}^0 \varphi(\xi_s)
&=
\langle A\xi^n_s + b(\xi^n_s,\xi^n_s), h \rangle
-
\langle A\xi_s + b(\xi_s,\xi_s), h \rangle
\\
&\quad+
\int_H \psi_{\xi^n_s} (w) d\mu(w)
-
\int_H \psi_{\xi_s} (w) d\mu(w).
\end{align*}
Let us bound the terms in the right-hand-side of the previous expression separately. Making use of the usual estimates on $b$, we have
\begin{align*}
|\langle A(\xi^n_s-\xi_s) , h \rangle|
&\leq
\|\xi^n_s-\xi_s\|_H \|h\|_{H^2};
\\
|\langle b(\xi^n_s,\xi^n_s) - b(\xi_s,\xi_s), h \rangle|
&\leq
|\langle b(\xi^n_s,\xi^n_s-\xi_s), h \rangle|
+
|\langle b(\xi^n_s-\xi_s,\xi_s), h \rangle|
\\
&\lesssim (\|\xi^n_s\|_H+\|\xi_s\|_H) \|\xi^n_s-\xi_s\|_H \|h\|_{H^{\theta_0}} ;
\\
\left| \int_H \psi_{\xi^n_s} (w) d\mu(w)
-
\int_H \psi_{\xi_s} (w) d\mu(w) \right|
&\lesssim
\|h\|_{H^{\theta_1}} \|\xi^n_s-\xi_s\|_H \int_H \|w\|_{H^{\theta_0}}^2 d\mu(w) \\
&\lesssim
\|h\|_{H^{\theta_1}} \|\xi^n_s-\xi_s\|_H .
\end{align*}
Putting all together, we finally obtain the following bound
\begin{align*} 
|\mathscr{L}^0 \varphi(\xi^n_s) -  \mathscr{L}^0 \varphi(\xi_s)|
&\lesssim
(1+\|\xi^n_s\|_H+\|\xi_s\|_H)
\|\xi^n_s-\xi_s\|_H.
\end{align*}
In particular, recalling that $\xi^n \to \xi$ in $L^2([0,T],H)$ we have:
\begin{align*}
&\sup_{t \in [0,T]} \left| \int_0^t \mathscr{L}^0 \varphi(\xi^n_s) ds -  \int_0^t \mathscr{L}^0 \varphi(\xi_s) ds \right|
\leq
\int_0^T |\mathscr{L}^0 \varphi(\xi^n_s) -  \mathscr{L}^0 \varphi(\xi_s)| ds
\\
&\qquad\lesssim 
\left( \int_0^T (1+\|\xi^n_s\|_H+\|\xi_s\|_H)^2 ds \right)^{1/2}
\left( \int_0^T \|\xi^n_s-\xi_s\|_H^2 ds \right)^{1/2}\to 0.
\end{align*}

\emph{Step 3}.
By weak convergence $\mathbb{Q}^n \to \mathbb{Q}$ and previous steps it is easy to show (cf. for instance \cite[Theorem 3.1]{FlGa95} or \cite[Chapter 8.4]{DPZa14}) that 
the couple $(\rho,\omega)$ is a continuous square-integrable martingale on $(\mathcal{X},\mathcal{A},\{\mathcal{A}_t\}_{t \in [0,T]}, \mathbb{Q})$ with quadratic covariations ($\{e_k\}_{k \in \N}$ is a complete orthonormal system of $H$):
\begin{align*}
[\rho,\rho ]_t 
&= 
\int_0^t \|Q^{1/2} D_y \varphi_1(\xi_s) \|_H^2 ds,
\\
[\rho,\langle e_k,\omega\rangle]_t 
&= 
\int_0^t \langle Q^{1/2} D_y \varphi_1(\xi_s) , Q^{1/2}e_k \rangle ds,
\\
[\langle e_k,\omega\rangle , \langle e_k,\omega\rangle]_t 
&= 
\langle Q^{1/2} e_k, Q^{1/2} e_k \rangle t.
\end{align*}
This is basically due to the fact that $\rho$ can be written as the sum of a martingale on $(\mathcal{X},\mathcal{A},\{\mathcal{A}^n_t\}_{t \in [0,T]}, \mathbb{Q}^n)$ plus remainder terms which are infinitesimal in mean square as $n \to \infty$.
By \cite[Theorem 8.2]{DPZa14}, up to a possible enlargement of the underlying probability space, there exists a cylindrical Wiener process $\tilde{\omega}$ on $(\mathcal{X},\mathcal{A},\{\mathcal{A}_t\}_{t \in [0,T]}, \mathbb{Q})$ such that the following martingale representation formulae hold $\mathbb{Q}$-almost surely for every $t \in [0,T]$:
\begin{align*}
\omega_t 
&= 
\int_0^t Q^{1/2} d\tilde{\omega}_s = Q^{1/2} \tilde{\omega}_t
,
\\
\rho_t 
&= 
\int_0^t \langle D_y \varphi_1(\xi_s) , Q^{1/2} d\tilde{\omega}_s \rangle
=
\int_0^t \langle D_y \varphi_1(\xi_s) , d\omega_s \rangle.
\end{align*}
In particular, since the auxiliary Wiener process $\tilde{\omega}$ satisfies $\omega_t = Q^{1/2} \tilde{\omega}_t$, the equation for $\rho_t$ above holds true also in the original probability space, without necessarily taking an enlargement thereof.
Thus, recalling \eqref{eq:rho} we have the following $\mathbb{Q}$-almost sure identity on the path space $\mathcal{X}$:
\begin{align*}
\varphi(\xi_t)
=
\varphi(\xi_0)
+
\int_0^t \mathscr{L}^0 \varphi(\xi_s) ds
+
\int_0^t \langle D_y \varphi_1(\xi_s) , d\omega_s \rangle,
\quad
\forall t \in [0,T],
\end{align*}
that by $\mathbb{Q} = \mathbb{P} \circ (u,Q^{1/2}W)^{-1}$ and the explicit expression of $\varphi_1$ \eqref{eq:varphi1_lim} is equivalent to our thesis.

\end{proof}

\subsection{It\=o-Stokes drift and Stratonovich corrector}

In this subsection, we provide an interpretation of the limiting equation in terms of different contribution to the dynamics. 
Recall that every weak accumulation point $u$ of the family $\{u^\epsilon_t\}_{\epsilon\in(0,1)}$ satisfies, for every $\varphi \in F$, $\varphi(u)=\langle u,h \rangle$ for some $h \in \mathscr{S}$, $t \in [0,T]$ the almost sure identity \eqref{eq:u}: 
\begin{align*} 
\varphi(u_t) 
&= 
\varphi(u_0)
+
\int_0^t \mathscr{L}^0 \varphi(u_s) ds
+
\int_0^t \langle b((-C)^{-1} Q^{1/2} dW_s, u_s), h \rangle,
\end{align*} 
where the limiting effective generator $\mathscr{L}^0$ is given by \eqref{eq:L0_lim}. For the reader's convenience, here we rewrite $\mathscr{L}^0 \varphi$ more explicitly as
\begin{align*}
\mathscr{L}^0 \varphi(u)
&=
\langle Au+b(u,u) , h \rangle
+
\int_H \langle b(w,u),D_u \varphi_1(w) \rangle d\mu(w)
+
\int_H \langle b(w,w), D_y \varphi_1(u) \rangle d\mu(w)
\\
&=
\langle Au+b(u,u) , h \rangle
+
\int_H \langle b((-C)^{-1}w,b(w,u)),h \rangle d\mu(w)
\\
&\quad+
\int_H \langle b((-C)^{-1}b(w,w), u),h \rangle d\mu(w).
\end{align*}

Let us compare \eqref{eq:u} with the dynamics of $u^\epsilon$, $\epsilon\in(0,1)$.
Of course, the term $\langle Au_s + b(u_s,u_s), h \rangle$ reflects the deterministic dynamics $\langle Au^\epsilon_s + b(u^\epsilon_s,u^\epsilon_s), h \rangle$ in the evolution of $u^\epsilon$.

On the other hand, the fast-oscillating term $\epsilon^{-1/2}\langle b(y^\epsilon_s,u^\epsilon_s), h \rangle$ in the equation for $u^\epsilon$ is responsible for the additional terms in the limit.
We can distinguish three different contributions:
\begin{itemize}
\item
The It\=o integral:
\begin{align} \label{eq:ito_integral}
\int_0^t \langle b((-C)^{-1} Q^{1/2} dW_s, u_s), h \rangle;
\end{align}
\item
The Stratonovich corrector:
\begin{align} \label{eq:corrector1}
\int_0^t \int_H \langle b((-C)^{-1}w,b(w,u_s)),h \rangle d\mu(w) ds;
\end{align}
\item
The It\=o-Stokes drift:
\begin{align}\label{e5.14}
\int_0^t \int_H \langle b((-C)^{-1}b(w,w), u_s),h \rangle d\mu(w)ds.
\end{align}
\end{itemize}

Of course the term denoted "It\=o-Stokes drift" equals $\int_0^t \langle b(r,u_s) ,h \rangle ds$ by the very definition of the It\=o-Stokes drift velocity $r = \int_H (-C)^{-1} b(w,w) d\mu(w)$.
Moreover, we shall see that the term called "Stratonovich corrector" \eqref{eq:corrector1} above is indeed the Stratonovich corrector of the term called "Ito integral" \eqref{eq:ito_integral}, namely:
\begin{align} \label{eq:strat}
\int_0^t \langle D_y \varphi_1(u_s) , Q^{1/2} \circ dW_s \rangle
&=
\int_0^t \langle D_y \varphi_1(u_s) , Q^{1/2} dW_s \rangle
\\ \nonumber
&\quad+
\int_0^t \int_H
\langle b(w,u_s) , D_u \varphi_1(w) \rangle d\mu(w)ds.
\end{align}
Therefore, \eqref{eq:u} is exactly the weak formulation of the equation in \autoref{thm:main}, which is then proved.

We are left to check the validity of \eqref{eq:strat}. Let $h \in \mathscr{S}$ be fixed, and let $W = \sum_k e_k W^k$, where $\{e_k\}_{k \in \N}$ is a complete orthonormal system in $H$ and $\{W^k\}_{k \in \N}$ is a family of one-dimensional i.i.d. Wiener processes.
As a matter of fact, one can rewrite the It\=o integral \eqref{eq:ito_integral} as
\begin{align*}
\int_0^t \langle b((-C)^{-1}Q^{1/2} dW_s, u_s) ,h  \rangle
&=
-\sum_{k \in \N}
\int_0^t \langle b((-C)^{-1}Q^{1/2} e_k,h) , u_s  \rangle dW^k_s;
\end{align*}
since for $h \in \mathscr{S}$ it holds $b((-C)^{-1}Q^{1/2} e_k,h) \in \mathscr{S}$ as well, the quadratic variation between the processes $\langle b((-C)^{-1}Q^{1/2} e_k,h) , u \rangle$ and $W^k$ is given by
\begin{align} \label{eq:aux03}
\left[ \langle b((-C)^{-1}Q^{1/2} e_k,h) , u \rangle , W^k \right]_t
&=
- \int_0^t \langle b((-C)^{-1} Q^{1/2} e_k, u_s) , b((-C)^{-1}Q^{1/2} e_k,h) \rangle ds.
\end{align}

On the other hand, \eqref{eq:corrector1} equals
\begin{align} \label{eq:aux04}
\int_0^t \int_H
\langle b((-C)^{-1} w , b(w,u_s)) , h \rangle d\mu(w) ds
&=
\sum_{k \in \N} 
\int_0^t \langle b((-C)^{-1} Q^{1/2}_\infty e_k , b(Q^{1/2}_\infty e_k,u_s)) , h \rangle ds
\\
&=
-\sum_{k \in \N}  \nonumber
\int_0^t \langle b((-C)^{-1} Q^{1/2}_\infty e_k , h) , b(Q^{1/2}_\infty e_k,u_s) \rangle ds.
\end{align} 

Recall that, since $C$ and $Q$ commute, the covariance operator $Q_\infty$ of the invariant measure $\mu=\mathcal{N}(0,Q_\infty)$ can be written as $Q_\infty = \frac{1}{2} (-C)^{-1}Q$.
In particular, there exists a complete orthonormal system $\{ e_k \}_{k \in \N}$ of $H$ that diagonalizes $C$, $Q$ and $Q_\infty$ simultaneously, namely
\begin{align*}
C e_k = -\lambda_k e_k,
\quad
Q e_k = q_k e_k,
\quad
Q_\infty e_k = \frac{q_k}{2\lambda_k} e_k,
\end{align*}
and therefore by \eqref{eq:aux03} and \eqref{eq:aux04} we finally get
\begin{align*}
\frac12 \left[ \langle b((-C)^{-1}Q^{1/2} e_k,h) , u \rangle , W^k \right]_t
=
\int_0^t \int_H
\langle b((-C)^{-1} w , b(w,u_s)) , h \rangle d\mu(w) ds.
\end{align*}

\begin{rmk} \label{rmk:commuting+stokes}
\emph{i})
The attentive reader would have noticed that this is the first time we are actually using that $C$ and $Q$ commute.
In particular, the convergence result described in the previous sections still holds true in the case $CQ \neq QC$. In this case, we believe that \begin{align*}
\sum_{k \in \N} 
\int_0^t &\langle b((-C)^{-1} Q^{1/2} e_k, u_s) , b((-C)^{-1}Q^{1/2} e_k,h) \rangle ds
\\
&\neq
\sum_{k \in \N}  \nonumber
\int_0^t \langle b((-C)^{-1} Q^{1/2}_\infty e_k , h) , b(Q^{1/2}_\infty e_k,u_s) \rangle ds
\end{align*}
and the limit equation cannot be interpreted as an equation with Strotonovich transport noise. 

\emph{ii})
Assume that $Q$ is isotropic, i.e. every basis of eigenvectors of $A$ also diagonalizes $Q$. Under this circumstance, the It\=o-Stokes drift equals zero for the Navier-Stokes system\footnote{Similar conditions on $Q$ can be found for other models, as Primitive Equations and Surface Quasi-Geostrophic equations.}, since for the particular choice 
\begin{align*}
e_{\mathbf{k},i} = a_{\mathbf{k},i} \cos(2 \pi \mathbf{k} \cdot x),
\quad \mbox{or }
e_{\mathbf{k},i} = a_{\mathbf{k},i} \sin(2 \pi \mathbf{k} \cdot x),
\end{align*}
where $\mathbf{k} \in \mathbb{Z}^d \setminus \{\textbf{0}\}$ ,
$i = 1,\dots,d-1$, and $a_{\mathbf{k},i} \in \mathbf{k}^\perp$ for every $\mathbf{k},i$, it holds
\begin{align*}
b(e_{\mathbf{k},i},e_{\mathbf{k},i})
&=
-\Pi \left( (e_{\mathbf{k},i} \cdot \nabla) e_{\mathbf{k},i}\right)
=
\pm \Pi \left( \underbrace{(a_{\mathbf{k},i} \cdot 2 \pi \mathbf{k})}_{=0}  \cos(2 \pi \mathbf{k} \cdot x)\sin(2 \pi \mathbf{k} \cdot x)\right) = 0.
\end{align*} 
\end{rmk}

\section{Eddy viscosity and convergence to deterministic Navier-Stokes equations} \label{sec:eddy}

In this section we restrict our attention to the Navier-Stokes system, and discuss the interesting problem of a varying covariance operator $Q=Q_N$, where $N \in \N$ is a parameter possibly dependent on the scaling parameter $\epsilon\in(0,1)$.
We shall identify conditions under which it is possible to prove the convergence to a deterministic process $u$, that solves the Navier-Stokes equations with an additional dissipative term:
\begin{align} \label{eq:u_bis}
du_t
&= 
A u_t dt + b(u_t,u_t) dt+ \kappa(u_t)dt,
\end{align}
for some (possibly unbounded) linear operator $\kappa : D(\kappa) \subset H \to H$, and the convergence is meant in the sense of \autoref{prop:identification}.

It is well-known from experiments that, in fluids, turbulent regimes manifest a dissipation enhancement with respect to laminar flows. 
Following the literature on the topic we refer to the operator $\kappa$ as \emph{eddy viscosity}. By our knowledge this is the first rigorous result showing that eddy viscosity in 3D Navier-Stokes can be deduced from additive noise, acting on idealized small scales of the fluid itself, in the limit of infinite separation of scales (cf. \cite{Pa21+} for a similar result on transport equation).
In the SPDE community similar results have been proved for equations with Stratonovich transport noise, taking advantage of the explicit  expression for the Stratonovich-to-It\=o corrector and a suitable scaling limit to show that the It\=o integral vanishes in the limit, see \cite{FlLu20,FlGaLu22,FlGaLu21+}. 
However, these results are mostly restricted to the two-dimensional case. 
The only result in dimension three we are aware of is \cite{FlLu21}, where Navier-Stokes equations in vorticity form are investigated: however, by admission of the same authors, the noise used in this paper lacks of physical meaning, since the vortex stretching term is neglected for technical limitations.
Here, we borrow their original ideas of scaling parameters in such a way that, in the limit, the It\=o integral does vanish whereas the Stratonovich corrector does not.

Having said this, we consider a modification of \eqref{eq:system_2} where the covariance operator depends on a parameter $N \in \N$:
\begin{align} \label{eq:system_N}
\begin{cases}
du^{\epsilon,N}_t = Au^{\epsilon,N}_t dt + b(u^{\epsilon,N}_t,u^{\epsilon,N}_t) dt + \epsilon^{-1/2} b(y^{\epsilon,N}_t,u^{\epsilon,N}_t) dt,
\\
dy^{\epsilon,N}_t = \epsilon^{-1} Cy^{\epsilon,N}_t dt +
Ay^{\epsilon,N}_t dt + b(u^{\epsilon,N}_t,y^{\epsilon,N}_t) dt + \epsilon^{-1/2}b(y^{\epsilon,N}_t,y^{\epsilon,N}_t) dt + \epsilon^{-1/2} Q_N^{1/2} dW_t.
\end{cases}
\end{align}

We assume that $Q_N$ satisfies assumptions (Q1)-(Q2) for every fixed $N>0$. 
It is easy to show that for every fixed $N$ there exists a bounded-energy family $\{(u^{\epsilon,N},y^{\epsilon,N})\}_{\epsilon\in(0,1)}$ of weak martingale solutions to \eqref{eq:system_N}. 
As a consequence, results of the previous sections hold true for the families $\{u^{\epsilon,N}\}_{\epsilon\in(0,1)}$, $N$ fixed, so that for every $N \in \N$ there exists a subsequence $\epsilon_n \to 0$ such that $u^{\epsilon_n,N} \to u^N$ in law as $n \to \infty$, where $u^N$ satisfies for every $\varphi \in F$, $t \in [0,T]$ the almost sure identity
\begin{align} \label{eq:003}
\varphi(u^N_t) 
&= 
\varphi(u_0)
+
\int_0^t \mathscr{L}^{0,N} \varphi(u^N_s) ds
+
\int_0^t \langle D_y \varphi_1(u^N_s) , Q^{1/2}_N dW^N_s \rangle,
\end{align} 
with ($\mu^N$ below denotes the invariant measure of $dv=Cv dt + Q_N^{1/2} dW_t$)
\begin{align*}
\mathscr{L}^{0,N} \varphi (u)
&=
\langle Au + b(u,u), h \rangle
+
\int_H \psi_u (w) d\mu^N(w).
\end{align*}
In addition, the family $\{u^{\epsilon,N}\}_{\epsilon\in(0,1),N \in \N}$ can be choosen in such a way that energy bound (S3) holds uniformly in $\epsilon\in(0,1)$ and in $N \in \N$ (energy estimates on $u^{\epsilon,N}$ do not involve the covariance operator $Q_N$).

When $Q_N$ is isotropic, we have seen in part \emph{ii}) of \autoref{rmk:commuting+stokes} that the It\=o-Stokes drift vanishes.
As for the Stratonovich corrector, recall $Q_{N,\infty} = \frac{1}{2} (-C)^{-1}Q_N$ and there exists a simultaneous eigenbasis $\{e_k\}_{k \in\N}$ of $C$ and $Q_N$ such that $C e_k = -\lambda_k e_k$ and $Q_N e_k = q_{N,k} e_k$, and therefore the candidate operator $\kappa$ is formally given by
\begin{align} \label{eq:kappa}
\kappa(u)
&\coloneqq
\lim_{N \to \infty}
\kappa_N(u)
\coloneqq
\lim_{N \to \infty}
\sum_{k \in \N}
\frac{q_{N,k}}{2 \lambda_k^2}
b(e_k,b(e_k,u)) 
\end{align}
and it is easy to check by (B4) that $\kappa$ is symmetric and
\begin{align*}
\langle \kappa(u),u \rangle 
&=
\lim_{N \to \infty}
\sum_{k \in \N}
\frac{q_{N,k}}{2 \lambda_k^2}
\langle b(e_k,b(e_k,u)) , u \rangle
=
-\lim_{N \to \infty}
\sum_{k \in \N}
\frac{q_{N,k}}{2 \lambda_k^2} \|b(e_k,u)\|_H^2.
\end{align*}

\begin{thm} \label{thm:deterministic_limit}
Let $Q_N$ satisfy (Q1)-(Q2) for every fixed $N \in \N$. Suppose $Q_N$ is isotropic, and there exists some $\sigma_0<\infty$ such that for every $N \in \N$ the operator $\kappa_N$ given by \eqref{eq:kappa} is well defined and continuous from $H^{\sigma_0}$ to $H$, and for every fixed $u \in H^{\sigma_0}$ it holds $\kappa_N(u) \to \kappa(u)$ in $H$ as $N \to \infty$.
Finally, assume 
\begin{align*}
\lim_{N \to \infty} \sum_{k \in\N} \|Q_N^{1/2} e_k\|_{H^{-1-2\gamma}}^2 = 0.
\end{align*}
Then for every $\beta > 0$ the laws of the processes $\{u^{\epsilon,N}\}_{\epsilon\in(0,1), N \in \N}$ are tight as probability measures on the space $L^2([0,T],H)$ $\cap$ $C([0,T],H^{-\beta})$, and every weak accumulation point $u$ of the family $\{u^{\epsilon,N}\}_{\epsilon\in(0,1), N \in \N}$ as $\epsilon \to 0$ and then $N \to \infty$ is a weak solution of \eqref{eq:u_bis}, meaning that for every $t \in [0,T]$ and $h \in \mathscr{S}$ it holds almost surely
\begin{align*} 
\langle u_t , h \rangle
=
\langle u_0 , h \rangle
+
\int_0^t \langle u_s , A h\rangle ds 
+ 
\int_0^t \langle b(u_s,u_s) , h \rangle ds 
+ 
\int_0^t \langle u_s , \kappa(h) \rangle ds.
\end{align*}

Assume in addition $u_0 \in H^1$ and the existence of $\kappa_0$ sufficiently large such that $\langle \kappa(u),u \rangle \leq -\kappa_0 \|u\|_{H^1}^2$ for every $u \in H^1$.
Then $u \in L^\infty([0,T],H^1) \cap L^2([0,T],H^2)$ is the unique strong solution of \eqref{eq:u_bis} on the time interval $[0,T]$, and the whole sequence $\{u^{\epsilon,N}\}_{\epsilon\in(0,1), N \in \N}$ converges in law to $u$ as $\epsilon \to 0$ and $N \to \infty$. 
\end{thm}
\begin{rmk}
A careful inspection of the convergence proof given in 
\autoref{s5.2} shows that we have the same result for the family $\{u^{\epsilon,N_\epsilon}\}_{\epsilon \in (0,1)}$ when $\epsilon \to 0$ and $N=N_\epsilon$ satisfies the condition that there exists $\delta \in (0,\gamma-1/4)$ such that 
$$
\epsilon^{{\delta}/{2(1-\gamma)}} Tr((-A)^{\theta_0/2} (-C_\epsilon)^{-1} Q_{N_\epsilon}^{1/2}) \to 0.
$$
\end{rmk}
\begin{proof}

By previous \autoref{thm:main} for every $N \in \N$ the family $\{u^{\epsilon,N}\}_{\epsilon\in(0,1)}$ converges in law as processes taking values in $L^2([0,T],H) \cap C([0,T],H^{-\beta})$ towards a weak martingale solution $u^N$ of
\begin{align*}
d u^N_t 
&=
A u^N_t + b(u^N_t,u^N_t) dt 
+
b((-C)^{-1}Q_N^{1/2} \circ dW^N_t,u^N_t)
\\
&= 
A u^N_t + b(u^N_t,u^N_t) dt + \kappa_N(u^N_t) dt
+
b((-C)^{-1}Q_N^{1/2} dW^N_t,u^N_t).
\end{align*}

The It\=o integral above vanishes when tested against smooth functions in the limit $N \to \infty$: indeed, for every given $h \in \mathscr{S}$ and $s,t \in [0,T]$, $s<t$ it holds
\begin{align*}
\mathbb{E} \left[ \left( \int_s^t  \langle b((-C)^{-1}Q_N^{1/2} dW^N_t,u^N_t) , h \rangle \right)^2 \right]
&\lesssim
|t-s|
\| h \|_{H^{\theta_1}}^2 
\sum_{k \in\N} \|Q_N^{1/2} e_k\|_{H^{-1-2\gamma}}^2  
 \to 0.
\end{align*}

In addition, by Simon compactness criterium the family $\{u^N\}_{N \in \N}$ is tight in $L^2([0,T],H)$ $\cap$ $C([0,T],H^{-\beta})$ since, on the one hand, uniform bounds in $N$ for the expectation of the norm of $u^N$ in  $L^\infty([0,T],H)$ $\cap$ $L^2([0,T],H^1)$ are inherited from the approximating family $\{u^{\epsilon,N}\}_{\epsilon\in(0,1),N \in \N}$ (cf. \autoref{rmk:selection,energy bounds}\emph{ii})), whereas by our assumption on $\kappa$ and Banach-Steinhaus Theorem it holds $\sup_{N \in \N}\| \kappa_N \|_{H^{\sigma_0} \to H} < \infty$ and thus
\begin{align*}
\sup_{N \in \N} \mathbb{E} \left[
\|u^N\|_{W^{\alpha,p}([0,T],H^{-\sigma})}^p \right]
\lesssim
1+\sup_{N \in \N}
\int_0^T \int_0^T \frac{
\mathbb{E}\left[\| u^N_t-u^N_s \|_{H^{-\sigma}}^p\right]}
{|t-s|^{1+\alpha p}} < \infty
\end{align*}  
for some $\sigma$ sufficiently large, $p>2$ and $\alpha > 1/p$.

To summarise the previous steps, by a diagonal argument there exists a subsequence $\{(\epsilon_n,N_n)\}_{n \in \N}$, $\epsilon_n \to 0$, $N_n \to \infty$ such that $u^{\epsilon_n,N_n}$ converges weakly to some random variable $u$ taking values in $L^2([0,T],H) \cap C([0,T],H^{-\beta})$. Moreover, arguing as in the proof of \autoref{prop:identification} it is easy to show that $u$ is almost surely a weak solution of the deterministic equation \eqref{eq:u_bis} (notice that $u$ could be random nonetheless, because of the lack of uniqueness in the limit equation).

Finally, under the stronger assumptions on $u_0$ and $\kappa$ it is well-known \cite{Te95} that Navier-Stokes equations admits a unique strong solution on $[0,T]$, where $T=T(u_0,\nu,\kappa_0)$ can be made arbitrary large taking $\kappa_0$ sufficiently large. By weak-strong uniqueness $u$ is the strong solution, and therefore the whole sequence $\{u^{\epsilon,N}\}_{\epsilon\in(0,1),N \in \N}$ converges since its limit does not depend on the subsequence. 
\end{proof}

\begin{rmk}
As an example we recall the scaling limit construction in \cite{FlLu21}. Let $d=3$, $\mathbb{Z}^d_+ \cup \mathbb{Z}^d_-$ be a partition of $\mathbb{Z}^d_0 \coloneqq  \mathbb{Z}^d \setminus \{\textbf{0}\}$ such that $\mathbb{Z}^d_+=-\mathbb{Z}^d_-$, $\textbf{k} \in \mathbb{Z}^d_0 $, $i=1,2$ and $\{a_{\textbf{k},i}\}_{i=1,2}$ be an orthonormal basis of $\mathbf{k}^\perp$ such that $\{a_{\textbf{k},1},a_{\textbf{k},2},\textbf{k}^\perp/|\textbf{k}|\}$ is right-handed and $a_{-\textbf{k},i} = -a_{\textbf{k},i}$ for every $\textbf{k} \in \mathbb{Z}^d_+$.
Introducing the Fourier basis $e_{\textbf{k},i}(x) \coloneqq \sqrt{2} a_{\textbf{k},i}\cos(2\pi \textbf{k}\cdot x)$ for $\textbf{k} \in \mathbb{Z}^d_+$, and $e_{\textbf{k},i}(x) \coloneqq \sqrt{2} a_{\textbf{k},i}\sin(2\pi \textbf{k}\cdot x)$ for $\textbf{k} \in \mathbb{Z}^d_-$, and the covariance operators $Q_N$ given by
\begin{align*}
Q_N^{1/2} e_{\textbf{k},i}
\coloneqq
C_{\kappa_0} \left(\sum_{N \leq |\textbf{k}| \leq 2N} \frac{1}{|\textbf{k}|^{2\delta}}\right)^{-1/2} \frac{e_{\textbf{k},i}}{|\textbf{k}|^\delta} \mathbf{1}_{\{N \leq |\textbf{k}| \leq 2N\}},
\quad
\delta > 0,
\end{align*}
we are in the setting of \autoref{thm:deterministic_limit} with limit eddy viscosity $\kappa$ given by a multiple of the Laplace operator: $\kappa(u) = C_{\kappa_0}' \Delta u$, for every $u \in H^2$, cf. \cite[Theorem 5.1]{FlLu21}. 
Moreover, for every $\kappa_0>0$ the constant $C_{\kappa_0}$ can be chosen in such a way that $\langle \kappa(u),u \rangle \leq -\kappa_0 \|u\|_{H^1}^2$ for every $u \in H^1$, so that the second part of our theorem holds, too.
\end{rmk}

\section{Other models} \label{sec:models}
In the main body of this paper our principal focus was on the system of 3D Navier-Stokes equations.
In this section we present instead other models of interest, namely the 2D Navier-Stokes equations with strong friction and unitary viscosity at small scales, the 2D Surface Quasi-Geostrophic equations and the Primitive equations.
We shall see how the proof of \autoref{thm:main} (as well as the preliminary work needed in the proof) can be adapted to the aforementioned models with little changes in the arguments.
Notice that the results in the previous \autoref{sec:eddy} will not be extended to other models, but remain a peculiarity of the Navier-Stokes equations.

\subsection{2D Navier-Stokes with pure friction} \label{ssec:2DNS}

In the 2D case, we can weaken the assumptions on the operator $C$. 
Unless some additional dissipation mechanism is present at small scales, assumption (C2) is not completely satisfactory and we would therefore consider a pure friction, {\it i.e.} $C=-\chi Id$ for some $\chi>0$.  

For instance, comparing our model with \cite{FlPa22+}, where the authors study $L^\infty$ solutions of 2D Navier-Stokes equations in vorticity form, it could be sensible and physically meaningful to consider instead the system 
\begin{align} \label{eq:2DNS_eps}
\begin{cases}
du^\epsilon_t 
= 
\nu \Delta u^\epsilon_t dt
-
\Pi(u^\epsilon_t \cdot \nabla) u^\epsilon_t dt
-
\epsilon^{-1/2}\Pi(y^\epsilon_t \cdot \nabla) u^\epsilon_t dt,
\\
dy^\epsilon_t 
= 
\nu \Delta y^\epsilon_t dt
-
\chi \epsilon^{-1} y^\epsilon_t dt
-
\Pi(u^\epsilon_t \cdot \nabla) y^\epsilon_t dt
-
\epsilon^{-1/2}\Pi(y^\epsilon_t \cdot \nabla) y^\epsilon_t dt
+
\epsilon^{-1/2}
d\Pi \mathcal{W}_t,
\end{cases}
\end{align}
where the large dissipation has the form of a friction term ($C=-\chi Id$). 
The previous equations are posed in the two-dimensional torus $\mathbb{T}^2$, and take the abstract form
\begin{align*}
\begin{cases}
du^\epsilon_t 
= 
A u^\epsilon_t dt
+
b(u^\epsilon_t,u^\epsilon_t) dt
+
\epsilon^{-1/2}b(y^\epsilon_t,u^\epsilon_t) dt,
\\
dy^\epsilon_t 
= 
A y^\epsilon_t dt
-
\chi \epsilon^{-1} y^\epsilon_t dt
+
b(u^\epsilon_t,y^\epsilon_t) dt
+
\epsilon^{-1/2}b(y^\epsilon_t,y^\epsilon_t) dt
+
\epsilon^{-1/2}
Q^{1/2}dW_t.
\end{cases}
\end{align*}

Therefore, in this section, we restrict to the $2D$ case and prove that our result indeed extend to a pure friction provided a smoother initial data and noise are taken.

In this setting, our analysis revolves around the modified $\epsilon$-dependent Ornstein-Uhlenbeck semigroup associated to the small-scale equation
\begin{align*}
dY^y_t 
=
(\epsilon A - \chi) Y^y_t dt 
+
Q^{1/2} dW_t,
\quad
Y^y_0 = y.
\end{align*}

Assume (Q1), and replace (Q2) with the assumption $\int_H \|w\|^2_{H^4} d\mathcal{N}(0,Q)(w) < \infty$. Moreover, let the initial datum $y_0 \in H^1$ (whereas $u_0 \in H$ as usual).
Then we can prove the anologous of \autoref{thm:main}, namely
\begin{thm}
Under the previous assumptions, let $\{(u^\epsilon,y^\epsilon)\}_{\epsilon \in (0,1)}$ be a bounded-energy family of probabilistically strong solutions to \eqref{eq:2DNS_eps} on a common stochastic basis $(\Omega, \mathcal{F}, \{\mathcal{F}_t \}_{t \geq 0}, \mathbb{P}, W)$, which exists by \autoref{prop:existence} and pathwise uniqueness.
Then, for every $\beta>0$, $u^\epsilon$ converges in $L^2([0,T],H) \cap C([0,T],H^{-\beta})$ in probability as $\epsilon \to 0$ towards the unique solution $u$ of the limiting equation (here the It\=o-Stokes drift velocity $r$ is defined as $r=\chi^{-1}\int b(w,w)d\mu(w)$, and $\mu=\mathcal{N}(0,Q_\infty)$ with $Q_\infty=(2\chi)^{-1}Q$)
\begin{align*}
du_t 
&= 
Au_t dt + b(u_t,u_t) dt +\chi^{-1} b(Q^{1/2}\circ dW_t,u_t) +  b(r,u_t) dt
\\
&=
\nu \Delta u_t dt 
-
\Pi (u_t \cdot \nabla) u_t dt
-
\chi^{-1} \Pi (Q^{1/2}\circ dW_t \cdot \nabla) u_t
-
\Pi (r \cdot \nabla) u_t dt.
\end{align*}
\end{thm}

Here we limit ourselves to point out the major modifications needed to prove the result in the present setting (for simplicity we take $\chi=1$ in the following):
\\
$\bullet$
The thesis of \autoref{lem:int_y} changes into
\begin{align*}
\sup_{\substack{\epsilon\in(0,1),\\n \in\N}} \sup_{t \in [0,T]}
\int_0^t \mathbb{E} \left[ \|y^{\epsilon,n}_s\|^{2p}_H \right] ds
\lesssim 1,
\end{align*} 
with minor modifications in the proof; notice that this is still sufficient to prove the analogous of \autoref{lem:E_y} in the present setting;

$\bullet$ 
Making use of the key identity $\langle b(u,u),Au \rangle = 0$ (cf. \cite[Lemma 3.1]{Te95}), one can improve \autoref{prop:y-Y} in the following way: applying It\=o Formula to $\|\zeta^{\epsilon,n}_t\|_{H^1}^2 = \langle \zeta^{\epsilon,n}_t , (-A)\zeta^{\epsilon,n}_t \rangle$, and using the embedding $H^{1/2} \subset L^4$ and Young inequality, one gets
\begin{align*}
\frac12 \|\zeta^{\epsilon,n}_t\|_{H^1}^2
&-
\frac12 \|y_0\|_{H^1}^2
+
\int_0^t \|\zeta^{\epsilon,n}_s\|_{H^2}^2 ds
+
\epsilon^{-1}
\int_0^t \|\zeta^{\epsilon,n}_t\|_{H^1}^2 ds
\\
&=
\int_0^t \langle b(u^{\epsilon,n}_s, \zeta^{\epsilon,n} + Y^{\epsilon,n}_s) , A \zeta^{\epsilon,n}_s \rangle ds
-
\epsilon^{-1/2}
\langle b(\zeta^{\epsilon,n} + Y^{\epsilon,n}_s, \zeta^{\epsilon,n} + Y^{\epsilon,n}_s) , A Y^{\epsilon,n}_s \rangle ds
\\
&\leq
\int_0^t \|u^{\epsilon,n}_s\|_{H^{1/2}} 
\|\zeta^{\epsilon,n} + Y^{\epsilon,n}_s\|_{H^{3/2}} 
\|\zeta^{\epsilon,n}_s\|_{H^2} ds
\\
&\quad+
\epsilon^{-1/2} \int_0^t
\|\zeta^{\epsilon,n} + Y^{\epsilon,n}_s\|_H
\|\zeta^{\epsilon,n} + Y^{\epsilon,n}_s\|_{H^1} 
\|Y^{\epsilon,n}_s\|_{H^4} ds
\\
&\leq
\frac12 \int_0^t \|\zeta^{\epsilon,n}_s\|_{H^2}^2 ds
+
M \int_0^t \|u^{\epsilon,n}_s\|_{H^1}^2 ds
 +
M \int_0^t \|\zeta^{\epsilon,n}_s\|_H^{16} ds
\\
&\quad
+ 
\frac12 \epsilon^{-1} \int_0^t \|\zeta^{\epsilon,n}_s\|_{H^1}^2 ds
+
M \int_0^t \|Y^{\epsilon,n}_s\|_{H^4}^3 ds,
\end{align*}
for some unimportant finite constant $M$. 
Therefore because $y_0 \in H^1$ it holds
\begin{align} \label{eq:2d_zeta}
\sup_{t \in [0,T]} \left( 
\mathbb{E} \left[ \|\zeta^{\epsilon,n}_t\|_{H^1}^2\right]
+
\int_0^t \mathbb{E} \left[ \|\zeta^{\epsilon,n}_s\|_{H^2}^2\right] ds
+
\epsilon^{-1}
\int_0^t \mathbb{E} \left[ \|\zeta^{\epsilon,n}_t\|_{H^1}^2\right] ds
\right) 
\lesssim 1.
\end{align}
Similarly, for every $p\geq 2$ one can prove
\begin{align*}
\sup_{t \in [0,T]} \left( 
\mathbb{E} \left[ \|\zeta^{\epsilon,n}_t\|_{H^1}^p\right]
+
\int_0^t \mathbb{E} \left[ \|\zeta^{\epsilon,n}_t\|_{H^1}^{p-2}\|\zeta^{\epsilon,n}_s\|_{H^2}^2\right] ds
+
\epsilon^{-1}
\int_0^t \mathbb{E} \left[ \|\zeta^{\epsilon,n}_t\|_{H^1}^p\right] ds
\right) 
\lesssim 1;
\end{align*}

$\bullet$
Thanks to bounds \eqref{eq:2d_zeta}, the following estimates for the time increments of $u^\epsilon$, $y^\epsilon$ hold true (\autoref{lem:increments}).
\begin{align*}
\mathbb{E} \left[ \| u^\epsilon_t-u^\epsilon_s \|^p_{H^{-1}} \right]
&\lesssim \epsilon^{-1/2}|t-s|^{p/2};
\\
\mathbb{E} \left[ \|Y^\epsilon_t-Y^\epsilon_s \|_H^p \right]
&\lesssim \epsilon^{-1/2}|t-s|^{p/2};
\\
\mathbb{E} \left[ \| y^\epsilon_t-y^\epsilon_s-(Y^\epsilon_t-Y^\epsilon_s) \|_H^p \right]
&\lesssim \epsilon^{-1/2}|t-s|^{p/2};
\end{align*}

$\bullet$
The results of \autoref{sec:poisson} can be obtained replacing $C_\epsilon$ with $\epsilon A- Id$. As a consequence, Poisson equation improves regularity of the datum (cf. \autoref{cor:regularity_phi} and \autoref{prop:inverse_theta}) only at the price of factors $\epsilon^{-1}$, whereas preserves regularity uniformly in $\epsilon$;

$\bullet$
The first corrector $\varphi_1^\epsilon$ is given by $\varphi_1^\epsilon(u,y)=\phi_u^\epsilon(y)=\langle b((-\epsilon A + Id)^{-1}y,u),h \rangle$. 
In particular, by interpolation we can prove the following analogous of \autoref{prop:regularity_varphi1}: for every $\theta \in [0,1]$ and $s \in \R$ there exists $\theta_1=\theta_1(s)$ such that
\begin{align*}
\|D_y\varphi_1^\epsilon(u) \|_{H^{2\theta+s}} \lesssim \epsilon^{-\theta} \|h\|_{H^{\theta_1}} \|u\|_{H^s},
\quad
\|D_u\varphi_1^\epsilon(y) \|_{H^{2\theta+s}} \lesssim \epsilon^{-\theta} \|h\|_{H^{\theta_1}} \|y\|_{H^s};
\end{align*}

$\bullet$
The linearisation trick relies on \autoref{prop:linearisation}. As usual, denote $\zeta = y-Y$. Then for every $\theta \in (1,2)$ the following bounds hold true:
\begin{align*}
|\langle b(\zeta,u),D_u\varphi_1(Y) \rangle|
&\lesssim
\|\zeta\|_{H^1} \|u\|_H \|Y\|_{H^\theta} \|h\|_{\theta_1},
\\
|\langle b(y,u),D_u \varphi_1(\zeta) \rangle|
&\lesssim
\| b(y,u)\|_{H^{-1}} \|\zeta\|_{H^1} \|h\|_{\theta_1}
\\
&\lesssim
\|y\|_{H^1} \|u\|_{H^{\theta-1}} 
\|\zeta\|_{H^1} \|h\|_{\theta_1},
\\
|\langle b(\zeta,Y), D_y \varphi_1(u) \rangle
&\lesssim
\|\zeta\|_{H^1} \|Y\|_{H^\theta} \|u\|_H \|h\|_{\theta_1},
\\
|\langle b(y,\zeta),D_y \varphi_1(u)|
&\lesssim
\|y\|_{H^1} \|\zeta\|_{H^1} \|u\|_{H^{\theta-1}}\|h\|_{\theta_1};
\end{align*}

$\bullet$
To find the second corrector $\varphi_2^\epsilon$, we apply the analogous of \autoref{prop:inverse_theta} to $\Psi_u^\epsilon = \psi_u^\epsilon- \int_H \psi_u^\epsilon(w)d\mu(w)$, where
\begin{align*}
\psi_u^\epsilon = \langle b(\cdot,u) , D_u \varphi_1^\epsilon \rangle + \langle b(\cdot,\cdot),D_y \varphi_1^\epsilon \rangle \in \mathcal{E}_\theta,
\quad
\|\psi_u^\epsilon\|_{\mathcal{E}_\theta} \lesssim \|h\|_{H^{\theta_1}} \|u\|_H
\end{align*}
for every $\theta > 1$.
As a consequence, we get $\Phi_u^\epsilon \in \mathcal{E}_1 \cap D(\mathscr{L}_y^\epsilon)$ satisfying $\mathscr{L}_y^\epsilon \Phi_u^\epsilon = -\Psi_u^\epsilon$, and defining $\varphi_2^\epsilon(u,Y)=\Phi_u^\epsilon(Y)$ it holds $\|\varphi_2^\epsilon(u,\cdot)\|_{\mathcal{E}_1} \lesssim \|h\|_{H^{\theta_1}} \|u\|_H$.
Concerning regularity of $D_u\varphi_2^\epsilon(Y)$ (\autoref{prop:regularity_varphi2}), we notice that $\langle D_u^\epsilon \Psi_u , v \rangle \in \mathcal{E}$ for every $v \in H$, with $\|\langle D_u \Psi_u^\epsilon , v \rangle\|_\mathcal{E} \lesssim \|h\|_{H^{\theta_1}}\|v\|_H$, and with arguments similar to those in the proof of \autoref{prop:regularity_varphi2} the same holds true for $\langle D_u \varphi_2^\epsilon , v \rangle$;

$\bullet$
To prove \autoref{lem:Sobolev_increments}, we follow an approach similar to that of the proof in \autoref{sec:conv}.
The main difference in the proof consists in recovering a suitable bound for $|\varphi_1^k(u^\epsilon_s,y^\epsilon_s) - \varphi_1^k(u^\epsilon_t,y^\epsilon_t)|$, $s,t \in [0,T]$.
For given $s,t \in [0,T]$, $s<t$ we control
\begin{align*}
|\varphi_1^k(u^\epsilon_s,y^\epsilon_s) - \varphi_1^k(u^\epsilon_t,y^\epsilon_t)|
&\lesssim
\epsilon^{-1/4}
\|u^\epsilon_s-u^\epsilon_t\|_{H^{-1}}^{1/4} 
\|u^\epsilon_s-u^\epsilon_t\|_{H}^{3/4} \|y^\epsilon_s \|_{H^{1/2}} \|e_k\|_{H^{\theta_0}}
\\
&\quad+
\|y^\epsilon_s-y^\epsilon_t\|_{H^{-1}}^{1/2}
\|y^\epsilon_s-y^\epsilon_t\|_{H^1}^{1/2} \|u^\epsilon_t\|_H \|e_k\|_{H^{\theta_0}},
\end{align*}
where we have used $\|(-\epsilon A+Id)^{-1} y\|_{H^1} \lesssim \epsilon^{-1/4}\|y\|_{H^{1/2}}$ in the first estimate. 
By the analogous of \autoref{lem:increments} showed above, we have for some $\alpha>0$
\begin{align*}
\epsilon^{p/2} \mathbb{E} \left[|\varphi_1^k(u^\epsilon_s,y^\epsilon_s) - \varphi_1^k(u^\epsilon_t,y^\epsilon_t)|^p\right]
&\lesssim
\|e_k\|_{\theta_0}^p |t-s|^{\alpha p}, 
\end{align*}
and using $H^2$ bounds on $y^\epsilon$ we arrive to
\begin{align*}
\mathbb{E} \left[ \|u^\epsilon_t-u^\epsilon_s\|_{H^{-\sigma}}^p \right] \lesssim
|t-s|^{\alpha p};
\end{align*}

$\bullet$
To identify the equation solved by the limiting process $u$, we argue as in the proof of \autoref{prop:identification}. 

\subsection{2D Surface Quasi-Geostrophic equations}

On the space $H=L^2(\mathbb{T}^2)$ of zero-mean square integrable vorticity fields of the two-dimensional torus $\mathbb{T}^2$ we consider the Surface Quasi-Geostrophic system 
\begin{align} \label{eq:SQG}
\begin{cases}
d\xi^\epsilon_t 
= 
-(-\nu\Delta)^{1/2} \xi^\epsilon_t dt
-
u^\epsilon_t \cdot \nabla \xi^\epsilon_t dt
-
\epsilon^{-1/2}
y^\epsilon_t \cdot \nabla \xi^\epsilon_t dt,
\\
d\eta^\epsilon_t 
= 
\epsilon^{-1}C\eta^\epsilon_t + (-\nu\Delta)^{1/2} \eta^\epsilon_t dt
-
(u^\epsilon_t \cdot \nabla) \eta^\epsilon_t dt
-
\epsilon^{-1/2}
(y^\epsilon_t \cdot \nabla) \eta^\epsilon_t dt
+
\epsilon^{-1/2}
d\mathcal{W}_t,
\\
u^\epsilon_t = -\nabla^\perp(-\Delta)^{-1/2} \xi^\epsilon_t,
\\
y^\epsilon_t = -\nabla^\perp(-\Delta)^{-1/2} \eta^\epsilon_t,
\end{cases}
\end{align}
where $\nu>0$, $\nabla^\perp = (-\partial_y,\partial_x)$, the velocity fields $u^\epsilon,y^\epsilon$ are reconstructed from the vorticity fields $\xi^\epsilon,\eta^\epsilon$ via the Riesz transform $\mathcal{R} = -\nabla^\perp(-\Delta)^{-1/2}$, and $\mathcal{W}=Q^{1/2}W$ is a Wiener process on $H$.

We remark that in this setting elements of $H$ are scalar functions, whereas in the context of Navier-Stokes equations they were vector-valued. Sobolev spaces $H^s$, $s \in \R$ are then defined accordingly. 
Notice that the Riesz transform preserves Sobolev regularity component-wise.

We can set the previous system as an abstract equation in $H$ defining $\Lambda\xi = -(-\nu\Delta)^{1/2}\xi$ and the nonlinear operator $b$ as
\begin{align*}
b(\xi,\eta)
=
\nabla^\perp(-\Delta)^{-1/2} \xi \cdot \nabla \eta
=
-\mathcal{R}\xi \cdot \nabla \eta.
\end{align*} 
Seldom in the literature concerning Surface Quasi-Geostrophic equations the operator $\Lambda$ is defined to be equal to $(-\nu\Delta)^{1/2}$; in particular notice that with our convention on the sign, $\Lambda$ is a negative definite operator, and system \eqref{eq:SQG} reads
\begin{align} \label{eq:SQG_abstract}
\begin{cases}
d\xi^\epsilon_t 
= 
\Lambda \xi^\epsilon_t dt
+
b(\xi^\epsilon_t,\xi^\epsilon_t) dt
+
\epsilon^{-1/2}
b(\eta^\epsilon_t,\xi^\epsilon_t) dt,
\\
d\eta^\epsilon_t 
= 
\epsilon^{-1}C\eta^\epsilon_t dt 
+
\Lambda \eta^\epsilon_t dt
+
b(\xi^\epsilon_t,\eta^\epsilon_t) dt
+
\epsilon^{-1/2}
b(\eta^\epsilon_t,\eta^\epsilon_t) dt
+
\epsilon^{-1/2}
Q^{1/2} dW_t.
\end{cases}
\end{align}

Properties (B1)-(B4) hold true for the nonlinear operator $b$ (and $d=2$), so most of the arguments shown in previous sections for Navier-Stokes equations in velocity form can be adapted to the 2D Surface Quasi-Geostrophic equations; the main difference between the two models lies in the dissipative operator $\Lambda$, which is less regularizing in this latter case.
However, there are two main factors that help us dealing with loss of regularity: first, borrowing ideas from \cite{RoZhZh15}, we can prove additional bounds for the $\eta$ component of solutions to 
\eqref{eq:SQG_abstract} in the space $\mathcal{B}([0,T],L^p(\Omega,L^p(\mathbb{T}^2)))$, assuming the initial condition $\eta_0$ to be in $H \cap L^p(\mathbb{T}^2)$ for $p>2$; second, having restricted the Surface Quasi-Geostrophic equations in a two-dimensional space, one can take advantage of better Sobolev embeddings and better estimates for the product of Sobolev functions to circumvent the difficulties coming with a less regularizing operator $\Lambda$. 
As a consequence, we can allow $\Gamma \geq \gamma >1/4$ in (C2), but to exploit $L^p=L^p(\mathbb{T}^2)$ integrability of $\eta$ we need in addition that $C$ generates a semigroup $e^{Ct}$ satisfying $\|e^{Ct}\eta\|_{L^q} \lesssim \|\eta\|_{L^q}$ uniformly in $t>0$ for some $q>p$, so that by the same arguments of \cite[Lemma 5.5]{RoZhZh15} there exists $c_0$ sufficiently small such that
\begin{align} \label{eq:C-c_0}
\int_{\mathbb{T}^2} |\eta(x)|^p \eta(x) \left( C-c_0 Id\right) \eta(x) dx \geq 0.
\end{align}
Notice that, in particular, the choice $C=\Lambda$ is allowed in the present setting.
Finally, assumptions (Q1)-(Q2) must be modified replacing $A$ with $\Lambda$.

Here we state the analogous of \autoref{thm:main}. In the case of Surface Quasi-Geostrophic equations, the limiting equation below has been studied in \cite{ReMeCh17}. 
\begin{thm}
Under the previous assumptions, let $\{(\xi^\epsilon,\eta^\epsilon)\}_{\epsilon \in (0,1)}$ be a bounded-energy family of martingale solutions to \eqref{eq:SQG}.
Then for every $\beta>0$, the laws of the processes $\{\xi^\epsilon\}_{\epsilon\in(0,1)}$ are tight as probability measures on the space $L^2([0,T],H) \cap C([0,T],H^{-\beta})$, and every weak accumulation point $(\xi,Q^{1/2} W)$ of $(\xi^\epsilon,Q^{1/2} W^\epsilon)$, $\epsilon \to 0$, is an analytically weak solution of the equation with transport noise and It\=o-Stokes drift velocity $r = \int_H (-C)^{-1} b(w,w) d\mu(w)$:
\begin{align*} 
d \xi _t
&= 
\Lambda \xi_t dt 
+ 
b(\xi_t, \xi_t) dt
+
b((-C)^{-1} Q^{1/2} \circ dW_t, \xi_t)
+
b(r, \xi_t) dt
\\
&= 
-\nu^{1/2}(-\Delta)^{1/2} \xi_t dt
+
\nabla^\perp (-\Delta)^{-1/2} \xi_t \cdot \nabla \xi_t dt
\\
&\quad+ 
\nabla^\perp (-\Delta)^{-1/2} (-C)^{-1} Q^{1/2} \circ dW_t \cdot \nabla \xi_t
+
\nabla^\perp (-\Delta)^{-1/2} r \cdot \nabla \xi_t dt.
\end{align*}
If in addition pathwise uniqueness holds for the limit equation then the whole sequence converges in law; moreover, convergence in $\mathbb{P}$-probability holds true if solutions to \eqref{eq:SQG} are probabilistically strong.
\end{thm}

\begin{rmk}
In \cite{RoZhZh15} the authors prove existence of probabilistically strong solutions to \eqref{eq:SQG_abstract} only in the subcritical case $\Lambda=-(-\nu \Delta)^\alpha$, $\alpha>1/2$.
Concerning the limiting equation, \cite[Theorem 3.3]{SQGpathwise} estiblishes local well-posedness of probabilistically strong solutions under suitable regularity assumptions on the initial datum $\xi_0$.
\end{rmk}

Let us therefore illustrate the main differences in the proof:
\\

$\bullet$ The existence of weak solutions $(\xi^\epsilon,\eta^\epsilon)$ to \eqref{eq:SQG_abstract} with bounded $\eta$ component in the space $\mathcal{B}([0,T],L^p(\Omega,L^p(\mathbb{T}^2)))$, $p>2$ can be proved among the lines of \cite[Theorem 3.3]{RoZhZh15}, under the hypothesis $\eta_0 \in H \cap L^p(\mathbb{T}^2)$. We point out that no $L^p$ bound on the large-scale process at time zero is required, making this assumption not very demanding by the modelling point of view (the small-scale process $\eta^\epsilon$ has a typical decorrelation time of order $\epsilon$, so that we care very little about the initial datum $\eta_0$). 
Let us only discuss why the bound in $\mathcal{B}([0,T],L^p(\Omega,L^p(\mathbb{T}^2)))$ for $\eta^\epsilon$ is indeed uniform in $\epsilon$. 

For the sake of a clear presentation, we present only a formal argument which relies on a $L^p$ It\=o Formula for $\eta^\epsilon$ developed in \cite{Kr09}, and a positivity lemma from \cite{RoZhZh15}.
To make the argument rigorous, one should perform the next computations at the level of a particular smooth approximations of $\eta^\epsilon$ and check that the obtained bounds pass to the limit, as done in full detail in the proof of \cite[Theorem 3.3]{RoZhZh15}. 
It is worth noticing that the aforementioned approximation is different from the Galerkin, since it is necessary that \eqref{eq:IBP} below holds true for the approximation process, too. We refer to \cite{RoZhZh15} for details.

Let us therefore move to the actual computations. Applying formally $L^p$ It\=o Formula \cite[Lemma 5.1]{Kr09} to $\eta^\epsilon$, one has for every $t \in [0,T]$:
\begin{align} \label{eq:Ito_p}
\|\eta^{\epsilon}_t\|_{L^p}^p 
&=
\|\eta_0\|_{L^p}^p
+
p \epsilon^{-1}\int_0^t \int_{\mathbb{T}^2}
|\eta^{\epsilon}_s(x)|^{p-2} \eta^{\epsilon}_s(x) C \eta^{\epsilon}_s(x) dx ds
\\
&\quad+ \nonumber
p \int_0^t \int_{\mathbb{T}^2}
|\eta^{\epsilon}_s(x)|^{p-2} \eta^{\epsilon}_s(x) \Lambda \eta^{\epsilon}_s(x) dx ds
\\
&\quad \nonumber
+
p \int_0^t \int_{\mathbb{T}^2}
|\eta^{\epsilon}_s(x)|^{p-2} \eta^{\epsilon}_s(x) b(\xi^{\epsilon}_s,\eta^{\epsilon}_s)(x) dx ds
\\
&\quad \nonumber
+ 
p \epsilon^{-1/2} \int_0^t \int_{\mathbb{T}^2}
|\eta^{\epsilon}_s(x)|^{p-2} \eta^{\epsilon}_s(x) b(\eta^{\epsilon}_s,\eta^{\epsilon}_s)(x) dx ds
\\
&\quad \nonumber
+
p \epsilon^{-1/2} \int_0^t \int_{\mathbb{T}^2}
|\eta^{\epsilon}_s(x)|^{p-2} \eta^{\epsilon}_s(x) Q^{1/2} dW_s(x) dx
\\
&\quad \nonumber
+
\frac{p(p-1)}{2} \epsilon^{-1} \int_0^t \int_{\mathbb{T}^2}
|\eta^{\epsilon}_s(x)|^{p-2} Tr(Q) dx ds.
\end{align}
Moreover, by integration by parts, it is easy to check that for every $\eta,\xi$ sufficiently regular and $p>2$: 
\begin{align} \label{eq:IBP}
\int_{\mathbb{T}^2} |\eta|^{p-2} \eta b(\xi,\eta) dx 
=
-(p-1)
\int_{\mathbb{T}^2} |\eta|^{p-2} \eta b(\xi,\eta) dx.
\end{align}
Notice that the two quantities above differ only in the factor $(p-1) \neq 1$, therefore 
\begin{align} \label{eq:IBP2}
\int_{\mathbb{T}^2} |\eta|^{p-2} \eta b(\xi,\eta) dx 
= 
0.
\end{align}

Hence, taking into account \eqref{eq:IBP2}, \eqref{eq:C-c_0} and \cite[Lemma 5.5]{RoZhZh15}, arguing as in the proof of \autoref{lem:int_y} we deduce:
\begin{align} \label{eq:E_eta}
\sup_{\substack{\epsilon\in(0,1)}} \,\sup_{t \in [0,T]}
\mathbb{E} \left[ \|\eta^{\epsilon}_t \|_{L^p}^p \right] \lesssim 1;
\end{align}

$\bullet$
\autoref{lem:E_y} becomes: for every $p \geq 2$ it holds
\begin{align*}
\sup_{\substack{\epsilon\in(0,1),\\ n \in \N}}\,
\sup_{t \in [0,T]}
\left(
\mathbb{E}\left[\|\eta^{\epsilon}_t\|_H^p \right]
+
\int_0^t
\mathbb{E}\left[\|\eta^{\epsilon}_s\|_H^{p-2}\|\eta^{\epsilon}_s\|_{H^{1/2}}^2\right] ds \right)
\lesssim 1,
\end{align*}
with identical proof. \autoref{prop:y-Y} remains the same.

$\bullet$
The results contained in \autoref{sec:poisson} need to be suitably modified taking into account the fact that the semigroup generated by $C_\epsilon = C + \epsilon \Lambda$ is less regularizing than the semigroup generated by $C + \epsilon A$, as for the Navier-Stokes system. The precise meaning of the previous phrase is that additional spatial regularity comes at higher price (in terms of factors $\epsilon^{-1}$). In particular, we need to interpolate between Sobolev spaces $H^\gamma$ and $H^{1/2}$ (instead of $H^\gamma$ and $H^1$).

$\bullet$
To find the corrector $\varphi_1^\epsilon$, one simply applies \autoref{prop:inverse_theta} as in the main body of this paper. The resulting corrector is 
\begin{align*}
\varphi_1^\epsilon(\xi,\eta) = \langle b((-C-\epsilon \Lambda)^{-1}\eta,\xi), h \rangle;
\end{align*} 

$\bullet$
In order to fully exploit $L^p$ integrability of the small-scale process $\eta^\epsilon$, we introduce the (homogeneous) Bessel potential spaces on the torus, defined as the spaces
\begin{align*}
H^{s,p} \coloneqq \left\{ \xi \in \mathscr{S}' : (-\Lambda)^s \xi \in L^p \right\},
\quad 
s \in \R,\, p \in (1,\infty),
\end{align*} 
which are Banach when equipped with the norm 
\begin{align*}
\|\xi\|_{H^{s,p}}
\coloneqq
\|(-\Lambda)^s \xi\|_{L^p}.
\end{align*}

Bessel potential spaces are always reflexive, and $(H^{s,p})^* = H^{-s,q}$, where $1/p+1/q=1$. Of course, $H^{0,p} = L^p$ for every $p \in (1,\infty)$, and $H^{s,2} = H^s$ for every $s \in \R$, with equivalence of norms. 

In \autoref{prop:regularity_varphi1}, the same argument yields in fact $D_\eta \varphi_1^\epsilon(\xi),D_\xi \varphi_1^\epsilon(\eta) \in H^{2\theta+s,p}$ for every $\xi,\eta \in H^{s,p}$, $s\in \R$, $p \in (1,\infty)$, $\theta \in [\gamma,1/2]$, with (here $\theta_1 = \theta_1(s,p)$ and $\gamma \leq 1/2$ without loss of generality)
\begin{align*}
\|D_\eta \varphi_1^\epsilon(\xi)\|_{H^{2\theta+s,p}} 
\lesssim 
\epsilon^{-\frac{\theta-\gamma}{1/2-\gamma}}
\|h\|_{H^{\theta_1}} \|\xi\|_{H^{s,p}},
\quad
\|D_\xi \varphi_1^\epsilon(\eta)\|_{H^{2\theta+s,p}} \lesssim 
\epsilon^{-\frac{\theta-\gamma}{1/2-\gamma}}
\|h\|_{H^{\theta_1}} \|\eta\|_{H^{s,p}};
\end{align*}

$\bullet$
The linearisation trick relies on \autoref{prop:linearisation}. 
As usual, denote $\zeta = y - Y$. 
In the present setting, we can prove the following bounds for any $p$ sufficiently large and $\delta$ sufficiently small (we use $b:L^p \times H^{1/2-\delta} \to H^{-1}$ continuous by Sobolev embedding, and assume $\gamma \leq 1/3$ without loss of generality):
\begin{align*}
|\langle b(\zeta,\xi),D_\xi \varphi_1(Y)\rangle|
&\lesssim
\|\zeta\|_{H} \|h\|_{H^{\theta_1}} \|Y\|_{H^{\theta_0-\gamma}} \|\xi\|_{H},
\\
|\langle b(\eta,\xi),D_\xi \varphi_1(\zeta)\rangle|
&\lesssim
\epsilon^{-\frac{1-3\gamma}{1-2\gamma}}
\|\zeta\|_{H^\gamma} \|h\|_{H^{\theta_1}} \|\xi\|_{H^{1/2-\delta}} \|\eta\|_{L^p},
\\
|\langle b(\zeta,Y),D_\eta \varphi_1(\xi)\rangle|
&\lesssim
\|\zeta\|_{H} \|h\|_{H^{\theta_1}} \|Y\|_{H^{\theta_0-\gamma}} \|\xi\|_{H},
\\
|\langle b(\eta,\zeta),D_\eta \varphi_1(\xi)\rangle|
&\lesssim
\epsilon^{-\frac{1-3\gamma}{1-2\gamma}}
\|\zeta\|_{H^\gamma} \|h\|_{H^{\theta_1}} \|\xi\|_{H^{1/2-\delta}} \|\eta\|_{L^p};
\end{align*}

$\bullet$
To find $\varphi_2^\epsilon$, we apply \autoref{prop:inverse_theta} to $\Psi_\xi^\epsilon = \psi_\xi^\epsilon - \int_H \psi_\xi^\epsilon (w) d\mu(w)$, where for every $\theta > 1-\gamma$:
\begin{align*}
\psi_\xi^\epsilon 
=
\langle b(\cdot,\xi), D_\xi \varphi_1^\epsilon \rangle 
+ 
\langle b(\cdot,\cdot), D_\eta \varphi_1^\epsilon \rangle
\in \mathcal{E}_{\theta}, 
\quad
\|\psi_\xi^\epsilon\|_{\mathcal{E}_\theta} \lesssim \|h\|_{H^{\theta_1}} \|\xi\|_{H}.
\end{align*}
As a result, we get $\Phi_\xi^\epsilon \in \mathcal{E}_{\theta-\gamma} \cap D(\mathscr{L}_\eta^\epsilon)$  satisfying $\mathscr{L}_\eta^\epsilon \Phi_\xi^\epsilon = - \Psi_\xi^\epsilon$, and defining $\varphi_2^\epsilon(\xi,Y) = \Phi_\xi^\epsilon(Y)$ we have $\|\varphi_2^\epsilon(\xi,\cdot) \|_{\mathcal{E}_{\theta-\gamma}} \lesssim \|h\|_{\theta_1} \|\xi\|_H$;

$\bullet$
For every $0<\theta <\theta_0+2\gamma-1$, and $v \in H^{-\theta}$ it holds $\langle v, D_\xi \Psi_\xi^\epsilon(\cdot) \rangle \in \mathcal{E}_{\theta_0}$, with
\begin{align*}
\|\langle v, D_\xi \Psi_\xi^\epsilon(\cdot) \rangle\|_{\mathcal{E}_{\theta_0}}
\lesssim
\|h\|_{H^{\theta_1}} \|v\|_{H^{-\theta}}.
\end{align*}
Therefore, \autoref{prop:regularity_varphi2} becomes
\begin{align*}
\|\langle v, D_\xi \varphi_2^\epsilon(\xi,\cdot) \rangle\|_{\mathcal{E}_{\theta_0}}
\lesssim
\|h\|_{H^{\theta_1}} \|v\|_{H^{-\theta}};
\end{align*}

$\bullet$
It\=o Formula (\autoref{lem:ito}) can be proved in a similar fashion, making sure of exploiting $L^p$ integrability of $\eta^\epsilon$ and \autoref{prop:regularity_varphi1}; 

$\bullet$
In order to prove tightness (\autoref{lem:Sobolev_increments}) we estimate 
\begin{align*}
|\varphi_1^{k,\epsilon}(\xi^\epsilon_s,\eta^\epsilon_s)-\varphi_1^{k,\epsilon}(\xi^\epsilon_t,\eta^\epsilon_t)|
&\lesssim
\|\xi^\epsilon_s-\xi^\epsilon_t\|_{H^{-2\gamma}} \|\eta^\epsilon_s\|_H \|e_k\|_{H^{\theta_1}}
+
\|\eta^\epsilon_s-\eta^\epsilon_t\|_{H^{-2\gamma}}\|\xi^\epsilon_t\|_H\|e_k\|_{H^{\theta_0}},
\end{align*}
and by interpolation it is sufficient to prove an easy analogous of \autoref{lem:increments} for the norm $H^{-\theta_0}$.
To control the term involving $\Phi^{k,\epsilon}$, we use \autoref{prop:regularity_varphi1} with $\theta=\gamma \geq 1/4$ and the continuity of $b:L^p \times H^{1/2-\delta} \to H^{1/2+2\gamma}$ for $p$ sufficiently large and $\delta$ sufficiently small.
As for the bound on $\Phi^{k,\epsilon}_1$, and in particular on the term $\langle b(\xi^\epsilon_s,\xi^\epsilon_s), D_\xi \varphi_1^\epsilon(\eta^\epsilon_s) \rangle$, we use Sobolev embedding $H^{1/4} \subset L^{8/3}$ and interpolation between $H^{\gamma+1/2,p}$ and $H^{\gamma+1-\delta,2}$ to control $\|D_\xi \varphi_1^\epsilon(\eta^\epsilon_s) \|_{H^{1,4}}$. Indeed, $1=(1-\theta)(\gamma+1-\delta) + \theta(\gamma+1/2)$ for $\theta=\frac{\gamma-\delta}{1/2-\delta}$, and thus $1/4 = (1-\theta)/2 + \theta/p$ holds for $p$ sufficiently large since $1-\theta<1/2$ for $\gamma>1/4$ and $\delta$ sufficiently small.
The identification of the It\=o-Stokes drift and the Stratonovich corrector goes as for the Navier-Stokes equations.

\subsection{Primitive equations}
Let $\mathbb{T}^d = (\R/\mathbb{Z})^d$ denote the $d$-dimensional torus, identified with points $(\mathbf{x},z) \in [0,1]^{d-1} \times [0,1]$ with periodic boundary conditions, and consider the system of Primitive equations
\begin{align} \label{eq:primitive}
\begin{cases}
du_t^\epsilon
= 
\nu \Delta u_t^\epsilon dt
-
(u_t^\epsilon \cdot \nabla_\mathbf{x}) u_t^\epsilon dt
-
v_t^\epsilon \partial_z u_t^\epsilon dt 
\\
\qquad-
\epsilon^{-1/2}
(y_t^\epsilon \cdot \nabla_\mathbf{x}) u_t^\epsilon dt
-
\epsilon^{-1/2}
w_t^\epsilon\partial_z u_t^\epsilon dt
+
\nabla_\mathbf{x} p_t^\epsilon dt,
\\
dy_t^\epsilon 
= 
\epsilon^{-1}Cy_t^\epsilon dt 
+
\nu \Delta y_t^\epsilon dt
-
(u_t^\epsilon \cdot \nabla_\mathbf{x}) y_t^\epsilon dt
-
v_t^\epsilon \partial_z y_t^\epsilon dt
\\
\qquad-
\epsilon^{-1/2}
(y_t^\epsilon \cdot \nabla_\mathbf{x}) y_t^\epsilon dt
-
\epsilon^{-1/2}
w_t^\epsilon \partial_z y_t^\epsilon dt
+
\epsilon^{-1/2}
d\mathcal{W}_t
+
\nabla_\mathbf{x} q_t^\epsilon dt,
\\
\partial_z p_t^\epsilon = 0, \,
\partial_z q_t^\epsilon = 0, 
\\
\mbox{div}_\mathbf{x} u_t^\epsilon + \partial_z v_t^\epsilon = 0, \,
\mbox{div}_\mathbf{x} y_t^\epsilon + \partial_z w_t^\epsilon = 0,
\end{cases}
\end{align}
where $\nu>0$ and $p^\epsilon,q^\epsilon$ are pressure fields. Conditions $\partial_z p_t^\epsilon = 0$, $\partial_z q_t^\epsilon = 0$ correspond to the so called hydrostatic approximation, whereas $\mbox{div}_\mathbf{x} u_t^\epsilon + \partial_z v_t^\epsilon = 0$ and $\mbox{div}_\mathbf{x} y_t^\epsilon + \partial_z w_t^\epsilon = 0$ prescribe incompressibility of the composite velocity fields $(u^\epsilon,v^\epsilon)$ and $(y^\epsilon,w^\epsilon)$.

Define $H$ as the space of horizontal velocity fields on the torus satisfying
\begin{align*}
H = \left\{ u \in [L^2(\mathbb{T}^d)]^{d-1} : \, \int_{\mathbb{T}^d} u(\mathbf{x},z) d\mathbf{x} dz = 0, \, \int_0^1 \mbox{div}_\mathbf{x} u(\mathbf{x},z) dz = 0  \right\}.
\end{align*}
For any $u \in H$, incompressibility condition $\mbox{div}_\mathbf{x} u_t + \partial_z v_t = 0$ uniquely determines the vertical velocity $v$ via the relation
\begin{align} \label{eq:vertical_vel}
v(\mathbf{x},z) 
= 
- \int_0^z \mbox{div}_\mathbf{x} u_t(\mathbf{x},z')dz',
\end{align}
so that \eqref{eq:primitive} reduces to a system for the unknown horizontal velocities $u^\epsilon,y^\epsilon$ only.

Projecting \eqref{eq:primitive} via the $L^2$ projector $\Pi$, $(\Pi u) (\mathbf{x},z) = u(\mathbf{x},z) - \int_0^1 u(\mathbf{x},z') dz'$, we can eliminate pressures and obtain the equivalent system on $H$:
\begin{align*}
\begin{cases}
du_t^\epsilon
= 
\nu \Delta u_t^\epsilon dt
-
\Pi(u_t^\epsilon \cdot \nabla_\mathbf{x}) u_t^\epsilon dt
-
\Pi v_t^\epsilon \partial_z u_t^\epsilon dt 
\\
\qquad-
\epsilon^{-1/2}
\Pi (y_t^\epsilon \cdot \nabla_\mathbf{x}) u_t^\epsilon dt
-
\epsilon^{-1/2}
\Pi w_t\partial_z u_t^\epsilon dt,
\\
dy_t^\epsilon
= 
\epsilon^{-1}Cy^\epsilon_t dt
+
\nu \Delta y_t^\epsilon dt
-
\Pi (u_t^\epsilon \cdot \nabla_\mathbf{x}) y_t^\epsilon dt
-
\Pi v_t^\epsilon \partial_z y_t^\epsilon dt
\\
\qquad-
\epsilon^{-1/2}
\Pi (y_t^\epsilon \cdot \nabla_\mathbf{x}) y_t^\epsilon dt
-
\epsilon^{-1/2}
\Pi w_t^\epsilon \partial_z y_t^\epsilon dt
+
\epsilon^{-1/2}
\Pi d\mathcal{W}_t
\\
\mbox{div}_\mathbf{x} u_t^\epsilon + \partial_z v_t^\epsilon = 0, \,
\mbox{div}_\mathbf{x} y_t^\epsilon + \partial_z w_t^\epsilon = 0,
\end{cases}
\end{align*}
and thus we can recast Primitive equations in an abstract setting with
\begin{align*}
Au = \nu \Delta u,
\quad
b(u,u')=-\Pi (u \cdot \nabla_\mathbf{x}) u' -\Pi v \partial_z u',
\quad
Q^{1/2}W = \Pi \mathcal{W},
\end{align*}
with $v$ recovered from $u$ by \eqref{eq:vertical_vel}.
It is immediate to verify the antisymmetric property (B4) of the nonlinear term $b$ above. However, $b$ is less regular than the nonlinear term of Navier-Stokes equations since the expression of $b(u,u')$ involves the horizontal gradient of $u$, which is not compensated by the integral in the vertical direction: indeed, it holds
\begin{itemize}
\item[({\bf B1})]
$b:H^s \times H^{\theta_0} \to H^{s-1}$ is bilinear and continuous for every $s \in \R$, $s<d/2$, $\theta_0>1+d/2$.
\item[({\bf B2})]
$b:H^s \times H^{\theta_1} \to H^{s-1}$ is bilinear and continuous for every $s \in \R$, $s \geq d/2$ and $\theta_1>1+s$.
\item[({\bf B3})]
$b:H^s \times H^r \to H^{s+r-2-d/2}$ is bilinear and continuous if $s,r \in (1-d/2,1+d/2)$, $s+r>1$.
\end{itemize}  

Due to the very mild regularity of $b$, to prove our convergence result one could either: \textit{i}) restrict to dimension $d=2$, and include bounds on $\partial_z u^\epsilon, \partial_z v^\epsilon$ in the space $L^p(\Omega, L^\infty([0,T],H) \cap L^2([0,T],H^1))$ in the notion of solution to \eqref{eq:primitive}, relying on a priori estimates as those carried on in \cite{GHZi08}; or \textit{ii)} require a sufficiently strong dissipation $C$ at small scales, satisfying (C1) and (C2) for some $\Gamma \geq \gamma > 9/8$. 
The introduction of the operator $C$ at small scales allows us to bound solutions $y^\epsilon$ uniformly in the space $L^p(\Omega,L^2([0,T],H^{\gamma}))$, with minor modifications in the proof of \autoref{prop:existence}.
We will also retain (Q1)-(Q2) unchanged. 

Under the previous assumptions, \autoref{thm:main} applied to Primitive equations yields:
\begin{thm}
Let $\{(u^\epsilon,y^\epsilon)\}_{\epsilon \in (0,1)}$ be a bounded-energy family of martingale solutions to \eqref{eq:primitive}.
Then for every $\beta>0$, the laws of the processes $\{u^\epsilon\}_{\epsilon\in(0,1)}$ are tight as probability measures on the space $L^2([0,T],H) \cap C([0,T],H^{-\beta})$, and every weak accumulation point $(u,Q^{1/2} W)$ of $(u^\epsilon,Q^{1/2} W^\epsilon)$, $\epsilon \to 0$, is an analytically weak solution of the equation with transport noise and It\=o-Stokes drift velocity $r = \int_H (-C)^{-1} b(w,w) d\mu(w)$:
\begin{align*} 
d u _t
&= 
A u_t dt 
+ 
b(u_t, u_t) dt
+
b((-C)^{-1} Q^{1/2} \circ dW_t, u_t)
+
b(r, u_t) dt
\\
&= 
\nu \Delta u_t dt
- \Pi (u_t \cdot \nabla_\mathbf{x}) u_t dt
- \Pi v_t \partial_z u_t dt
\\
&\quad 
- \Pi ((-C)^{-1} Q^{1/2} \circ dW_t \cdot \nabla_\mathbf{x}) u_t
- \Pi dw_t \partial_z u_t dt
- \Pi (r \cdot \nabla_\mathbf{x}) u_t dt
- \Pi q \partial_z u_t dt,
\end{align*}
where $v,w,q$ are defined implicitly by the incompressibility conditions
\begin{align*}
\textup{div}_\mathbf{x} u_t + \partial_z v_t = 0,
\quad
\textup{div}_\mathbf{x} (-C)^{-1}Q^{1/2}W_t + \partial_z w_t = 0,
\quad
\textup{div}_\mathbf{x} r + \partial_z q = 0.
\end{align*}
If in addition pathwise uniqueness holds for the limit equation then the whole sequence converges in law; moreover, convergence in $\mathbb{P}$-probability holds true if solutions to \eqref{eq:primitive} are probabilistically strong.
\end{thm}

We refer to \cite[Theorem 2.6]{BrSl20+} for conditions ensuring pathwise uniqueness of the limit equations and the approximating system \eqref{eq:primitive}.
{See also \cite{AHHS22,AHHS22+} for additional reference.}
Let us finally take a look at the main modification needed to prove the analogous of \autoref{thm:main} for the solution of Primitive equations.
More specifically, we can prove:
\\

$\bullet$
The same statement of \autoref{lem:E_y}, although the bound \eqref{eq:energy_zeta_p} need to be modified in
\begin{align*}
\|\zeta^{\epsilon,n}_t\|_H^{p}
&+
\epsilon^{-1}pM \int_0^t\|\zeta^{\epsilon,n}_s\|_H^{p-2}\|\zeta^{\epsilon,n}_s\|_{H^\gamma}^2 ds
\\
&\leq \nonumber
\|y_0\|_H^p
+
p \int_0^t
\|\zeta^{\epsilon,n}_s\|_H^{p-2}
\langle b(u^{\epsilon,n}_s , Y^{\epsilon,n}_s),\zeta^{\epsilon,n}_s \rangle ds
\\
&\quad+ \nonumber
\epsilon^{-1/2} p
\int_0^t
\|\zeta^{\epsilon,n}_s\|_H^{p-2} 
\langle b(y^{\epsilon,n}_s,Y^{\epsilon,n}_s),\zeta^{\epsilon,n}_s \rangle ds
\\
&\leq\|y_0\|^p
+ \nonumber
M_1\int_0^t
\|\zeta^{\epsilon,n}_s\|_H^{p-2} 
\|\zeta^{\epsilon,n}_s\|_{H^\gamma}
\|u^{\epsilon,n}_s\|_{H^{1-\gamma}} \|Y^{\epsilon,n}_s\|_{H^{\theta_0}}ds
\\
&\quad+ \nonumber
\epsilon^{-1/2}
M_1\int_0^t
\|\zeta^{\epsilon,n}_s\|_H^{p-2}
\|\zeta^{\epsilon,n}_s\|_{H^\gamma}
 \|y^{\epsilon,n}_s\|_{H^{1-\gamma}} \|Y^{\epsilon,n}_s\|_{H^{\theta_0}} ds,
\end{align*}
but still can be controlled with the same techniques;

$\bullet$
The following variant of \autoref{prop:regularity_varphi1}: for every $u,y \in H^s$, $s\in \R$, we have $D_y \varphi_1^\epsilon(u),D_u \varphi_1^\epsilon(y) \in H^{2\gamma+s-1}$ for every, with
\begin{align*}
\|D_y \varphi_1^\epsilon(u)\|_{H^{2\gamma+s-1}} 
\lesssim 
\|h\|_{H^{\theta_1}} \|u\|_{H^s},
\quad
\|D_u \varphi_1^\epsilon(y)\|_{H^{2\gamma+s-1}} \lesssim \|h\|_{H^{\theta_1}} \|y\|_{H^s};
\end{align*}

$\bullet$
For every $\delta >0$ sufficiently small, $u\in H^1$, $y \in H^\gamma$ and $Y \in H^{\theta_0-\gamma}$ it holds:
\begin{align*}
|\langle b(\zeta,u), D_u \varphi_1^\epsilon(Y) \rangle|
&\lesssim
\|\zeta\|_H \|h\|_{H^{\theta_1}} \|Y\|_{H^{\theta_0-\gamma}} \|u\|_{H^1},
\\
|\langle b(y,u) , D_u \varphi_1^\epsilon(\zeta) \rangle|
&\lesssim 
\|\zeta\|_{H^\gamma}\|h\|_{H^{\theta_1}} \|y\|_{H^{\gamma-\delta}}\|u\|_{H},
\\
|\langle b(\zeta,Y) , D_y \varphi_1^\epsilon(u) \rangle|
&\lesssim
\|\zeta\|_H \|h\|_{H^{\theta_1}}\|Y\|_{H^{\theta_0-\gamma}} \|u\|_{H^1},
\\
|\langle b(y,\zeta) , D_y \varphi_1^\epsilon(u) \rangle|
&\lesssim
\|\zeta\|_{H^\gamma} \|h\|_{H^{\theta_1}} \|y\|_{H^{\gamma-\delta}} \|u\|_H ,
\end{align*}
where we have used $b:H^{\gamma-\delta} \times H^{3\gamma-1} \to H$ and $b:H^{\gamma-\delta} \times H^{2\gamma-1} \to H^{-\gamma}$ continuous by (B3), for $\gamma>9/8$ and our choice of $\delta$.

$\bullet$
The map $\psi_u^\epsilon = \langle b(\cdot,u), D_u \varphi_1^\epsilon \rangle + \langle b(\cdot,\cdot), D_y \varphi_1^\epsilon \rangle \in \mathcal{E}_{7/4-\gamma}$, with
$\|\psi_u^\epsilon\|_{\mathcal{E}_{7/4-\gamma}} \lesssim \|h\|_{H^{\theta_1}} \|u\|_{H}$.
For every $v \in H^{1-\theta}$, $\theta<\theta_0-1$, it holds $\langle v , D_u \Psi_u^\epsilon(\cdot) \rangle \in \mathcal{E}_{\theta_0}$ with $\|\langle v , D_u \Psi_u^\epsilon(\cdot) \rangle\|_{\mathcal{E}_{\theta_0}}\lesssim \|h\|_{H^{\theta_1}} \|v\|_{H^{1-\theta}}$;
As usual, we can construct $\varphi_2^\epsilon$ applying \autoref{prop:inverse_theta};

$\bullet$
The proof of It\=o Formula \autoref{lem:ito} needs to exploit the $H^{2\gamma-1}$ estimate of derivatives of $\Phi$: $\|D_u \Phi(u,y) \|_{H^{2\gamma-1}} ,\|D_y \Phi(u,y) \|_{H^{2\gamma-1}} \lesssim 1+\|u\|_H +\|y\|_H$, coming from \autoref{prop:regularity_varphi1};

$\bullet$
In \autoref{lem:Sobolev_increments}, we rely on suitable bounds on suitable negative Sobolev norms of the time increments $u^\epsilon_s-u^\epsilon_t$, $y^\epsilon_s-y^\epsilon_t$, analogous to those of \autoref{lem:increments}. Here we need to use $b:H^1 \times H^{\theta_0} \to H$ continuous.
To control the terms involving $\Phi^{\epsilon,k}$ and $\Phi^{\epsilon,k}_1$ we use the usual arguments and the continuity of $b:H^{\gamma} \times H^{3\gamma-1-\delta} \to H$ for $\gamma>9/8$ and $\delta$ sufficiently small;

$\bullet$
Finally, the proof of \autoref{prop:identification} is similar, paying attention to estimate
\begin{align*}
|\langle b(u^{\epsilon_n}_s,u^{\epsilon_n}_s) - b(u,u) , h \rangle|
&\lesssim
(\|u^{\epsilon_n}_s\|_{H^1}+\|u_s\|_{H^1})
\| u^{\epsilon_n}_s - u_s \|_{H}
\|h\|_{H^{\theta_0}},
\end{align*}
and introduce bounds on the $L^2([0,T],H^1)$ norm in the definition of the path space $\mathcal{X}$.
The identification of the It\=o-Stokes drift and the Stratonovich corrector goes as for the Navier-Stokes equations. 

\appendix

\section{Proof of \autoref{prop:existence}} \label{sec:existence_sol}
Existence of weak martingale solutions to \eqref{eq:system_2} is by now classical, at least when $\epsilon\in(0,1)$ is fixed, see \cite{FlGa95}.
For the sake of completeness here we provide a brief proof of our existence result for bounded-energy families of weak martingale solutions \autoref{prop:existence}.

We recall the following result, which is an immediate corollary of Ascoli-Arzelà Theorem.
\begin{lem} \label{lem:AA}
Let $E$ be a separable Banach space and let $F \subset E$ be a dense subset. Let $\{f^n\}_{n \in \N}$ be a sequence of measurable functions such that $f^n : [0, T]  \to E^*$. 
Assume that for every $t \in [0,T]$ the sequence $\{f^n_t\}_{n \in \N}$ is equibounded in $E^*$, and for any fixed $h \in F$ the sequence of real-valued functions $\{t \mapsto \langle f^n_t, h \rangle\}_{n \in \N}$ is equicontinuous.
Then, $f^n \in C([0, T]; (E^*)_w)$ for every $n \in \N$, and there exists $f \in C([0, T]; (E^*)_w)$ such that, up to a subsequence,
\begin{align*}
f^n \to f \mbox{ strongly in } C([0, T]; (E^*)_w).
\end{align*}
\end{lem}

We are ready to prove our existence result.
\begin{proof}[Proof of \autoref{prop:existence}] 
Fix a stochastic basis $(\Omega, \mathcal{F}, \{\mathcal{F}_t \}_{t \geq 0}, \mathbb{P}, W)$.
Since the Galerkin system \eqref{eq:system_Gal} is finite-dimensional for every $\epsilon\in(0,1)$ and $n \in \N$, it is classical to show that for every $\epsilon\in(0,1)$ and $n \in \N$ a strong solution to \eqref{eq:system_Gal} exists on $(\Omega, \mathcal{F}, \{\mathcal{F}_t \}_{t \geq 0}, \mathbb{P}, W)$.
 
Hereafter, we fix $\epsilon\in(0,1)$ and we focus on the sequences $\{u^{\epsilon,n}\}_{n \in \N}$ and $\{y^{\epsilon,n}\}_{n \in \N}$.

{
In \cite{FlGa95} it is shown (minor modifications of Theorem 3.1) that for every $\epsilon$ the families of the laws of the processes $\{u^{\epsilon,n}\}_{n \in \N}$ and $\{y^{\epsilon,n}\}_{n \in \N}$ are tight in $L^2([0,T],H) \cap C([0,T],H^{-\theta_0})$, and (up to a change in the underlying probability space) they converge almost surely to progressively measurable processes $u^\epsilon,y^\epsilon \in C([0,T],H_w) \cap L^2([0,T],H^1)$ that satisfy (S1)-(S2).

We need to show that (S3)-(S4) hold true.}
Applying It\=o Formula to the function $\|u^{\epsilon,n}_t\|_H^2$, $t \in [0,T]$, we get for every $t \in [0,T]$ and $\epsilon \in (0,1)$:
\begin{align*} 
\|u^{\epsilon,n}_t\|_H^2 + 2 \int_0^t \|u^{\epsilon,n}_s\|_{H^1}^2 ds 
&= 
\|\Pi_n u_0\|_H^2.
\end{align*}
This implies (S3) since $u^{\epsilon,n} \to u^\epsilon$ in $L^2([0,T],H) \cap C([0,T],H^{-\theta_0})$.
Moreover, recalling \eqref{eq:zeta_pwr_p} and the fact that $\mathbb{E}[\|Y^{\epsilon,n}_t\|_{H^1}^p] \lesssim 1$ uniformly in $\epsilon,n$ and $t \in [0,T]$, we also have for every $p$:
\begin{align*}
\sup_{\substack{\epsilon\in(0,1),\\ n \in \N}}\,
\mathbb{E} \left[ \left(\int_0^T
\|y^{\epsilon,n}_s\|_{H^1}^2 ds\right)^{p/2} \right]
\lesssim 1,
\end{align*}
and the uniform bound on $y^\epsilon$ in $L^p(\Omega,L^2([0,T],H^1))$ descends readily.
{ We are left to check the uniform bound on $y^\epsilon$ in $\mathcal{B}([0,T],L^p(\Omega,H))$. 
By \autoref{lem:E_y}, 
\begin{align} \label{eq:bound_unif_yen}
\sup_{\substack{\epsilon\in(0,1),\\ n \in \N}}\,
\sup_{t \in [0,T]}
\mathbb{E} \left[  \|y^{\epsilon,n}_t\|_{H}^p \right]
\lesssim
1,
\end{align}
thus for every $p<\infty$ the functions $\{y^{\epsilon,n}\}_{n \in \N}$ are measurable maps from $[0,T]$ with values in $L^p(\Omega,H) = (L^q(\Omega,H))^*$, $1/p+1/q=1$, that are equibounded in $L^p(\Omega,H)$ for every fixed $t \in [0,T]$ (actually uniformly in $t \in [0,T]$). 
Moreover, for every fixed $h \in L^\infty(\Omega, \mathscr{S})$ and $s,t \in [0,T]$, $s<t$ we have
\begin{align*}
\left| \mathbb{E}\left[
\langle y^{\epsilon,n}_t-y^{\epsilon,n}_s , h\rangle \right] \right|
&\leq
\epsilon^{-1} \int_s^t
\mathbb{E}\left[ |\langle y^{\epsilon,n}_r , C_\epsilon h \rangle|\right] dr
+
\int_s^t
\mathbb{E}\left[ |\langle b(u^{\epsilon,n}_r,y^{\epsilon,n}_r) , h \rangle|\right] dr
\\
&\quad
\epsilon^{-1/2} \int_s^t
\mathbb{E}\left[ |\langle b(y^{\epsilon,n}_r,y^{\epsilon,n}_r) , h \rangle|\right] dr
\\
&\quad+
\epsilon^{-1/2} \mathbb{E}\left[ \langle \Pi_nQ^{1/2}(W_t-W_s),h \rangle \right]
\\
&\lesssim
|t-s| \left(1+\epsilon^{-1} \right)\|h\|_{L^\infty(\Omega, H^{\theta_0})} + |t-s|^{1/2} \epsilon^{-1/2} \|h\|_{L^\infty(\Omega, H)},
\end{align*}
meaning that the sequence of real-valued functions $\{ t \mapsto \langle y^{\epsilon,n}_t,h \rangle \}_{n \in \N}$ are equicontinuous for every fixed $h \in L^\infty(\Omega, \mathscr{S})$.
Since $L^\infty(\Omega, \mathscr{S})$ is dense in $L^q(\Omega,H)$, by  previous \autoref{lem:AA} we have, up to a subsequence that we still denote $n$: 
\begin{align*} 
y^{\epsilon,n} \to y^\epsilon,
\quad
\mbox{ strongly in } C([0, T], (L^p(\Omega,H))_w).
\end{align*} 
By Banach-Steinhaus theorem, $C([0, T], (L^p(\Omega,H))_w) \subset \mathcal{B} ([0, T], L^p(\Omega,H))$ and therefore $y^\epsilon$ inherits the bound \eqref{eq:bound_unif_yen}, which is uniform in $\epsilon$. 
The proof is complete.}
\end{proof}

\bibliographystyle{abbrv}

\end{document}